\documentclass{amsart}

\usepackage{amssymb,mathrsfs,bm}
\usepackage{enumerate}
\usepackage{graphicx}
\usepackage{hyperref}

\theoremstyle{plain}
 \newtheorem{thm}{Theorem}[section]
 \newtheorem*{thm*}{Theorem}
 \newtheorem{prop}{Proposition}[section]
 \newtheorem{lem}{Lemma}[section]
 \newtheorem{ques}{Question}
\theoremstyle{definition}
 \newtheorem{defn}{Definition}[section]
 \newtheorem{exmp}{Example}[section]
 \newtheorem{fct}{Fact}[section]
\theoremstyle{remark}
 \newtheorem{rem}{Remark}[section]
\numberwithin{equation}{section}

\newtheoremstyle{citing}
  {3pt}
  {3pt}
  {\itshape}
  {}
  {\bfseries}
  {.}
  {.5em}
  {\thmnote{#3}}
\theoremstyle{citing}
 \newtheorem*{cthm}{}

\DeclareMathOperator\card{card}
\DeclareMathOperator\dist{dist}
\DeclareMathOperator\dst{d}
\DeclareMathOperator\hdist{\mathrm{d}_H}
\DeclareMathOperator\hdim{\dim_H}
\DeclareMathOperator\sgp{sgp}

\let\xb\mathbf
\let\xc\mathcal
\let\xs\mathscr

\let\eps\varepsilon

\newcommand\Q{\mathbb Q}
\newcommand\N{\mathbb N}
\newcommand\R{\mathbb R}
\newcommand\Z{\mathbb Z}
\newcommand\ms{\mathcal S}
\newcommand\mt{\mathcal T}
\newcommand\TDC{\mathrm{TDC}}
\newcommand\OEC{\mathrm{OSC^E_1}}

\newcommand\TOEC{\TDC\cap\OEC}
\newcommand\LCN{h_{\mathrm{L}}}
\newcommand\PI{\mathrm{PI}}

\begin{document}

\title[Lipschitz Equivalence Class, Ideal Class and Class number Problem]%
{Lipschitz Equivalence Class, Ideal Class and the Gauss Class Number Problem}

\author{Li-Feng Xi}

\address{Institute of Mathematics, Zhejiang Wanli University, Ningbo,
Zhejiang, 315100, P.~R. China}

\email{xilifengningbo@yahoo.com}

\author{Ying Xiong}

\address{Department of Mathematics, South China University of Technology,
Guangzhou, 510641, P.~R. China}

\email{xiongyng@gmail.com}



\thanks{Ying Xiong is the corresponding author.}

\thanks{Supported by National Natural Science Foundation of China (Grant Nos
11101159, 11071224, 11071082), NCET, Fundamental Research Funds for the
Central Universities, SCUT (2012ZZ0073), and Morningside Center of
Mathematics.}

\begin{abstract}
  Classifying fractals under bi-Lipschitz mappings in fractal geometry just
  as important as classifying topology spaces under homeomorphisms in
  topology. This paper concerns the Lipschitz equivalence of totally
  disconnected self-similar sets in~$\R^d$ satisfying the OSC and with
  commensurable ratios. We obtain the complete Lipschitz invariants for
  such self-similar sets. The key invariant we found is an ideal related to
  the IFS. This discovery establishes a one-to-one correspondence between
  the Lipschitz equivalence classes of self-similar sets and the ideal
  classes in a related ring. Accordingly, two self-similar sets $A$ and $B$
  with the same dimension and ratio root are Lipschitz equivalent if and
  only if their ideals $I_A$ and $I_B$ are equivalent, i.e., $aI_A=bI_B$
  for some elements $a$ and $b$ in the related ring~$R$. This result
  reveals an interesting relationship between the Lipschitz class number
  problem and the Gauss class number one problem for real quadratic fields,
  which was proposed by Gauss in~1801 but still remains a open question
  today. Our result implies that the development on the Lipschitz class
  number problem may lead to deeper understanding of the Gauss class number
  problem.

  By the Jordan-Zassenhaus Theorem in algebraic number theory on the
  finiteness of ideal classes, we further prove a finiteness result about
  the Lipschitz equivalence classes under the commensurable condition. This
  result says that the geometrical structures of such self-similar sets are
  essentially finite in view of Lipschitz equivalence, although the OSC
  allows small copies of the self-similar sets touch in infinite geometric
  manners. In other words, the above finiteness result describe the open
  set condition in terms of Lipschitz equivalence.

  By contrast, we also study the non-commensurable case. It turns out that
  the difference between the commensurable case and the non-commensurable
  case is essential. In fact, we consider the family of self-similar sets
  under the same restrictions only dropping the commensurable condition,
  and then find that there may exist infinitely many Lipschitz equivalent
  classes.

  The simplest case of our result is that the related ring is a principal
  ideal domain. Then the class number is one, there is only one Lipschitz
  equivalence class, all self-similar sets in this class are Lipschitz
  equivalent to a symbolic metric space. For example, the ring $\Z[1/N]$ is
  a principal ideal domain for positive integer $N\ge 2$, then the above
  result implies: suppose $A$ and $B$ are totally disconnected self-similar
  sets satisfying the open set condition, if both $A$ and $B$ are generated
  by $N$ contracting similarities with the same ratio $r$, then $A$ and $B$
  are Lipschitz equivalent. This very special corollary of our main result
  generalizes many known results on the Lipschitz equivalence of
  self-similar sets.
\end{abstract}

\maketitle

\newpage

\tableofcontents

\newpage

\section{Introduction}\label{sec:intro}

\subsection{Background}

Two metric spaces $(X_1,\dst_1)$ and $(X_2,\dst_2)$ are said to be
\emph{Lipschitz equivalent}, denoted by $X_1\simeq X_2$, if there is a
bijection $f\colon X_1\to X_2$ which is \emph{bi-Lipschitz}, i.e., there
exists a constant $L\ge1$ such that
\[L^{-1}\dst_1(x,y)\le\dst_2\bigl(f(x),f(y)\bigr)\le L\dst_1(x,y)\quad
\text{for all $x,y\in X_1$}.
\]
Roughly speaking, the spaces $X_1$ and $X_2$ are ``almost the same'' from the
viewpoint of metric.

Two Lipschitz equivalent fractals can be considered as having the same
geometrical structure since many important geometrical properties are
invariant under the bi-Lipschitz mappings, such as
\begin{itemize}
  \item fractal dimensions: Hausdorff dimension, packing dimension, etc;
  \item properties of measures: doubling, Ahlfors-David regularity, etc;
  \item metric properties:  uniform perfectness, uniform disconnectedness,
      etc.
\end{itemize}
By contrast, Gromov pointed out in~\cite{Gromo07} that ``isometry'' leads to
a poor and rather boring category and ``continuity'' takes us out of geometry
to the realm of pure topology. So Lipschitz equivalence is suitable for the
study of fractal geometry of sets. Classifying fractals under bi-Lipschitz
mappings in fractal geometry just as important as classifying topology spaces
under homeomorphisms in topology.

Another interesting motivation of studying Lipschitz equivalence of fractals
comes from geometry group theory (see~\cite{Bonk06,FarMo98}). For example,
Farb and Mosher~\cite{FarMo98} established a quasi-isometry (in Gromov's
sense) from the group
\begin{equation*}
  \xb{BS}(1,n)=\langle a,b\,|\, aba=b^{n}\rangle
\end{equation*}
to some space with its upper boundary $C_n$ being a self-similar fractal.
Then they proved that $\xb{BS}(1,n)$ and $\xb{BS}(1,m)$ are quasi-isometric
if and only if two self-similar fractals $C_{n}$ and $C_{m}$ are Lipschitz
equivalent. In the appendix of~\cite{FarMo98}, Cooper obtained that $C_n
\simeq C_m$ if and only if $\log m/\log n\in \Q$.

In general it is very difficult to determine whether two fractals are
Lipschitz equivalent or not. Indeed, there is \emph{little} known about the
Lipschitz equivalence of fractals, even for the most familiar fractals---self
similar sets in Euclidean spaces.

This paper concerns the Lipschitz equivalence of totally disconnected
self-similar sets in~$\R^d$ satisfying the OSC and with commensurable ratios
(Definition~\ref{d:cms}). We obtain the \emph{complete} Lipschitz invariants
for such self-similar sets (Theorem~\ref{t:TOCElip}). This is the first
\emph{general} result on the Lipschitz equivalence of self-similar sets with
overlaps. The key invariant we found is an ideal related to the IFS
(Definition~\ref{d:ideal}). This discovery establishes the connection between
Lipschitz equivalence class of self-similar sets and \emph{ideal class} in
algebraic number theory. (See, e.g., \cite{Lang94,Neuki99} for detailed
introduction of algebraic number theory).
\begin{defn}[ideal class]\label{d:class}
  Two nonzero ideals $I$ and $J$ of an integral domain~$R$ is said to be in
  the same class if $aI=bJ$ for some $a,b\in R$. The corresponding
  equivalence classes is called the \emph{ideal classes} of~$R$. The
  \emph{class number} of~$R$, denoted by $h(R)$, is defined to be the number
  of ideal classes.
\end{defn}

Historically, ideal theory was developed in the investigation of Fermat's
Last Theorem. In 1844, Kummer proved Fermat's Last Theorem for every odd
prime number $p\le 19$, based on the fact that the ring $\Z[e^{2\pi i/p}]$ is
a unique factorization domain for such~$p$. This is equivalent to $\Z[e^{2\pi
i/p}]$ has class number $h_{p}=1$. But $h_{p}>1$ for every odd prime number
$p\ge 23$, and so this proof failed for such cases. To settle this problem,
in 1847, Kummer introduced ``ideal numbers'' to recover a form of unique
factorization for the ring~$\Z[e^{2\pi i/p}]$. As a result, Kummer can prove
Fermat's Last Theorem for all regular prime numbers $p$, which are the prime
numbers such that $p\nmid h_{p}$ (all odd prime numbers less than $100$ are
regular except for $37,59,67$). Kummer's idea of ``ideal number'' was further
developed by Dedekind, who established the modern theory of ideal in Algebra.

The study of ideal classes goes back to Lagrange and Gauss, before Kummer and
Dedekind's work on ideal. In 1773, Lagrange developed a general theory to
handle the problem of when an integer~$m$ is representable by a given binary
quadratic form
\[m=ax^2+bxy+cy^2,\]
where $a$, $b$ and $c$ are fixed integers with $\gcd(a,b,c)=1$. Some special
cases of this problem had been studied by Fermat and Euler. Lagrange defined
two quadratic forms $ax^2+bxy+cy^2$ and $AX^2+BXY+CY^2$ to be equivalent if
there exists an invertible integral linear change of variables
\[\begin{pmatrix}
    x \\
    y \\
  \end{pmatrix}=
  \begin{pmatrix}
    \alpha & \beta \\
    \gamma & \delta \\
  \end{pmatrix}
  \begin{pmatrix}
    X \\
    Y \\
  \end{pmatrix},\quad
  \text{where $\alpha,\beta,\gamma,\delta\in\Z$ with}\
  \begin{vmatrix}
    \alpha & \beta \\
    \gamma & \delta \\
  \end{vmatrix}=\pm1,
\]
that transforms $ax^2+bxy+cy^2$ to $AX^2+BXY+CY^2$. Note that equivalent
forms have the same discriminant~$D=b^2-4ac$, and more important that
equivalent forms represent the same set of integers. Let $h(D)$ denote the
number of equivalence classes of binary quadratic forms with discriminant
$D$. In \emph{Disquisitiones Arithmeticae} published in~1801, Gauss further
studied this problem and proved that $h(D)$ is finite for every value~$D$.
Moreover, Gauss put forward his famous class number problems (see
Section~\ref{ssec:clsnmb}). The relationship between the equivalence class of
quadratic forms and ideal class is that, for every negative square free
integer~$D$, the equivalence classes of binary quadratic forms with
discriminant~$D$ correspond \emph{one-to-one} to the ideal classes of the
ring~$\xc O_D$, which is the ring of all algebraic integers
in~$\mathbb{Q}(\sqrt D)$. In other words, $h(D)$ equals the class number
of~$\xc O_D$ for every negative square free integer~$D$ (see, e.g.,
\cite{Goldf85,Stark07}).

Just like the binary quadratic forms, for some appropriate families of
self-similar sets, we can establish a \emph{one-to-one} correspondence
between the Lipschitz equivalence classes and the ideal classes of a related
ring (Theorem~\ref{t:TOCEpr} and~\ref{t:TOCELCN}).

\subsection{Main results}

For convenience, we recall some basic notions of self-similar sets,
see~\cite{Falco97,Falco03,Hutch81} for more details. Let
$\ms=\{S_1,S_2,\dots,S_N\}$ be an \emph{iterated functions system} (IFS) on a
complete metric space~$(X,\dst)$ where each $S_i$ is a contracting similarity
of ratio~$r_i\in(0,1)$, i.e., $\dst\bigl(S_i(x),S_i(y)\bigr)=r_i\dst(x,y)$.
The self-similar set generated by the IFS~$\ms$ is the unique nonempty
compact set $E_\ms\subset X$ such that $E_\ms=\bigcup_{i=1}^N S_i(E_\ms)$. We
say that the IFS~$\ms$ satisfies the \emph{strong separation condition} (SSC)
if the sets $\{S_i(E_\ms)\}$ are pairwise disjoint. The IFS~$\ms$ satisfies
the \emph{open set condition} (OSC) if there exists a nonempty bounded open
set~$O$ such that the sets $S_i(O)$ are disjoint and contained in~$O$. If
furthermore $O\cap E_\ms\ne\emptyset$, we say that the IFS~$\ms$ satisfies
the \emph{strong open set condition} (SOSC). Obviously, we have
\[\text{SSC}\Rightarrow\text{SOSC}\Rightarrow\text{OSC}.\]
For IFSs on Euclidean spaces, Hutchinson~\cite{Hutch81}, Bandt and
Graf~\cite{BanGr92} and Schief~\cite{Schie94} proved that $\hdim E_\ms$
equals the \emph{similarity dimension}~$s$ (the unique positive solution of
$\sum_{i=1}^N r_i^s=1$) if $\ms$ satisfies the OSC, and that
\[\text{SOSC}\Leftrightarrow\text{OSC}\Leftrightarrow\xc H^s(E_\ms)>0.\]
Here $\xc H^s$ denotes the $s$-dimensional Hausdorff measure.

\begin{defn}[commensurable]\label{d:cms}
  The ratios~$r_1,\dots,r_N$ of the IFS~$\ms$ are said to be
  \emph{commensurable} if the multiplicative group generated by
  $\{r_1,\dots,r_N\}$ can be generated by a single number $r_\ms\in(0,1)$. In
  other words, $r_i=r_\ms^{\lambda_i}$ with $\lambda_i\in\N$ for each $i$ and
  $\gcd(\lambda_1,\dots,\lambda_N)=1$. We call $r_\ms$ the \emph{ratio root}
  of~$\ms$.
\end{defn}
\begin{rem}
  The ratios $r_1,\dots,r_N$ are commensurable if and only if $\log r_i/\log
  r_j\in\Q$ for $1\le i,j\le N$.
\end{rem}
We say the IFS~$\ms$ satisfies the TDC, denoted by $\ms\in\TDC$, if the
corresponding self-similar set~$E_\ms$ is totally disconnected. Write
\begin{multline*}
  \OEC=\bigl\{\ms\colon\text{$\ms$ is an IFS on a Euclidean space}\\
  \text{satisfying the OSC and with commensurable ratios}\bigr\}.
\end{multline*}

In this paper, we introduce an ideal related to~$\ms\in\TOEC$ which turns out
to be a very important Lipschitz invariant. For an IFS
$\ms=\{S_1,\dots,S_N\}$ satisfying the OSC, the \emph{natural
measure}~$\mu_\ms$ of~$\ms$ is defined to be the normalized $s$-dimensional
Hausdorff measure restricted to~$E_\ms$, where $s=\hdim E_\ms$, i.e.,
$\mu_\ms=\xc H^s|_{E_\ms}/\xc H^s(E_\ms)$. Note that $\mu_\ms$ is the unique
Borel probability measure such that
\[\mu_\ms(A)=\sum_{i=1}^Nr_i^s\mu_\ms\bigl(S_i^{-1}(A)\bigr)\quad
\text{for all Borel sets~$A$}.
\]
For $\ms\in\OEC$, we call $p_\ms=r_\ms^s$ the \emph{measure root} of~$\ms$
(see Remark~\ref{r:muB}). It follows from $\sum_{i=1}^Nr_i^s=1$ and $\log
r_i/\log r_\ms$ is a positive integer that $p_\ms^{-1}$ is an algebraic
integer. Let
\[\Z[p_\ms]=\bigl\{P(p_\ms)\colon\text{$P$ is a polynomial with integer
coefficients}\bigr\}\]
be the ring generated by $p_\ms$ over the integer set~$\Z$.

\begin{defn}[interior separated set]
  Suppose that $\ms\in\TOEC$. A compact set $F\subset E_\ms$ is called an
  \emph{interior separated set} of~$E_\ms$ if $E_\ms\setminus F$ is also
  compact and $F\subset O$ for some open set $O$~satisfying the OSC.
\end{defn}

We remark that $\mu_\ms(F)\in\Z[p_\ms]$ for every interior separated set~$F$.

\begin{defn}[ideal of IFS]\label{d:ideal}
  Suppose that $\ms\in\TOEC$. The ideal of~$\ms$, denoted by~$I_\ms$, is
  defined to be the ideal of~$\Z[p_\ms]$ generated by
  \[\bigl\{\mu_\ms(F)\colon\text{$F$ is an interior separated set
  of~$E_\ms$}\bigr\}.\]
\end{defn}
\begin{rem}
  It is worth noting that the ideal~$I_\ms$ depends not only on the algebraic
  properties of ratios, but also on the geometrical structure of the
  self-similar set~$E_\ms$.
\end{rem}
\begin{rem}
  At first sight it seems that we need find all open sets which satisfy the
  SOSC if we want to determine the ideal of an IFS. Fortunately, one such
  open set is enough, see Remark~\ref{r:ideal} in Section~\ref{ssec:BDnum}.
\end{rem}
\begin{exmp}\label{e:idSSC}
  We have $I_\ms=\Z[p_\ms]$ when $\ms$ satisfies the SSC. Indeed, the open
  set $O=\{x\colon\dist(x,E_\ms)<\eps\}$ satisfies the OSC for $\eps$ small
  enough. And so $E_\ms\subset O$ is an interior separated set. Therefore
  $1=\mu_\ms(E_\ms)\in I_\ms$ and $I_\ms=\Z[p_\ms]$.
\end{exmp}

Our main result gives the complete Lipschitz invariants of self-similar sets
generated by IFSs in $\TOEC$. This is the first \emph{general} result on the
Lipschitz equivalence of self-similar sets with overlaps.

\begin{thm}\label{t:TOCElip}
  Suppose that $\ms,\mt\in\TOEC$, then $E_\ms\simeq
  E_\mt$ if and only if
  \begin{enumerate}[\upshape(i)]
    \item $\hdim E_\ms=\hdim E_\mt$;
    \item $\log r_\ms/\log r_\mt\in\Q$;
    \item $I_\ms=aI_\mt$ for some $a\in\R$.
  \end{enumerate}
\end{thm}
\begin{rem}
  We emphasize that, in Theorem~\ref{t:TOCElip}, the two IFSs~$\ms$ and~$\mt$
  are allowed to be defined on Euclidean spaces of different dimensions. For
  example, the IFS~$\ms$ in Example~\ref{e:NPI} is defined on~$\R^1$ and the
  IFS~$\ms$ in Example~\ref{e:ideal} is defined on~$\R^2$. By
  Theorem~\ref{t:TOCElip}, the two corresponding self-similar sets are
  Lipschitz equivalent, see Figure~\ref{fig:NPI} and~\ref{fig:ideal}.
\end{rem}

Theorem~\ref{t:TOCElip} offers much deep insight into the geometrical
structure of self-similar sets generated by IFSs in~$\TOEC$. To make this
more clear, we shall consider a family of self-similar sets with the same
ratio root to eliminate the influence of ratios. Let $\OEC(p,r)$ denote the
set consisting of all IFS~$\ms\in\OEC$ with $p_\ms=p$ and $r_\ms=r$. Given
$\ms,\mt\in\TOEC(p,r)$, we have $\hdim E_\ms=\hdim E_\mt=\log p/\log r$ and
$r_\ms=r_\mt=r$. Therefore, Conditions~(i) and~(ii) in
Theorem~\ref{t:TOCElip} are fulfilled and the ideals $I_\ms$ and~$I_\mt$
belong to the same ring~$\Z[p]$. Consequently, we have
\begin{thm}\label{t:TOCEpr}
  Suppose that $\ms,\mt\in\TOEC(p,r)$, then the two self-similar sets $E_\ms$
  and $E_\mt$ are Lipschitz equivalent if and only if their ideals $I_\ms$
  and $I_\mt$ belong to the same ideal class of~$\Z[p]$.
\end{thm}
Roughly speaking, Theorem~\ref{t:TOCEpr} tells us that different Lipschitz
equivalence classes correspond to different ideal classes, see
Example~\ref{e:NPI}. It is natural to ask whether the correspondence induced
by Theorem~\ref{t:TOCEpr} is one-to-one. Our next result gives an affirm
answer to this question. We define the number of Lipschitz equivalence
classes of self-similar sets generated by IFSs in~$\TOEC(p,r)$ to be the
\emph{Lipschitz class number} of~$\TOEC(p,r)$, denoted by $\LCN(p,r)$.

\begin{thm}\label{t:TOCELCN}
  Suppose that $\TOEC(p,r)\ne\emptyset$. Then the Lipschitz equivalent
  classes of self-similar sets generated by IFSs in~$\TOEC(p,r)$ correspond
  one-to-one to the ideal classes of~$\Z[p]$. This means
  $\LCN(p,r)=h(\Z[p])$. Moveover, the Lipschitz class number $\LCN(p,r)$ is
  finite for every pair $p,r$.
\end{thm}

The most significance of Theorem~\ref{t:TOCELCN} is the one-to-one
correspondence between the Lipschitz equivalence classes and the ideal
classes. We will further discuss this point in next subsection.

It is also worth noting that the finiteness result about the Lipschitz
equivalence classes gives some interesting information about the OSC. For
self-similar fractals, the OSC is a generally accepted separation condition,
but it is too complicated to describe completely in geometry. For example,
the OSC allows the small copies of self-similar set touch in infinitely many
geometric manners. However, the finiteness result in Theorem~\ref{t:TOCELCN}
says that the touching manners is essentially finite in view of Lipschitz
equivalence. In other words, this finiteness result describes the open set
condition in term of Lipschitz equivalence.

We present two examples to illustrate Theorem~\ref{t:TOCELCN}. For
positive square free integer~$D$, let $\xc O_D$ be the ring of all algebraic
integers in the field~$\Q(\sqrt D)$. We know from algebraic number theory
that
\[\xc O_D=\begin{cases}
  \Z[\sqrt D], &\text{if $D\equiv 2$ or $3\pmod4$};\\
  \Z[\frac{1+\sqrt D}2], &\text{if $D\equiv1\pmod4$};
\end{cases}\]
and
\begin{align*}
  h(\xc O_D)=1\quad & \text{for $D=2,3,5,6,7,11,13,14,17,19,21,22,23,29,\dots$}; \\
  h(\xc O_D)=2\quad & \text{for $D=10,15,26,30,34,35,39,42,51,55,58,65,\dots$}.
\end{align*}
For more square free $D>0$ with $h(\xc O_D)=1$ or $2$,
see~\cite{MolWi91,MolWi92}. Using the above facts about the class number
of~$\xc O_D$, we give the following two examples.

\begin{exmp}
  Let $p=(\sqrt5-1)/2$, then $p^4+p^3+p=1$ and $p^3+2p^2=1$. Suppose that
  $\ms,\mt\in\TOEC(p,r)$ with ratios $r^4,r^3,r$ and $r^3,r^2,r^2$,
  respectively, then $p_\ms=p_\mt=p$. Since $\Z[p]=\Z[p+1]=\xc O_5$ has
  class number one, we have $E_\ms\simeq E_\mt$. In this example, the
  relative positions of the small copies of self-similar sets $E_\ms$ and
  $E_\mt$ do \emph{not} affect the Lipschitz equivalence.
\end{exmp}

The following example involves the ring~$\xc O_{10}$, which has class
number~$2$. We consider the IFS family $\TOEC(p,r)$ with $p=\sqrt{10}-3$ and
$r=1/10$. Theorem~\ref{t:TOCELCN} says that there are two Lipschitz
equivalence classes since $\Z[p]=\Z[\sqrt{10}]=\xc O_{10}$.

\begin{figure}[h]
  \centering
  \includegraphics{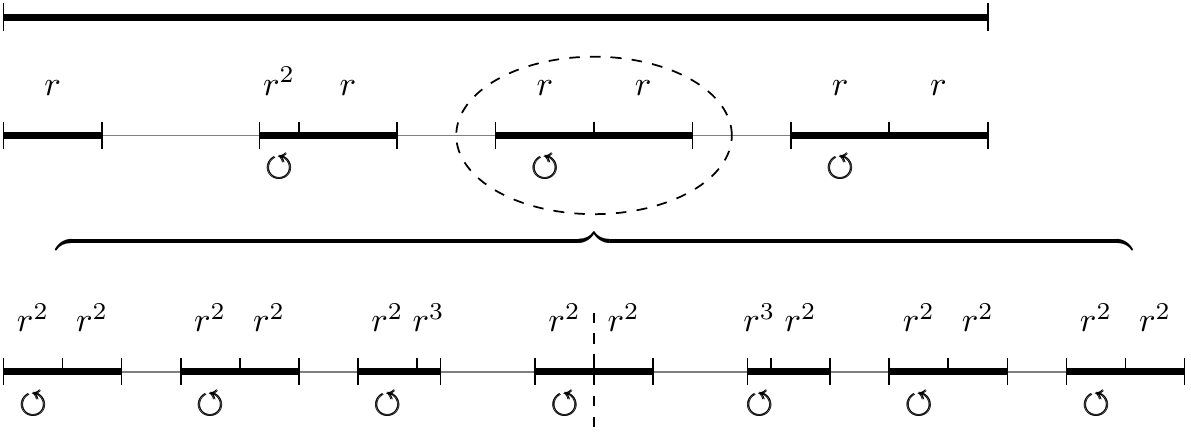}\\
  \caption{The structure of~$E_\ms$ in Example~\ref{e:NPI}}
  \label{fig:NPI}
\end{figure}

\begin{exmp}\label{e:NPI}
  Let $\ms=\{S_1,S_2,\dots,S_7\}\in\TOEC(p,r)$, where $p=\sqrt{10}-3$,
  $r=1/10$ and
  \begin{gather*}
    S_1\colon x\mapsto rx,\quad S_2\colon x\mapsto -r^2x+3r,\quad
    S_3\colon x\mapsto rx+3r,\\
    S_4\colon x\mapsto -rx+6r,\quad S_5\colon x\mapsto rx+6r,\quad
    S_6\colon x\mapsto -rx+9r,\quad S_7\colon x\mapsto rx+9r.
  \end{gather*}
  See Figure~\ref{fig:NPI}. (In Figure~\ref{fig:NPI}, the
  symbol~$\bm\circlearrowleft$ means that there is a minus sign in the
  contraction coefficient of the corresponding similarity. In geometry, this
  means a rotation by the angle~$\pi$.) Then $p_\ms=p=\sqrt{10}-3$ satisfying
  the equation $p^2+6p=1$, and Example~\ref{e:gdid} shows
  \[I_\ms=(2,p+1)=(2,\sqrt{10})=\bigl\{2a+b\sqrt{10}\colon a,b\in\Z\bigr\}.\]
  One can check that $I_\ms$ is not a principal ideal of the ring
  $\Z[p_\ms]=\xc O_{10}$.

  Let $\mt\in\TOEC(p,r)$ satisfying the SSC with ratios $r^2,r,r,r,r,r,r$.
  Then $p_\mt=p=\sqrt{10}-3$, and by Example~\ref{e:idSSC},
  $I_\mt=\Z[p_\mt]=\xc O_{10}$ is a principal ideal.

  Theorem~\ref{t:TOCELCN} implies that $E_\ms\not\simeq E_\mt$ and that
  either $E\simeq E_\ms$ or $E\simeq E_\mt$ for every self-similar set~$E$
  generated by IFS in $\TOEC(p,r)$. In this example, the relative positions
  of the small copies of self-similar sets do affect the Lipschitz
  equivalence.
\end{exmp}

We close this subsection with a necessary and sufficient condition for the
IFS family $\TOEC(p,r)\ne\emptyset$.
\begin{prop}\label{p:p,r}
  The IFS family $\TOEC(p,r)\ne\emptyset$ if and only if $p,r\in(0,1)$ and
  there exist positive integers $\lambda_1,\lambda_2,\dots,\lambda_N$ with
  $\gcd(\lambda_1,\dots,\lambda_N)=1$ such that
  \[p^{\lambda_1}+p^{\lambda_2}+\dots+p^{\lambda_N}=1.\]
  Here we allow that $\lambda_i=\lambda_j$ for $i\ne j$.
\end{prop}

\subsection{Class number problem}\label{ssec:clsnmb}

In this subsection, we will discuss the Lipschitz class number
of~$\TOEC(p,r)$ by making use of Theorem~\ref{t:TOCELCN}. This problem is
closely related to the famous class number problems posed by Gauss in
Articles~303 and~304 of \emph{Disquisitiones Arithmeticae} of~1801.

Recall that the class number of an algebraic number field is defined to be
the class number of the ring of all algebraic integers in it. We state some
Gauss class number problems in the modern terminology.

\begin{cthm}[Gauss class number conjecture]
  The number of imaginary quadratic fields $\Q(\sqrt D)$ $(D<0)$ which have a
  given class number~$n$ is finite.
\end{cthm}
This conjecture was solved by Heilbronn in~1934. Gauss further posed the
following problem.
\begin{cthm}[Gauss class number problem]
  For small~$n$, determine all $D<0$ such that the class number of~$\Q(\sqrt
  D)$ equals~$n$.
\end{cthm}
Watkins~\cite{Watki04} gave the solution for $n\le100$ in~2004. As a special
case,
\begin{cthm}[Gauss class number one problem for imaginary quadratic fields]
  There are only nine imaginary quadratic fields~$\Q(\sqrt D)$ $(D<0)$ with
  class number one. They are
  \[D\in\{-1,-2,-3,-7,-11,-19,-43,-67,-163\}.\]
\end{cthm}
This conjecture was solved by Baker~\cite{Baker68} and Stark~\cite{Stark67}
independently in~1967. For the contrasting case of real quadratic fields,
Gauss conjecture
\begin{cthm}[Gauss class number one problem for real quadratic fields]
  There are infinitely many real quadratic fields~$\Q(\sqrt D)$ $(D>0)$ with
  class number one.
\end{cthm}
This conjecture is still open. We do not even yet know whether there are
infinitely many algebraic number fields (of arbitrary degree) with class
number one.

\medskip

Like the Gauss class number problems, we can pose the Lipschitz class number
problem.
\begin{ques}\label{q:LCN}
  For a given integer~$n>0$, determine all~$p,r$ such that $\LCN(p,r)=n$.
\end{ques}

Theorem~\ref{t:TOCELCN} says that $\LCN(p,r)=n$ if and only if
$\TOEC(p,r)\ne\emptyset$ (see Proposition~\ref{p:p,r}) and $h(\Z[p])=n$.
However, from the Gauss class number problems, we know that, in algebraic
number theory, it is very difficult to determine all the algebraic
numbers~$p$ such that the ring~$\Z[p]$ has a given class number~$n$. In other
words, it seems very hard to solve Question~\ref{q:LCN} by making use of
algebraic number theory. A natural question is: can we obtain some results on
Question~\ref{q:LCN} by analyzing the geometrical structure of self-similar
sets directly? Such results, if obtained, might lead to deeper understanding
of the Gauss class number problems. Of course, this is also very difficult.
We don't know how to do it so far.

\subsection{More results on principal ideal}

The simplest case of Theorem~\ref{t:TOCELCN} is that the ring~$\Z[p]$ is a
principal ideal domain.
\begin{thm}\label{t:PID}
  Suppose that $\ms,\mt\in\TOEC(p,r)$. If $\Z[p]$ is a principal ideal
  domain, then $E_\ms\simeq E_\mt$.
\end{thm}
\begin{rem}
  Like the SSC, in this case the relative positions of the small copies of
  self-similar sets do \emph{not} affect the Lipschitz equivalence.
\end{rem}

The simplest example of $\Z[p]$ to be a principal ideal domain is that
$p=1/N$ for integers $N\ge2$. Given $\ms=\{S_1,\dots,S_N\}\in\TOEC$ with the
$N$ ratios are all equal to $r$, the ring related to~$\ms$ is just $\Z[1/N]$.
This leads to the following theorem.

\begin{thm}\label{t:srt}
  Suppose $E\subset\R^d$ and $E'\subset\R^{d'}$ are totally disconnected
  self-similar sets satisfying the open set condition, if both $E$ and $E'$
  are generated by $N$ contracting similarities with the same ratio $r$, then
  $E$ and $E'$ are Lipschitz equivalent.
\end{thm}
\begin{rem}\label{r:srt}
  Theorem~\ref{t:srt} generalizes many known results on the Lipschitz
  equivalence of self-similar sets (see Section~\ref{ssec:OSC}), although it
  is only a very special corollary of Theorem~\ref{t:TOCELCN}.
\end{rem}

On the other hand, it is natural to think that a self-similar set has the
simplest geometrical structure if it is Lipschitz equivalent to a
self-similar set satisfying the SSC. If the set is generated by an
IFS~$\ms\in\TOEC$, by Theorem~\ref{t:TOCElip} and Example~\ref{e:idSSC}, this
is equivalent to that the ideal $I_\ms$ is a principle ideal. Let $\PI$
denote the set of all IFS $\ms\in\TOEC$ such that the ideal~$I_\ms$ is a
principal ideal. It follows from Theorem~\ref{t:TOCElip} that
\begin{thm}\label{t:PI}
  Suppose that $\ms,\mt\in\PI$, then $E_\ms\simeq E_\mt$ if and only if
  \begin{enumerate}[\upshape(i)]
    \item $\hdim E_\ms=\hdim E_\mt$;
    \item $\log r_\ms/\log r_\mt\in\Q$;
    \item $\Z[p_\ms]=\Z[p_\mt]$.
  \end{enumerate}
\end{thm}
The point is that the condition $I_\ms=aI_\mt$ in Theorem~\ref{t:TOCElip} is
equivalent to $\Z[p_\ms]=\Z[p_\mt]$ provided that $I_\ms$ and $I_\mt$ are
both principle ideals. In fact, under the assumption of being principle
ideal, $I_\ms=aI_\mt$ is equivalent to $\Z[p_\ms]=b\Z[p_\mt]$ for some
$b\in\R$. Then observe that $b\in\Z[p_\ms]$ since $1\in\Z[p_\mt]$, and so
$b\Z[p_\ms]\subset\Z[p_\ms]=b\Z[p_\mt]$, i.e., $\Z[p_\ms]\subset\Z[p_\mt]$.
By symmetry we have $\Z[p_\ms]=\Z[p_\mt]$. However, in general $I_\ms=aI_\mt$
is not equivalent to $\Z[p_\ms]=\Z[p_\mt]$, see Example~\ref{e:NPIZ}.
\begin{exmp}\label{e:NPIZ}
  Let $p_\ms=\sqrt{10}-3$ be the positive solution of the equation
  $p_\ms^2+6p_\ms=1$ and $p_\mt=37\sqrt{10}-117$ the positive solution of the
  equation $p_\mt^2+234p_\mt=1$. Then
  $\Z[p_\ms]=\Z[\sqrt{10}]\ne\Z[p_\ms]=\Z[37\sqrt{10}]$. Let
  $I_\ms=I_\mt=37(\sqrt{10}\Z+2\Z)$. One can check that $I_\ms$ is an ideal
  of~$\Z[p_\ms]$ and $I_\mt$ is an ideal of~$\Z[p_\mt]$. Thus, $I_\ms=I_\mt$
  but $\Z[p_\ms]\ne\Z[p_\mt]$.
\end{exmp}

A natural question arises:
\begin{ques}\label{q:PI}
  For what IFS~$\ms\in\TOEC$, is the ideal $I_\ms$ a principal ideal?
\end{ques}
We remark that Example~\ref{e:idSSC} says that $\PI$ contains all IFS
$\ms\in\TOEC$ satisfying the SSC. It is also obviously that $\ms\in\PI$ if
$\Z[p_\ms]$ is a principle ideal domain. We further give another partial
answer to Question~\ref{q:PI}. An IFS~$\ms$ on~$\R^d$ is said to be
\emph{orthogonal homogeneous} if there is a $d\times d$ orthogonal
matrix~$\bm A$ such that each $S_i\in\ms$ has the form $S_i\colon x\mapsto
r_i\bm Ax+b_i$ with $r_i\in(0,1)$. In other words, the similarities in~$\ms$
have the same orthogonal part but their ratios may be different. We say an
IFS~$\ms$ satisfies the convex open set condition (COSC) if $\ms$ satisfies
the OSC with a convex open set. Let
\begin{multline}\label{eq:PIS}
  \xs S:=\Bigl\{\ms\in\TOEC\colon \text{$\ms$ satisfies the COSC}\\
  \text{and is orthogonal homogeneous}\Bigr\}.
\end{multline}
\begin{thm}\label{t:covopen}
 For every $\ms\in\xs S$, we have $I_\ms=\Z[p_\ms]$. As a result,
 $\xs S\subset\PI$.
\end{thm}
\begin{rem}
  We don't know whether Theorem~\ref{t:covopen} is still true if we drop the
  COSC. On the other hand, Example~\ref{e:NPI} and~\ref{e:ideal} implies that
  the condition that $\ms$ is orthogonal homogeneous can not be relaxed too
  much.
\end{rem}
\begin{rem}
  Under the commensurable case, Theorem~\ref{t:PI} and~\ref{t:covopen} extend
  the results in~\cite{RuWaX12,XiRu07} in a very general setting for IFSs
  on~$\R^d$ ($d\ge1$), see Section~\ref{ssec:OSC}.
\end{rem}

The paper is organized as follows. In Section~\ref{sec:SSS}, we review some
known results about the Lipschitz equivalence of self-similar sets and
present some new results on the non-commensurable case, including
Theorem~\ref{t:Z+} and~\ref{t:sublip}. Section~\ref{sec:Zp} concerns the
algebraic properties of measure root. As a result, we give the proof of
Proportion~\ref{p:p,r}. Section~\ref{sec:ideal} devoted to the proof of
Theorem~\ref{t:TOCELCN} and~\ref{t:covopen}. This is based on some techniques
of computing the ideal of IFS, see Theorem~\ref{t:gdid}
and~\ref{t:separated}. Section~\ref{sec:BD} introduces the notions of blocks
decomposition, interior blocks and measure polynomials. This section also
prove some basic results, such as the finiteness of the measure polynomials
(Proposition~\ref{p:finiteCP}) and the cardinality of boundary blocks and
interior blocks (Lemma~\ref{l:C(k)}, \ref{l:CP(k)} and~\ref{l:CB(k)}). All of
this are fundamental to our study. Section~\ref{sec:idea} discusses the main
ideas behind the proof of Theorem~\ref{t:TOCElip}, including the cylinder
structure (Definition~\ref{d:qscldr}), the dense island structure
(Definition~\ref{d:dnsilnd}), the measure linear property
(Definition~\ref{d:msrln}) and the suitable decomposition
(Definition~\ref{d:suitde}). We conclude this section with
Lemma~\ref{l:suitde}, which is the tool to construct the same cylinder
structure. The proof of Theorem~\ref{t:TOCElip} is presented in
Section~\ref{sec:LipBE} and~\ref{sec:proof}. By making use of cylinder
structure and dense island structure, we first prove the whole self-similar
set is Lipschitz equivalent to interior blocks of it
(Proposition~\ref{p:BlipE}), then deal with the Lipschitz equivalence between
interior blocks of different self-similar sets (Proposition~\ref{p:BlipB}).
Thus the proof of Theorem~\ref{t:TOCElip} is complete. Finally, we study the
non-commensurable case and give the proofs of Theorem~\ref{t:Z+}
and~\ref{t:sublip} in Section~\ref{sec:NC}.

\section{Geometrical structure of Self-similar Sets}
\label{sec:SSS}

\subsection{About self-similar sets}

Self-similar sets in Euclidean spaces are fundamental objects in fractal
geometry. However, we do not know much about them.

Given an IFS $\ms=\{S_i\}_{i=1}^N$ consisting of contracting similarities
$\ms=\{S_i\}_{i=1}^N$ on~$\R^d$, Hutchinson~\cite{Hutch81} showed that there
is a unique nonempty compact set~$E_\ms\subset\R^d$, called self-similar set,
such that
\[E_\ms=\bigcup_{i=1}^NS_i(E_\ms).\]
Conversely, given a self-similar set~$E$, it is not easy to determine all the
IFSs which generate~$E$, even under some reasonable additional conditions.
This is why we state our results by IFSs rather than self-similar sets. This
problem is rather fundamental and has some relationship with the Lipschitz
equivalence problem of self-similar sets. In fact, if two IFSs generate the
same self-similar set, they must satisfy the conditions necessary to the
Lipschitz equivalence. It is somewhat surprising that there is little known
about the generating IFSs of a given self-similar set. We refer
to~\cite{DenLa,FenWa09} for detailed study of this problem.

Another basic problem is to determine the dimension of self-similar sets. In
general this problem is very difficult. A open conjecture of Furstenberg says
that $\hdim E_\lambda=1$ for any $\lambda $ irrational, where
$E_\lambda=E_\lambda/3\cup (E_\lambda/3+\lambda/3)\cup (E_\lambda/3+2/3)$.
Although the IFSs involved are rather sample, the conjecture remained open
from 1970s until settled by Hochman~\cite{Hochm12} very recently, see
also~\cite{Kenyo97,RaoWe98,SVe02}. We know much more about the dimension of
self-similar sets if some separation conditions hold. Such conditions control
the overlaps between small copies of self-similar set. The OSC, which means
the overlaps are \emph{small}, was introduced by Moran~\cite{Moran46}. For
IFSs on Euclidean spaces, it is well known from Hutchinson~\cite{Hutch81}
that if $\ms$ satisfies the OSC, then $\hdim E_\ms$ equals the
\emph{similarity dimension}~$s$ (the unique positive solution of
$\sum_{i=1}^N r_i^s=1$) and the Hausdorff measure $\xc H^s(E_\ms)>0$.
Moreover, Bandt and Graf~\cite{BanGr92} and Schief~\cite{Schie94} proved that
\[\text{SOSC}\Leftrightarrow\text{OSC}\Leftrightarrow\xc H^s(E_\ms)>0.\]
Although there are various conditions which equivalent to the OSC obtained
by~\cite{BanGr92,Moran46,Schie94}, in general it is not known how to
determine whether a given IFS satisfies the OSC. We refer
to~\cite{BaHuR06,BanRa07} for more studies on the OSC. Another well studied
separation condition is the \emph{weak separation condition} (WSP), which
extends the OSC while allowing overlaps on the iteration,
see~\cite{DasEd11,LauNg99,Zerner1996}.

If one want to know more about the geometrical structure of self-similar
sets, the information of dimension is not enough, which only tells us about
the size of sets. It is natural to think that the self-similar sets in the
same Lipschitz equivalence class have the same geometrical structure. In this
sense, our result is a step towards the well-understanding of the geometrical
structure of self-similar sets satisfying the OSC. In the remainder of this
section, we review some known results about Lipschitz equivalence of
self-similar sets in Euclidean spaces and generalize almost all of them by
making use of our new results. For other related works on Lipschitz
equivalence, see~\cite{DavSe97,FalMa89,MatSa09,Xi04,Xi07,XiXi12,XioXi09}.

\subsection{The SSC case}

When the self-similar sets satisfy the SSC, their geometrical structure are
clear since there are no overlaps between the small copies $S_i(E_\ms)$. But
the problem of Lipschitz equivalence in this case is rather difficult. It is
not hard to see that, in the SSC case, the algebraic properties of the ratios
of the self-similar sets completely determine whether or not they are
Lipschitz equivalent. However, we do not yet know completely what algebraic
properties affect the Lipschitz equivalence.

Cooper and Pignataro~\cite{CooPi88} studied order-preserving bi-Lipschitz
mappings between self-similar subsets of $\R^1$ and proved the measure linear
property (see Section~\ref{ssec:ML}). Falconer and Marsh~\cite{FalMa92}
obtained two necessary conditions in terms of algebraic properties of ratios.
Based on the ideas in~\cite{CooPi88,FalMa92}, Rao, Ruan and
Wang~\cite{RaRuW12} completely characterize the Lipschitz equivalence for
several \emph{special} kinds of self-similar sets satisfying the SSC. Some
sufficient and necessary conditions on the Lipschitz equivalence in the SSC
case were obtained in Xi~\cite{Xi10}, Llorente and Mattila~\cite{LloMa10} and
Deng and Wen et~al.~\cite{DeWXX11}. But these conditions are not based on the
algebraic properties of ratios and so it is impossible to verify them for
given IFSs.

Our results substantially improves the study of the SSC case. By
Theorem~\ref{t:PI} and Example~\ref{e:idSSC}, we find the complete Lipschitz
invariants in terms of algebraic properties of ratios under the commensurable
condition.

\begin{thm}\label{t:SSC}
  Suppose that $\ms,\mt$ both satisfy the SSC and the ratios of them are
  both commensurable. Then $E_\ms\simeq E_\mt$ if and only if
  \begin{enumerate}[\upshape(i)]
    \item $\hdim E_\ms=\hdim E_\mt$;
    \item $\log r_\ms/\log r_\mt\in\Q$;
    \item $\Z[p_\ms]=\Z[p_\mt]$.
  \end{enumerate}
\end{thm}

We remark that the Conditions~(ii) and~(iii) are independent, see the
following two examples.
\begin{exmp}\label{e:ssc1}
  Let $\ms$ be an IFS satisfying the SSC with ratios~$3^{-1}$, $3^{-1}$,
  $3^{-2}$ and~$3^{-2}$, and $\mt$ an IFS satisfying the SSC with ratios
  \[\underbrace{3^{-3},\dots,3^{-3}}_{20},
  \underbrace{3^{-6},\dots,3^{-6}}_{8}.\]
  Then $p_\ms=\frac{\sqrt3-1}2$ is the positive solution of the equation
  $2p_\ms^2+2p_\ms=1$ and $p_\mt=\frac{3\sqrt3-5}4$ is the positive solution
  of the equation $8p_\mt^2+20p_\mt=1$. We have $\hdim E_\ms=\hdim E_\mt$,
  \[\frac{\log p_\ms}{\log p_\mt}=\frac13\in\Q\quad \text{and}\quad
  \Q(p_\ms)=\Q(p_\mt)=\Q(\sqrt3),\]
  but
  \[\Z[p_\ms]=\Z[\sqrt3,\frac12]\ne\Z[3\sqrt3,\frac12]=\Z[p_\mt].\]
\end{exmp}

\begin{exmp}\label{e:ssc2}
  Let $p_\ms=\frac{\sqrt5-1}4$ be the positive solution of the equation
  $4p_\ms^2+2p_\ms=1$ and $p_\mt=\frac{\sqrt5-2}2$ the positive solution of
  the equation $4p_\mt^2+8p_\mt=1$. Then
  \[\Z[p_\ms]=\Z[p_\mt]=\Z[\sqrt5,\frac12],\quad \text{but}\
  \log p_\ms/\log p_\mt\notin\Q\]
  since $p_\mt=4p_\ms^3$.
\end{exmp}

\subsection{The non-commensurable case}

It is interesting to compare Theorem~\ref{t:SSC} with Falconer and Marsh's
classic result in~\cite{FalMa92}. Without assuming the commensurable
condition, they obtained some \emph{necessary} conditions for $E_\ms\simeq
E_\mt$.

\begin{thm*}[Falconer and Marsh~\cite{FalMa92}]
Suppose that $\ms,\mt$ both satisfy the SSC and $r_1,\dots,r_n$ are ratios
of~$\ms$, $t_1,\dots,t_m$ are ratios of~$\mt$. The following conditions are
necessary for $E_\ms\simeq E_\mt$.
\begin{enumerate}[\upshape(i)]
  \item $\hdim E_\ms=\hdim E_\mt$;
  \item there exist positive integers $u,v$ such that
      \begin{gather*}
      \sgp(r_1^u,\dots,r_n^u)\subset\sgp(t_1,\dots,t_m), \
      \sgp(t_1^v,\dots,t_m^v)\subset\sgp(r_1,\dots,r_n),
      \end{gather*}
      where $\sgp(a_1,\dots,a_n)$ denotes the multiplicative sub-semigroup
      of positive real numbers generated by $a_1,\dots,a_n$;
  \item $\Q(r_1^s,\dots,r_n^s)=\Q(t_1^s,\dots,t_m^s)$, where $s=\hdim
      E_\ms=\hdim E_\mt$.
\end{enumerate}
\end{thm*}

If we assume the commensurable condition, then the Condition~(ii) in
Theorem~\ref{t:SSC} is equivalent to the Condition~(ii) in Falconer and
Marsh's theorem. While the Condition~(iii) in Theorem~\ref{t:SSC} is strictly
stronger than the Condition~(iii) in Falconer and Marsh's theorem. In fact,
let $\ms$ and $\mt$ be as in Example~\ref{e:ssc1}, then $\ms$ and $\mt$
satisfy all the conditions in Falconer and Marsh's theorem. However,
Condition~(iii) of Theorem~\ref{t:SSC} says that the self-similar sets
$E_\ms$ and $E_\mt$ are not Lipschitz equivalent.

This observation inspires the following theorem. For positive numbers
$a_1,\dots,a_n$, let $\Z^+[a_1,\dots,a_n]$ denotes the smallest set that
contains $a_1,\dots,a_n$ and all positive integers, and is closed under
addition and multiplication. In other words,
\begin{multline}\label{eq:Z+[a]}
  \Z^+[a_1,\dots,a_n]=\bigl\{P(a_1,\dots,a_n)\colon\\
  \text{$P$ is a polynomial with positive integer coefficients}\bigr\}.
\end{multline}
\begin{thm}\label{t:Z+}
  Let $\ms$ and $\mt$ be two IFSs satisfying the SSC. Suppose that
  $\ms\simeq\mt$ and $\hdim E_\ms=\hdim E_\mt=s$. Then
  \begin{equation}\label{eq:Z+}
    \Z^+[r_1^s,\dots,r_n^s]=\Z^+[t_1^s,\dots,t_m^s],
  \end{equation}
  where $r_1,\dots,r_n$ are the ratios of~$\ms$ and $t_1,\dots,t_m$ the
  ratios of~$\mt$.
\end{thm}
\begin{rem}
  Theorem~\ref{t:Z+} strengthens the condition~(iii) in Falconer and Marsh's
  theorem. For this, note that
  $\Z^+[r_1^s,\dots,r_n^s]=\Z^+[t_1^s,\dots,t_m^s]$ implies
  $\Z[r_1^s,\dots,r_n^s]=\Z[t_1^s,\dots,t_m^s]$, and the latter implies
  $\Q(r_1^s,\dots,r_n^s)=\Q(t_1^s,\dots,t_m^s)$.
\end{rem}
\begin{rem}\label{r:eqcond}
  Under the commensurable condition, if we assume that $\hdim E_\ms=\hdim
  E_\mt=s$ and the Condition~(ii) in Theorem~\ref{t:SSC}, then the
  Condition~(iii) in Theorem~\ref{t:SSC} is equivalent to~\eqref{eq:Z+}, see
  Lemma~\ref{l:Z[p]}(e).
\end{rem}

For convenience of further discussion, we introduce some notations. Let $\ms$
be an IFS consisting of contracting similarities with ratios $r_1,\dots,r_n$.
Write
\begin{equation}\label{eq:sgpZ+}
  \sgp\ms=\sgp(r_1,\dots,r_n)\quad\text{and}\quad
  \Z^+[\ms]=\Z^+[r_1^s,\dots,r_n^s],
\end{equation}
where $s=\hdim E_\ms$. We call two multiplicative sub-semigroup $G_1$ and
$G_2$ of~$(0,1)$ are equivalent, denoted by $G_1\sim G_2$, if there exist two
positive integers $u$ and $v$ such that $g_1^u\in G_2$ for all $g_1\in G_1$
and $g_2^v\in G_1$ for all $g_2\in G_2$. With these notations, we can rewrite
the above necessary conditions as: if $E_\ms\simeq E_\mt$, then
\begin{enumerate}[\upshape(i)]
  \item $\hdim E_\ms=\hdim E_\mt$;
  \item $\sgp\ms\sim\sgp\mt$;
  \item $\Z^+[\ms]=\Z^+[\mt]$.
\end{enumerate}

From Theorem~\ref{t:SSC}, Theorem~\ref{t:Z+} and Remark~\ref{r:eqcond}, one
might expect that the above necessary conditions are also sufficient for
$E_\ms\simeq E_\mt$. Unfortunately, it turns out that these conditions are
far from being sufficient. Indeed, we can find infinitely many IFSs
satisfying the SSC such that any two of them satisfy above conditions~(i),
(ii) and~(iii), but are not Lipschitz equivalent (Example~\ref{e:infSSC}).
This fact implies that the difference between the commensurable case and the
non-commensurable case is \emph{essential} and that the problem for the
non-commensurable case is much more difficult. This is also why we cannot
drop the commensurable assumption in Theorem~\ref{t:TOCElip}. Among many
difficulties, the lack of some finiteness result like
Proposition~\ref{p:finiteCP} in the non-commensurable case may be the biggest
obstacle. How to settle the problem for the non-commensurable case is still
not clear.

The insufficiency of the conditions (i), (ii) and (iii) follows from a new
criterion for the Lipschitz equivalence. To state it, we need some more
notations.

Let $\ms$ be an IFS consisting of contracting similarities. For every
multiplicative sub-semigroup $G$ of $(0,1)$, write
\[\ms^G=\bigl\{S\in\ms\colon (r\cdot\sgp\ms)\cap G\ne\emptyset,\
 \text{where $r$ is the ratio of~$S$}\bigr\}.\]
\begin{exmp}
  Let $\ms=\{S_1,S_2,S_3,S_4\}$. The corresponding ratios
  \[r_1=a,\quad r_2=a^2,\quad r_3=ab,\quad r_4=b,\]
  where $a,b\in(0,1)$ such that $\log a/\log b\notin\Q$. Then $\sgp\ms$ is
  the multiplicative semigroup generated by $a$ and $b$. Let $G_1$ be the
  multiplicative semigroup generated by $a$, $G_2$ the multiplicative
  semigroup generated by $b$, and $G_3$ the multiplicative semigroup
  generated by $ab$. Then
  \[\ms^{G_1}=\{S_1,S_2\},\quad \ms^{G_2}=\{S_4\},\quad \ms^{G_3}=\ms.\]
\end{exmp}

To simplify notation, we write $\ms\simeq\mt$ instead of $E_\ms\simeq E_\mt$.
When the IFS~$\mt$ is empty or contains only one similarity~$S$, we keep the
conventions that $\ms\simeq\emptyset$ if and only if $\ms=\emptyset$ and that
$\ms\simeq\{S\}$ if and only if $\ms$ also contains only one similarity.

\begin{thm}\label{t:sublip}
  Let $\ms$ and $\mt$ be two IFSs satisfying the SSC. Then $\ms\simeq\mt$ if
  and only if $\ms^G\simeq\mt^G$ for all multiplicative sub-semigroups $G$ of
  $(0,1)$.
\end{thm}

It follows from Theorem~\ref{t:sublip} that
\begin{exmp}\label{e:infSSC}
  Let $\ms$ be an IFS satisfying the SSC with ratios $1/9$ and $4/9$, then
  $\hdim E_\ms=1/2$. Let $\ms_1=\ms$; $\ms_2$ an IFS satisfying the SSC
  with ratios $1/81$, $1/81$, $1/81$ and $4/9$; \dots; $\ms_n$ an IFS
  satisfying the SSC with ratios
  \[\underbrace{9^{-n},\dots,9^{-n}}_{3^{n-1}},4/9;\]
  and so on. Then we have
  \begin{enumerate}[\upshape(i)]
    \item $\hdim E_{\ms_1}=\hdim E_{\ms_2}=\dots=\hdim
        E_{\ms_n}=\dots=1/2$,
    \item $\sgp\ms_1\sim\sgp\ms_2\sim\dots\sim\sgp\ms_n\sim\cdots$,
    \item $\Z^+[\ms_1]=\Z^+[\ms_2]=\dots=\Z^+[\ms_n]=\dots=\Z^+[1/3]$,
  \end{enumerate}
  but $\ms_i\not\simeq\ms_j$ whenever $i\ne j$ since $\hdim
  E_{\ms_n^G}=\frac{n-1}{2n}$, where $G$ is the multiplicative semigroup
  generated by~$1/9$.
\end{exmp}

Using the same idea, we have the following more general result.
\begin{prop}\label{p:infSSC}
  Let $\ms$ be an IFS satisfying the SSC and $\hdim E_\ms=s$. Suppose that
  one of the ratios of~$\ms$, say $r$, satisfies
  \begin{itemize}
    \item $\ms^G\ne\ms$, where $G$ is the multiplicative semigroup
        generated by~$r$;
    \item there exist positive integers $\lambda_1$, $\lambda_2$, \dots,
        $\lambda_m$ such that
        \[r^{\lambda_1}+r^{\lambda_2}+\dots+r^{\lambda_m}=1.\]
  \end{itemize}
  Then there exist infinitely many IFS $\ms_1$, $\ms_2$, \dots satisfying
  the SSC such that for each $n\ge1$,
  \begin{enumerate}[\upshape(i)]
    \item $\hdim E_{\ms_n}=s$,
    \item $\sgp\ms_n\sim\sgp\ms$,
    \item $\Z^+[\ms_n]=\Z^+[\ms]$,
  \end{enumerate}
  but $\ms_i\not\simeq\ms_j$ whenever $i\ne j$.
\end{prop}
\begin{rem}
  Note that, if we assume the commensurable condition, then the
  Conditions~(i), (ii) and~(iii) in Proposition~\ref{p:infSSC} ensure that
  there are only one Lipschitz equivalence class in the SSC case
  (Theorem~\ref{t:SSC}), or there are only finitely many Lipschitz
  equivalence classes in the OSC case (Theorem~\ref{t:TOCELCN}). However,
  Proposition~\ref{p:infSSC} says that such finiteness result does not hold
  without the commensurable condition. In other word, the difference between
  the commensurable case and the non-commensurable case is essential.
\end{rem}

\subsection{The OSC case}\label{ssec:OSC}

If the SSC does not hold, the situation is much more complicated. Unlike the
SSC case, generally, the geometrical structure depends not only on the
algebraic properties of the ratios, but also on the relative positions of the
small copies of self-similar sets due to the occurrence of the overlaps. In
fact, we know \emph{very little} about the geometrical structure of
self-similar sets with overlaps. This is a fundamental but extremely
difficult problem in fractal geometry. Here we only discuss some known
results about Lipschitz equivalence in the OSC case.

Wen and Xi~\cite{WenXi03} studied the self-similar arcs, a kind of connected
self-similar sets satisfying the OSC. They constructed two self-similar arcs
of the same Hausdorff dimension, which are not Lipschitz equivalent. This
means that the Hausdorff dimension is not enough to determine the Lipschitz
equivalence in this case. In general, more Lipschitz invariants other than
Hausdorff dimension remain unknown for the self-similar sets with non-trivial
connected component. In fact, we still have no efficient method to
investigate such self-similar sets.

In the OSC case, we do know more if the self-similar sets satisfy the
\emph{totally disconnectedness condition} (TDC). One reason is that the
geometrical structure of the totally disconnected self-similar sets
satisfying the OSC is similar to that of self-similar sets satisfying the
SSC, and so we can make use of some ideas appearing in the study of the SSC
case. Up to now all known results in the OSC and the TDC case are some
generalized versions of the $\{1,3,5\}$-$\{1,4,5\}$ problem. Let
\begin{gather*}
    E_{1,3,5}=(E_{1,3,5}/5)\cup(E_{1,3,5}/5+2/5)
    \cup(E_{1,3,5}/5+4/5),\\
    E_{1,4,5}=(E_{1,4,5}/5)\cup(E_{1,4,5}/5+3/5)
    \cup(E_{1,4,5}/5+4/5).
\end{gather*}
\begin{figure}[h]
  \centering
  \includegraphics{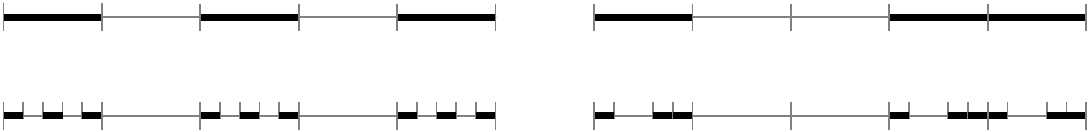}\\
  \caption{\{1,3,5\}-\{1,4,5\} problem posed by David and Semmes}
  \label{fig:135145}
\end{figure}%
The two sets are called $\{1,3,5\}$-set and $\{1,4,5\}$-set, respectively
(see Figure~\ref{fig:135145}). David and Semmes~\cite{DavSe97} asked whether
$E_{1,3,5}$ and $E_{1,4,5}$ are Lipschitz equivalent, and the question is
called the $\{1,3,5\}$-$\{1,4,5\}$ problem. Rao, Ruan and Xi~\cite{RaRuX06}
gave an affirmative answer to this problem by the method of using
graph-directed system (Definition~\ref{d:graphdir}) to investigate the
self-similar sets. So far all further developments depend on this method more
or less.

\begin{figure}[h]
  \centering
  \includegraphics{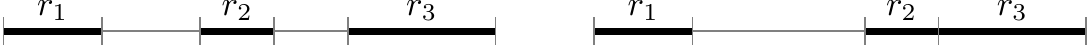}\\
  \caption{\{1,3,5\}-\{1,4,5\} problem with different ratios}
  \label{fig:135145g}
\end{figure}

Xi and Ruan~\cite{XiRu07} studied generalized $\{1,4,5\}$-sets in the line
(see Figure~\ref{fig:135145g}). This is a version of \{1,3,5\}-\{1,4,5\}
problem with different ratios. Given $r_1,r_2,r_3\in(0,1)$ with
$r_1+r_2+r_3<1$, let $\ms=\{S_1,S_2,S_3\}$, where
\[S_1\colon x\mapsto r_1x,\quad S_2\colon x\mapsto r_2x+(1-r_2-r_3),\quad
S_3\colon x\mapsto r_3x+(1-r_3).\]
Let $\mt$ be an IFS satisfying the SSC with ratios~$r_1$, $r_2$ and $r_3$.
They showed that
\[E_\ms\simeq E_\mt\Longleftrightarrow\log r_1/\log r_3\in\Q.\]
Recently, Ruan, Wang and Xi~\cite{RuWaX12} further study this problem for
IFSs containing more than three similarities. Although the IFSs studied
by~\cite{RuWaX12,XiRu07} are allowed to have non-commensurable ratios, their
settings, which only consider IFSs on~$\R^1$ and require an open interval to
satisfy the OSC and some other additional conditions, are very special. The
method of~\cite{RuWaX12,XiRu07}, depending heavily on the special settings,
sheds no light on how to settle the problem for the non-commensurable case in
general. Under the assumption that the ratios are commensurable,
Theorem~\ref{t:PI} and~\ref{t:covopen} extend their results in a very general
setting for IFSs on~$\R^d$ ($d\ge1$).

\begin{figure}[h]
  \centering
  \includegraphics{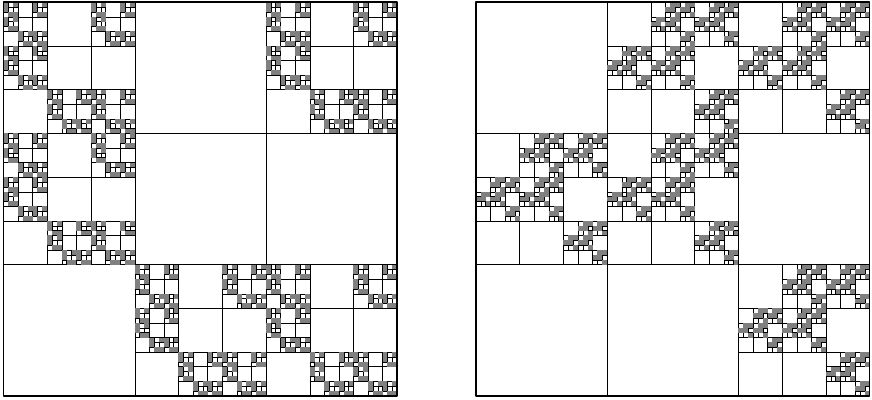}\\
  \caption{\{1,3,5\}-\{1,4,5\} problem in higher dimensional spaces}
  \label{fig:135145h}
\end{figure}

The authors~\cite{XiXi10} consider the \{1,3,5\}-\{1,4,5\} problem in~$\R^d$
(see Figure~\ref{fig:135145h}) and showed that if the two self-similar sets
\[E_A=\bigcup_{a\in A}N^{-1}(E_A+a),\quad E_B=\bigcup_{b\in B}N^{-1}(E_B+b)\]
are totally disconnected, where $A,B\subset\{0,\dots,N-1\}^d$, then
$E_A\simeq E_B$ if and only if $\card A=\card B$. Recently, Luo and
Lau~\cite{LuoLa12} and Deng and He~\cite{DenHe12} also studied the IFSs with
equal ratios in more general setting than~\cite{RaRuX06,XiXi10} and proved
some special cases of Theorem~\ref{t:srt}.

\begin{figure}[h]
  \centering
  \includegraphics{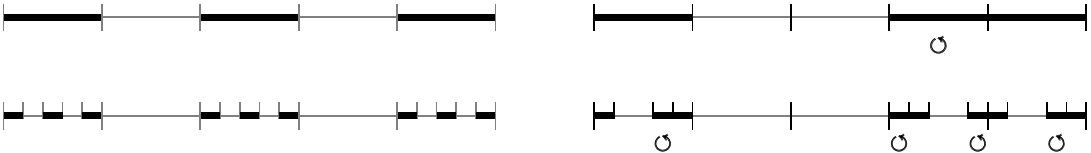}\\
  \caption{\{1,3,5\}-\{1,4,5\} problem with rotation}
  \label{fig:135145r}
\end{figure}

The authors~\cite{XioXi12} also proved a rotation version of
\{1,3,5\}-\{1,4,5\} problem (see Figure~\ref{fig:135145r}). Let $S_1\colon
x\mapsto x/5$, $S_2\colon x\mapsto (-x+4)/5$ and $S_3\colon x\mapsto
(x+4)/5$. The self-similar set $E_{1,-4,5}=\bigcup_{i=1}^3S_i(E_{1,-4,5})$ is
called the $\{1,-4,5\}$-set. Then $E_{1,-4,5}\simeq E_{1,4,5}\simeq
E_{1,3,5}$. (In Figure~\ref{fig:135145r}, the symbol~$\bm\circlearrowleft$
means that there is a minus sign in the contraction coefficient of the
corresponding similarity. In geometry, this means a rotation by the
angle~$\pi$.)

\smallskip

As we see, the method of using graph-directed system can deal with various
versions of the $\{1,3,5\}$-$\{1,4,5\}$ problem. But this method cannot give
a general result for the Lipschitz equivalence problem of self-similar sets
since it is in general very hard or even impossible to find a suitable
graph-directed system for a given family of self-similar sets. In other word,
only very special self-similar sets can be studied by using of graph-directed
system.

In this paper, we introduce the blocks to study the self-similar sets and
replace the graph-directed system by the interior blocks (see
Section~\ref{ssec:BDdef}). This new and powerful method leads to deeper
insights into geometrical structure of self-similar sets than the method of
the graph-directed system. Consequently, we are able to generalize almost all
of known results in the OSC and the TDC case. In fact, all the results
in~\cite{DenHe12,LuoLa12,RaRuX06,XiXi10,XioXi12} are very special cases of
Theorem~\ref{t:srt}, which is only a very special corollary of
Theorem~\ref{t:TOCELCN}. While Theorem~\ref{t:PI} and~\ref{t:covopen} also
generalize the results in~\cite{RuWaX12,XiRu07} under the commensurable case.
More important, we think this new method is also useful for the further study
on Lipschitz equivalence and other related problems.

\section{The Algebraic Properties of Measure Root}
\label{sec:Zp}

This section concerns the algebraic properties of measure root. As a result,
we give the proof of Proposition~\ref{p:p,r}.

The following lemma is the collection of some algebraic properties of measure
root. These properties may be known, we include the proof only for the
self-containedness since we don't find appropriate references (note that
$\Z[p]$ is in general not a Dedekind domain).
\begin{lem}\label{l:Z[p]}
  Let $p\in(0,1)$. Suppose that there exist positive integers
  $\lambda_1,\lambda_2,\dots,\lambda_N$ with
  $\gcd(\lambda_1,\dots,\lambda_N)=1$ such that
  \[p^{\lambda_1}+p^{\lambda_2}+\dots+p^{\lambda_N}=1.\]
  Then we have the following conclusions.
\begin{enumerate}[\upshape(a)]
  \item $p^{-1}$ is an algebraic integer and $p^{-1}\in\Z[p]$.
  \item The quotient ring $\Z[p]/I$ is finite for every nonzero ideal~$I$.
  \item For each nonzero ideal~$I$ of~$\Z[p]$, there exists a positive
      integer~$\ell$ such that $1-p^\ell\in I$.
  \item $\Z[p]$ is noetherian, i.e., every ideal is finitely generated.
  \item For each positive number $a\in\Z[p]$, there exists a polynomial~$P$
      with positive integer coefficients such that $a=P(p)$. In other
      words,
      \[\{a>0\colon a\in\Z[p]\}=\Z^+[p],\]
      where $\Z^+[p]$ is defined by~\eqref{eq:Z+[a]}.
  \item Let $a_1,\dots,a_m$ be positive numbers and $I=(a_1,\dots,a_m)$ the
      ideal of~$\Z[p]$ generated by $a_1,\dots,a_m$. Then, for each
      positive number $a\in I$, there exist positive numbers
      $b_1,\dots,b_m\in\Z[p]$ such that
      \[a=a_1b_1+a_2b_2+\dots+a_mb_m.\]
  \item $h(\Z[p])\le h(\Z[p^{-1}])$.
  \item $h(\xc O_p)\le h(\Z[p^{-1}])$, where $\xc O_p$ denotes the ring of
      all algebraic integers in the field~$\Q(p)$.
\end{enumerate}
\end{lem}
We remark that the inequalities in Lemma~\ref{l:Z[p]}(g) and~(h) may be
strict, see Example~\ref{e:clane1} and~\ref{e:clane2}. We need the following
fact.

\begin{fct}\label{f:square}
  The class number $h(\Z[\sqrt{n}\,])>1$ if the nonzero integer $n$ is not
  square free (i.e., $m^2\mid n$ for some integer $m>1$). For this, one can
  check that the ideal $(m,\sqrt n)$ is not a principle ideal, where $m$ is a
  prime number such that $m^2\mid n$.
\end{fct}

\begin{exmp}\label{e:clane1}
  Let $p=(\sqrt{10}-3)/2$ be the solution of $4p^2+12p=1$. Then
  \[h(\Z[p])=h(\Z[\sqrt{10},\tfrac12])=1<h(\Z[p^{-1}])=h(\Z[2\sqrt{10}]).\]
  To see that $h(\Z[\sqrt{10},\tfrac12])=1$, observe that the mapping
  \[\pi\colon I\to I^*=\{2^{-\ell}\alpha\colon\alpha\in I,\ell\ge0\}\]
  from the nonzero ideal~$I$ of~$\Z[\sqrt{10}]$ to the nonzero ideal~$I^*$
  of~$\Z[\sqrt{10},\tfrac12]$ is a surjection, and that $I=a J$ implies
  $I^*=a J^*$. Then since a nonzero ideal~$I$ of~$\Z[\sqrt{10}]$ is either a
  principle ideal or belongs to the same ideal class of~$(2,\sqrt{10})$, it
  follows from $\pi(2,\sqrt{10})=\Z[\sqrt{10},\tfrac12]$ that
  $h(\Z[\sqrt{10},\tfrac12])=1$.
\end{exmp}

\begin{exmp}\label{e:clane2}
  Let $p=5\sqrt2-7$ be the solution of $p^2+14p=1$. Then
  \[h(\xc O_p)=h(\Z[\sqrt2])=1<h(\Z[p^{-1}])=h(\Z[5\sqrt2]).\]
\end{exmp}

We remark that in Example~\ref{e:clane1}, $h(\Z[p])=1<h(\xc
O_p)=h(\Z[\sqrt{10}])=2$, while in Example~\ref{e:clane2}, $h(\xc
O_p)=1<h(\Z[p])=h(\Z[5\sqrt2])$.

\medskip

The remainder of this section is devoted to the proof of Lemma~\ref{l:Z[p]}
and Proposition~\ref{p:p,r}. We begin with a technical lemma.

\begin{lem}\label{l:primitive}
  Let
  \[\bm\Xi=\begin{pmatrix}
    \xi_1 & 1 & 0 & \hdotsfor4 & 0 \\
    \xi_2 & 0 & 1 & 0 & \hdotsfor3 & 0 \\
    \xi_3 & 0 & 0 & 1 & 0 & \hdotsfor2 & 0 \\
    \vdots & \vdots & \vdots & \vdots & \vdots & \vdots & \vdots & \vdots\\
    \vdots & \vdots & \vdots & \vdots & \vdots & \vdots & \vdots & \vdots\\
    \xi_{n-2} & 0 & \hdotsfor3 & 0 & 1 & 0 \\
    \xi_{n-1} & 0 & \hdotsfor4 & 0 & 1 \\
    \xi_n & 0 & \hdotsfor5 & 0 \\
  \end{pmatrix}\]
  be an $n\times n$ matrix, where $\xi_1,\xi_2,\dots,\xi_n$ are nonnegative
  integers. Let $\lambda_1,\dots,\lambda_m$ be all the indexes such that
  $\xi_{\lambda_i}>0$. If $\gcd(\lambda_1,\dots,\lambda_m)=1$ and
  $n\in\{\lambda_1,\dots,\lambda_m\}$, i.e., $\xi_n>0$, then the
  matrix~$\bm\Xi$ is primitive. Moreover, let $p$ be the unique positive
  solution of the equation $\xi_1p+\xi_2p^2+\dots+\xi_np^n=1$ and $\bm
  p=(1,p,\dots,p^n)$, then
  \[\bm p\bm\Xi=p^{-1}\bm p.\]
  In other word, the value $p^{-1}$ is the Perron-Frobenius eigenvalue
  of~$\bm\Xi$ and the vector~$\bm p$ is the corresponding left-hand
  Perron-Frobenius eigenvector.
\end{lem}

In what follows, $\bm A\ge\bm B$ means that each $a_{ij}\ge b_{ij}$ and $\bm
A>\bm B$ that each $a_{ij}>b_{ij}$ for arbitrary matrices $\bm A=(a_{ij})$
and $\bm B=(b_{ij})$.
\begin{proof}
  The equality $\bm p\bm\Xi=p^{-1}\bm p$ is obvious. It remains to show that
  the matrix~$\bm\Xi$ is primitive. Let $\bm A_{ij}$ be the $n\times n$
  matrix such that the $(i,j)$-entry of $\bm A_{ij}$ is~$1$ and all other
  entries are zero. Let $\bm B=(b_{ij})$ be the $n\times n$ matrix as
  \[b_{ij}=\begin{cases}
    1, & i+1 \equiv j \pmod n;\\
    0, & \text{otherwise}.
  \end{cases}\]
  It follows from the meanings of~$\lambda_1,\dots,\lambda_m$ that
  \[\bm\Xi\ge\bm B+\sum_{\lambda_i\ne n}\bm A_{\lambda_i1}.\]
  Since $\gcd\{\lambda_1,\lambda_2,\dots,\lambda_m\}=1$ and
  $n\in\{\lambda_1,\lambda_2,\dots,\lambda_m\}$, there exist positive
  integers $l_i$ and $l$ such that
  \[\sum_{\lambda_i\ne n}l_i\cdot\lambda_i=ln+1.\]
  Observe that $\bm B^{k-1} \bm A_{k1}=\bm A_{11}$ and $\bm B^n$ is the
  identity matrix. We have
  \[\biggl(\bm B +\sum_{\lambda_i\ne n}\bm A_{\lambda_i1}\biggr)
   ^{ln+1}\ge \bm B^{ln+1}+\prod_{\lambda_i\ne n}(\bm B^{\lambda_i-1}
   \bm A_{\lambda_i1})^{l_i}=\bm B+\bm A_{11}.\]
  Finally, a straightforward computation reveals that $(\bm B+\bm
  A_{11})^{2n-3}>\bm0$.
\end{proof}

The following lemma is a well-known property of the primitive matrix.
\begin{lem}\label{l:limMatrix}
  Let $\bm\Xi$, $p$ and $\bm p$ be as in Lemma~\ref{l:primitive}. Suppose
  that $\bm q$ is the right-hand Perron-Frobenius eigenvector of~$\bm\Xi$ such
  that $\bm p\cdot\bm q=1$. Then
  \[\lim_{k\to\infty}p^k\bm\Xi^k=\bm q\cdot\bm p.\]
\end{lem}

Now we are able to prove Lemma~\ref{l:Z[p]} and Proposition~\ref{p:p,r}.
\begin{proof}[Proof of Lemma~\ref{l:Z[p]}]
  (a) It is obvious.

  (b) First observe that $\Z[p]/(m)$ is finite for every nonzero integer~$m$.
  It remains to show that each nonzero ideal~$I$ contains a nonzero integer.
  Pick a nonzero number $a\in I$. By (a), for $\ell$ large enough,
  $p^{-\ell}a\in I$ is a algebraic integer. Thus for a fixed such $\ell$, we
  can find a polynomial~$P$ with integer coefficients such that
  $P(p^{-\ell}a)\in I$ is a nonzero integer.

  (c) By (b), we can find two integers $\ell_2>\ell_1>0$ with
  $p^{\ell_1}-p^{\ell_2}\in I$. We have $1-p^{\ell_2-\ell_1}\in I$ since
  $p^{-1}\in\Z[p]$ by (a).

  (d) Suppose on the contrary that $I$ is an ideal that is not finitely
  generated. Then the quotient group $I/(a)$ is infinite for all $a\in I$,
  which contradicts (b) since $I/(a)\subset\Z[p]/(a)$.

  (e) Suppose that
  \[p^{\lambda_1}+p^{\lambda_2}+\dots+p^{\lambda_N}=
  \xi_1p+\xi_2p^2+\dots+\xi_np^n,\]
  where $\xi_n>0$. Let $\bm\Xi$ be the matrix as in Lemma~\ref{l:primitive}.
  Since $\gcd(\lambda_1,\dots,\lambda_N)=1$, the conditions of
  Lemma~\ref{l:primitive} are fulfilled. Let $\bm p$ and $\bm q$ be the
  Perron-Frobenius eigenvectors as in Lemma~\ref{l:limMatrix}. Since
  $a\in\Z[p]$ is positive, there exist $\ell\ge0$ and a column vector $\bm
  a=(a_1,a_2,\dots,a_n)^T$ with integer entries such that $a=p^\ell\bm
  p\cdot\bm a>0$. Recall that $p^{-1}$ and $\bm p$ are the eigenvalue and the
  eigenvector of the matrix~$\bm\Xi$, respectively. And so $a=p^{\ell+k}\bm
  p\bm\Xi^k\bm a$ for all $k\ge0$. By Lemma~\ref{l:limMatrix},
  \[p^k\bm\Xi^k\bm a\to\bm q\cdot(\bm p\cdot\bm a)>\bm0\quad
  \text{as $k\to\infty$}.\]
  This implies that $\bm\Xi^k\bm a>\bm0$ for sufficiently large~$k$. Thus,
  Conclusion~(e) follows.

  (f) We prove this by induction on~$m$. The case $m=1$ is obvious. Now
  suppose this is true for $m-1$, let $a\in(a_1,\dots,a_m)$ be a
  positive number. We have $a=a_1b'_1+\dots+a_mb'_m$ for some
  $b'_1,\dots,b'_m\in\Z[p]$. Suppose without loss of generality that
  $b'_m>0$. By (c), we can find a positive integer~$\ell$ such that
  $1-p^\ell\in(a_1,\dots,a_{m-1})$. Pick $k$ large enough such that
  $a-a_mb'_mp^{k\ell}>0$. The proof is completed by the induction
  assumption since
  \[0<a-a_mb'_mp^{k\ell}=a-a_mb'_m+a_mb'_m(1-p^{k\ell})
  \in(a_1,\dots,a_{m-1}).\]

  (g) For each nonzero ideal $I$ of~$\Z[p]$, write $I^*=I\cap\Z[p^{-1}]$,
  then $I^*$ is a nonzero ideal of~$\Z[p^{-1}]$. It suffices to show the fact
  that if $I^*=a J^*$ for some $a\in\R$, then $I=a J$, where $I$ and $J$ are
  two nonzero ideals of~$\Z[p]$. Indeed, for each $b\in J$, there exists an
  integer~$\ell$ with $bp^{-\ell}\in J^*$, so $abp^{-\ell}\in aJ^*=I^*$. Thus
  $ab\in I$, i.e., $aJ\subset I$. By symmetric, $a^{-1}I\subset J$ and so
  $I=aJ$.

  (h) Recall that $p^{-1}$ is an algebraic integer and that
  $\Q(p)=\Q(p^{-1})$. Together with the fact that $\xc O_p$ is a finitely
  generated $\Z$-module, we know that there exists a positive integer $m$
  such that $m\xc O_p\subset\Z[p^{-1}]$. For each nonzero ideal $I$ of~$\xc
  O_p$, write $I^*=mI$, then $I^*$ is a nonzero ideal of~$\Z[p^{-1}]$. It is
  obviously that $aI^*=J^*$ if and only if $aI=J$, where $I$ and $J$ are
  two nonzero ideals of~$\xc O_p$. Therefore, $h(\xc O_p)\le h(\Z[p^{-1}])$.
\end{proof}

\begin{proof}[Proof of Proposition~\ref{p:p,r}]
  Suppose first that there is an IFS $\ms\in\TOEC(p,r)$. By the meanings
  of~$p,r$, we can assume that the ratios of~$\ms$ are
  $r^{\lambda_1},r^{\lambda_2},\dots,r^{\lambda_N}$, where
  $\gcd(\lambda_1,\dots,\lambda_N)=1$. Then we have
  \[p^{\lambda_1}+p^{\lambda_2}+\dots+p^{\lambda_N}=1.\]
  The conclusion that $p,r\in(0,1)$ is obvious.

  Conversely, fix an integer $\ell>0$ such that $r^{\ell}<1/2$ and
  $1-p^\ell-p^{\ell+1}>0$. By Lemma~\ref{l:Z[p]}(e), there exist nonnegative
  integers $\ell_1,\dots,\ell_m$ such that
  \[p^{-\ell}(1-p^\ell-p^{\ell+1})=p^{\ell_1}+p^{\ell_2}+\dots+p^{\ell_m}.\]
  Let $\ms$ be an IFS satisfying the OSC and consisting of $m+2$ similarities
  with ratios $r^\ell$, $r^{\ell+1}$,
  $r^{\ell+\ell_1},r^{\ell+\ell_2},\dots,r^{\ell+\ell_m}$. Since all the
  ratios are less than~$1/2$, such IFS does exist on~$\R^d$ with $2^d\ge m+2$.
  For example, let $\xc S=\{S_1,\dots,S_{m+2}\}$ with
  \begin{align*}
    S_1&\colon x\mapsto r^\ell x+(\underbrace{0,\dots,0,0}_d)/2,&
    S_2&\colon x\mapsto r^{\ell+1}x+(\underbrace{0,\dots,0,1}_d)/2, \\
    S_3&\colon x\mapsto r^{\ell+\ell_1}x+(\underbrace{0,\dots,1,0}_d)/2,&
    S_4&\colon x\mapsto r^{\ell+\ell_2}x+(\underbrace{0,\dots,1,1}_d)/2,\\
    &\dots\dots&\dots\dots
  \end{align*}
  Then $\ms$ satisfies the OSC with the open set $(0,1)^d$. Also note that
  $\gcd(\ell,\ell+1)=1$, so we have $r_\ms=r$ and $p_\ms=p$. Therefore,
  $\ms\in\TOEC(p,r)\ne\emptyset$.
\end{proof}

\section{The Ideal of IFS}\label{sec:ideal}

This section is devoted to the proofs of Theorem~\ref{t:TOCELCN} and
Theorem~\ref{t:covopen}, which are closely related to the problem of
determining the ideal of IFS in $\TOEC$. The difficult is that there is no
general method to determine such ideals. For our purpose, we consider the
problem in two special cases: the self-similar set has the graph-directed
structure and the self-similar set generated by IFSs in~$\xs S$, where $\xs
S$ is defined by~\eqref{eq:PIS}.

\subsection{The graph-directed structure}

The key point of the proof of Theorem~\ref{t:TOCELCN} is the following
theorem.
\begin{thm}\label{t:ideal}
  Suppose that $\TOEC(p,r)\ne\emptyset$. Then for each nonzero ideal~$I$ of
  the ring~$\Z[p]$, there exists an IFS $\ms\in\TOEC(p,r)$ such that
  $I_\ms=I$.
\end{thm}
We make use of the graph-directed sets to prove Theorem~\ref{t:ideal}. For
convenience, we recall the definition of graph-directed sets
(see~\cite{MauWi88}).
\begin{defn}[graph-directed sets]\label{d:graphdir}
  Let $G=(\xc V,\xc E)$ be a directed graph with vertex set $\xc V$ and
  directed-edge set~$\xc E$. Suppose that for each edge $e\in\xc E$, there is
  a corresponding similarity $S_e\colon\R^d\to\R^d$ of ratio~$r_e\in(0,1)$.

  The graph-directed sets on~$G$ with the similarities $\{S_e\}_{e\in\xc E}$
  are defined to be the unique nonempty compact sets $\{E_i\}_{i\in\xc V}$
  satisfying
\begin{equation}\label{eq:graphdir}
  E_i=\bigcup_{j\in\xc V}\bigcup_{e\in\xc E_{i,j}}S_e(E_j)
  \qquad\text{for $i\in\xc V$},
\end{equation}
  where $\xc E_{i,j}$ is the set of edges staring at~$i$ and ending at~$j$.
  In particular, if \eqref{eq:graphdir}~is a disjoint union for each $i\in\xc
  V$, we call $\{E_i\}_{i\in\xc V}$ are \emph{dust-like} graph-directed sets
  on~$(\xc V,\xc E)$.
\end{defn}
If the self-similar set $E$ has the graph-directed structure, we can
determine its ideal easily. Let $\ms\in\TOEC(p,r)$. Suppose that $E_\ms$ is
one of the dust-like graph-directed sets $\{E_i\}_{i\in\xc V}$ on $G=(\xc
V,\xc E)$, and that for all $e\in\xc E$, $\log r_e/\log r\in\N$. Without loss
of generality, we also suppose that $\xc V=\{0,1,\dots,n\}$ and $E_\ms=E_0$.
Let $\xc E_{i,j}^k$ denote the set of sequences of $k$ edges
$(e_1,e_2,\dots,e_k)$ which form a directed path from vertex~$i$ to
vertex~$j$. Let $O$ be an open set of~$\ms$ satisfying the SOSC. We use $\xc
V_O$ to denote the set of vertexes~$i$ such that there exists
$(e_1,e_2,\dots,e_k)\in\xc E_{0,i}^k$ for some $k\ge1$ satisfying
\[S_{e_1}\circ S_{e_2}\circ\dots\circ S_{e_k}(E_i)\subset O.\]
\begin{thm}\label{t:gdid}
  The ideal $I_\ms$ is generated by $\{\xc H^s(E_i)/\xc H^s(E_\ms)\colon
  i\in\xc V_O\}$, where $s=\hdim E_\ms$.
\end{thm}

The proof of Theorem~\ref{t:gdid} will be given in Section~\ref{ssec:BDnum}
since it requires a basic fact about the ideal of IFS (Remark~\ref{r:ideal}).
We give an example here.
\begin{exmp}\label{e:gdid}
  Let $\ms$ be as in Example~\ref{e:NPI}. Let $E_0=E_\ms$, $E_1=-rE_\ms\cup
  E_\ms$ and $E_2=-E_\ms+E_\ms$. It is easy to check that $E_0$, $E_1$ and
  $E_2$ forms a family of graph-directed sets. Let $O=(0,1)$, then $\xc
  V_O=\{1,2\}$. By Theorem~\ref{t:gdid},
  \[I_\ms=(\xc H^s(E_1)/\xc H^s(E_0),\xc H^s(E_2)/\xc H^s(E_0))=(p+1,2).\]
\end{exmp}

By Theorem~\ref{t:gdid}, we are able to give the proof of
Theorem~\ref{t:ideal}.

\begin{proof}[Proof of Theorem~\ref{t:ideal}]
Fix a nonzero ideal $I$ of $\Z[p]$. By Lemma~\ref{l:Z[p]}(c), there exists
a positive integer~$\ell$ such that $1-p^\ell\in I$. We can further require
that $r^\ell<1/6$ since $1-p^{k\ell}\in I$ for all $k\ge1$. By
Lemma~\ref{l:Z[p]}(d), we can choose positive numbers
$a_1,a_2,\dots,a_m\in\Z[p]$ such that $I=(a_1,\dots,a_m)$. By
Lemma~\ref{l:Z[p]}(f), there exist positive numbers $b_1,\dots,b_m\in I$ such
that
\begin{equation}\label{eq:pk}
  1-p^\ell=a_1b_1+\dots+a_mb_m.
\end{equation}
By Lemma~\ref{l:Z[p]}(e), for $1\le i\le m$, there exist nonnegative integers
$u_{i,j}$ ($1\le j\le N_i$) and positive integers $v_{i,j}$ with
$r^{v_{i,j}}<1/6$ ($1\le j\le M_i$) such that
\begin{equation}\label{eq:aibi}
\begin{aligned}
  a_i&=p^{u_{i,1}}+p^{u_{i,2}}+\dots+p^{u_{i,N_i}}, \\
  b_i&=p^{v_{i,1}}+p^{v_{i,2}}+\dots+p^{v_{i,M_i}}.
\end{aligned}
\end{equation}
We can further require that
\begin{equation}\label{eq:u,u+1}
  u_{1,1}+1=u_{1,2}
\end{equation}
since there exists a nonnegative integer $u$ such that $a_1-p^u-p^{u+1}>0$,
then set $u_{1,1}=u$, $u_{2,1}=u+1$ and apply Lemma~\ref{l:Z[p]}(e) to
$a_1-p^u-p^{u+1}$. Finally, choose a positive integer $d$ such that
\begin{equation}\label{eq:d}
  2^d\ge\max(N_1,N_2,\dots,N_m)\quad\text{and}\quad
  2^d\ge M_1+M_2+\dots+M_m.
\end{equation}

Now we are ready to construct the desired IFS~$\ms$. In the remainder of this
proof, we use $\bm x$ and~$\bm y$ to denote the points in~$\R^d$. For
$\Lambda\subset\{1,2,\dots,d\}$, define an isometric mapping
$T_\Lambda\colon\R^d\to\R^d$ by
\begin{equation}\label{eq:T}
  (T_\Lambda\bm x)_i=\begin{cases}
  x_i,& i\notin\Lambda,\\
  -x_i, & i\in\Lambda,
  \end{cases}\quad
  \text{for $\bm x=(x_1,\dots,x_d)\in\R^d$}.
\end{equation}
Since $2^d\ge\max(N_1,\dots,N_m)$, for each $i\in\{1,2,\dots,m\}$, we can
choose distinct
\[\Lambda_{i,1},\Lambda_{i,2}\dots,\Lambda_{i,N_i}\subset\{1,\dots,d\},\]
and then define an IFS $\mt_i$ as
\begin{equation}\label{eq:Ti}
  \mt_i=\bigl\{r^{u_{i,1}}T_{\Lambda_{i,1}},r^{u_{i,2}}T_{\Lambda_{i,2}},
  \dots,r^{u_{i,N_i}}T_{\Lambda_{i,N_i}}\bigr\}.
\end{equation}
Let
\[Y=\bigl\{\bm y=(y_1,\dots,y_d)\in\R^d\colon\text{$y_i=1/3$ or $2/3$ for
$i=1,\dots,d$}\}.\]
Since $2^d\ge M_1+\dots+M_m$, we can choose distinct points $\bm y_{i,j}\in
Y$ for $1\le i\le m$ and $1\le j\le M_i$. Define a contracting similarity
$S_0\colon\bm x\mapsto r^\ell\bm x$ on~$\R^d$ and IFSs
\begin{equation}\label{eq:Sij}
  \ms_{i,j}=\bigl\{r^{v_{i,j}}T+\bm y_{i,j}\colon T\in\mt_i\bigr\}
\end{equation}
for $1\le i\le m$, $1\le j\le M_i$. Finally, define
\[\ms=\{S_0\}\cup\bigcup_{i=1}^{m}\bigcup_{j=1}^{M_i}\ms_{i,j}.\]

It remains to show that $\ms\in\TOEC(p,r)$ and $I_\ms=I$.

We first prove that $\ms\in\OEC(p,r)$. Note that the ratios of $\ms$ are
\[r^\ell\quad\text{and}\quad r^{u_{i,j}+v_{i,j'}}\
(1\le i\le m,1\le j\le N_i,1\le j'\le M_i).\]
By~\eqref{eq:u,u+1}, we have $r_\ms=r$. By~\eqref{eq:pk} and~\eqref{eq:aibi},
we have $p_\ms=p$. We will show that $\ms$ satisfies the OSC for the open
set~$(0,1)^d$. Note that the ratios of~$\ms$ are all less than $1/6$ since
$r^\ell<1/6$, $r^{v_{i,j}}<1/6$ and $u_{i,j}\ge0$. And so
$S_0(0,1)^d\subset(0,1/6)^d$; $S(0,1)^d\subset(-1/6,1/6)^d+\bm y_{i,j}$ for
$S\in\ms_{i,j}$. Therefore, $S(0,1)^d\subset(0,1)^d$ for all $S\in\ms$. On
the other hand, for distinct $S,S'\in\ms$, we need to show $S(0,1)^d\cap
S'(0,1)^d=\emptyset$. There are three cases to consider.
\begin{description}
  \item[Case~1] One of $S,S'$ is $S_0$. Then there exists some $\bm y\in Y$
      such that
      \[S(0,1)^d\cap S'(0,1)^d\subset(0,1/6)^d\cap
      \bigl((-1/6,1/6)^d+\bm y\bigr)=\emptyset.\]
  \item[Case~2] $S\in\ms_{i,j}$ and $S'\in\ms_{i',j'}$ with
      $(i,j)\ne(i',j')$. Then the corresponding $\bm y_{i,j},\bm
      y_{i',j'}\in Y$ are distinct. And so
      \[S(0,1)^d\cap S'(0,1)^d\subset\bigl((-1/6,1/6)^d+\bm y_{i,j}\bigr)
      \cap\bigl((-1/6,1/6)^d+\bm y_{i',j'}\bigr)=\emptyset.\]
  \item[Case~3] $S,S'\in\ms_{i,j}$. By the definition of $\ms_{i,j}$, we
      have $S(0,1)^d\cap S'(0,1)^d=\emptyset$.
\end{description}
This completes the proof of $\ms\in\OEC(p,r)$.

\smallskip

Let $E_0=E_\ms$ be the self-similar set generated by~$\ms$ and
$E_i=\bigcup_{T\in\mt_i}T(E_0)$ for $1\le i\le m$. Define $S_{i,j}\colon
x\mapsto r^{v_{i,j}}x+\bm y_{i,j}$ for $1\le i\le m$, $1\le j\le M_i$. The
proof of $\ms\in\TDC$ and $I_\ms=I$ is based on the fact that the sets
$\{E_i\}_{i=0}^m$ are \emph{dust like graph-directed sets}. Indeed, we have
\begin{equation}\label{eq:E0}
  E_0=S_0(E_0)\cup\bigcup_{i=1}^m\bigcup_{j=1}^{M_i}S_{i,j}(E_i),
\end{equation}
and for $1\le i\le m$,
\begin{multline*}
  E_i=\bigcup_{T\in\mt_i}T(E_0)=\bigcup_{T\in\mt_i}T\circ S_0(E_0)\cup
  \bigcup_{T\in\mt_i}\bigcup_{i'=1}^m\bigcup_{j=1}^{M_{i'}}
  T\circ S_{i',j}(E_{i'}) \\
  =\bigcup_{T\in\mt_i}S_0\circ T(E_0)\cup
  \bigcup_{T\in\mt_i}\bigcup_{i'=1}^m\bigcup_{j=1}^{M_{i'}}
  T\circ S_{i',j}(E_{i'})
\end{multline*}
since $T\circ S_0=S_0\circ T$ for $T\in T_i$. It follows from
$E_i=\bigcup_{T\in\mt_i}T(E_0)$ that
\begin{equation}\label{eq:Ei}
  E_i=S_0(E_i)\cup\bigcup_{T\in\mt_i}\bigcup_{i'=1}^m\bigcup_{j=1}^{M_{i'}}
  T\circ S_{i',j}(E_{i'})\quad
  \text{for $1\le i\le m$}.
\end{equation}
We will show that all the unions in~\eqref{eq:E0} and~\eqref{eq:Ei} are
disjoint. By the definition of~$\ms$, we know that $E_0\subset[0,1]^d$,
$E_0\setminus(0,1)^d=\{\bm0\}$ and $E_i\subset(-1,1)^d$ for all $1\le i\le
m$. Note that the ratios of $S_0$ and $S_{i,j}$ are $r^\ell$ and
$r^{v_{i,j}}$, all less than $1/6$. This means
\[S_0(E_0)\subset[0,1/6]^d,\quad
S_{i,j}(E_{i})\subset(-1/6,1/6)^d+\bm y_{i,j}.\]
Recall that $\bm y_{i,j}\in Y$, we have
\[\bigcup_{i=1}^m\bigcup_{j=1}^{M_i}S_{i,j}(E_i)\subset(1/6,5/6)^d,\]
and
\begin{equation}\label{eq:SijDJNT}
  S_{i,j}(E_i)\cap S_{i',j'}(E_{i'})=\emptyset\quad
  \text{when $(i,j)\ne(i',j')$},
\end{equation}
since $\bm y_{i,j}\ne\bm y_{i',j'}$. Therefore, the unions in~\eqref{eq:E0}
are disjoint. For the unions in~\eqref{eq:Ei}, observe that, for $1\le i\le
m$ and distinct $T,T'\in\mt_i$, $T(0,1)^d\cap T'[0,1]^d=\emptyset$ by the
definition of~$\mt_i$. There are two results follow from the observation. The
first is
\[\bigcup_{i'=1}^m\bigcup_{j=1}^{M_{i'}}T\circ S_{i',j}(E_{i'})\cap
\bigcup_{i'=1}^m\bigcup_{j=1}^{M_{i'}}T'\circ S_{i',j}(E_{i'})\subset
T(1/6,5/6)^d\cap T'(1/6,5/6)^d=\emptyset\]
for distinct $T,T'\in\mt_i$. The second is, for $T\in\mt_i$, $1\le i'\le m$
and $1\le j\le M_{i'}$,
\begin{multline*}
  T\circ S_{i',j}(E_{i'})\cap S_0(E_i)=T\circ S_{i',j}(E_{i'})
  \cap\bigcup_{T'\in\mt_i}T'\circ S_0(E_0) \\
  =T\circ S_{i',j}(E_{i'})\cap T\circ S_0(E_0)=\emptyset
\end{multline*}
since $S_{i',j}(E_{i'})\cap S_0(E_0)=\emptyset$. It follows from the two
results and~\eqref{eq:SijDJNT} that the unions in~\eqref{eq:Ei} are also
disjoint. Thus, we have proved that the sets $\{E_i\}_{i=0}^m$ are dust like
graph-directed sets, and so $\ms\in\TDC$ follows.

Finally, we turn to prove that $I_\ms=I$ by making use of
Theorem~\ref{t:gdid}. Let $O=(0,1)^d$, then $\xc V_O=\{1,2,\dots,m\}$. For
$1\le i\le m$, we have
\begin{equation*}
  \xc H^s(E_i)=\sum_{T\in\xc T_i}\xc H^s(T(E_0))
  =(p^{u_{i,1}}+p^{u_{i,2}}+\dots+p^{u_{i,N_i}})\xc H^s(E_0)
  =a_i\xc H^s(E_0)
\end{equation*}
by~\eqref{eq:aibi} and~\eqref{eq:Ti}. Therefore, we have
$I_\ms=(a_1,\dots,a_m)=I$.
\end{proof}

\begin{figure}[h]
  \centering
  \includegraphics{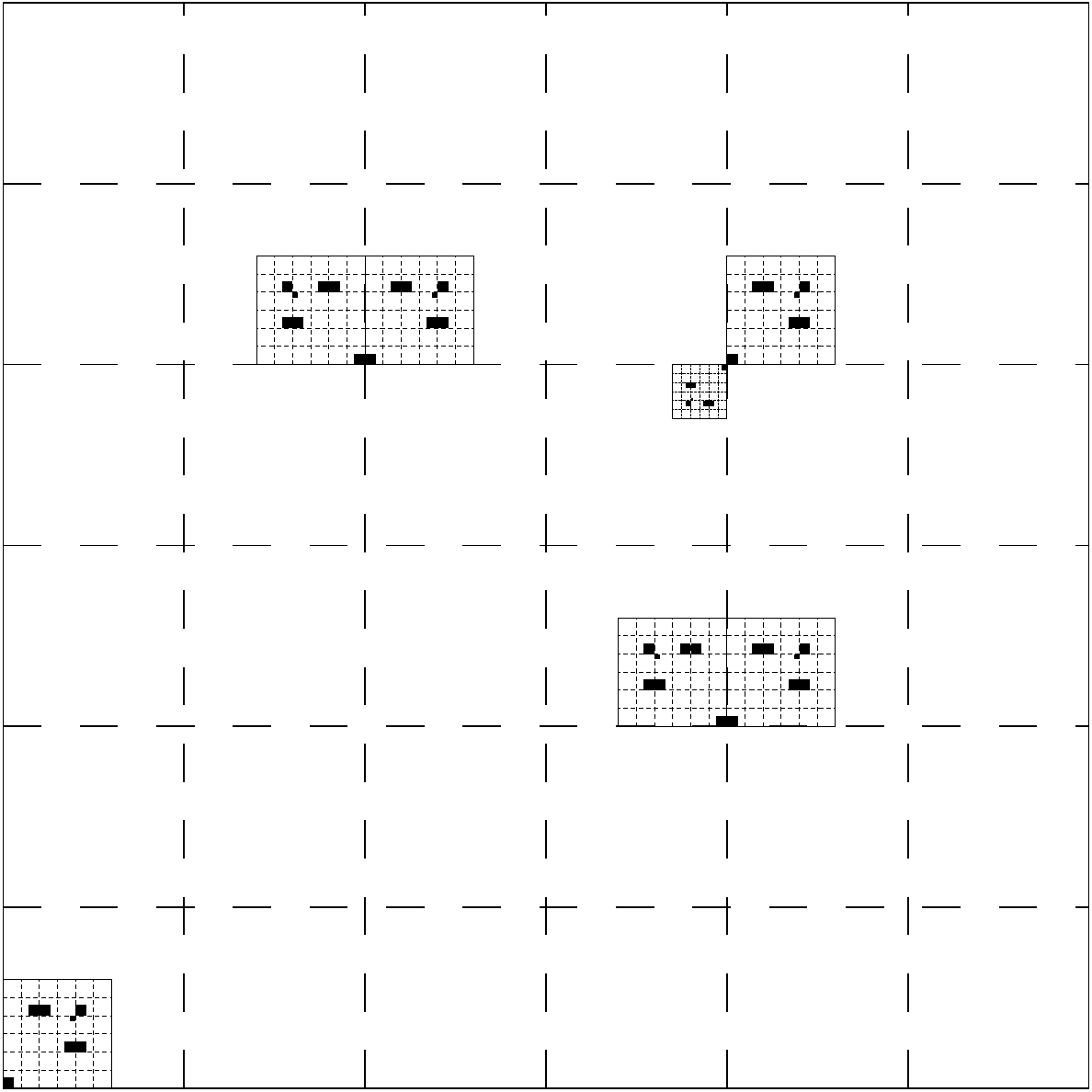}\\
  \caption{The structure of~$E_\ms$ in Example~\ref{e:ideal}}
  \label{fig:ideal}
\end{figure}

\begin{exmp}\label{e:ideal}
  Let $r=1/10$ and $p=\sqrt{10}-3$ be the positive solution of the equation
  $p^2+6p=1$. Let $I=(2,\sqrt{10})$ be an ideal of the ring
  $\Z[p]=\Z[\sqrt{10}]$. It is worth noting that $I$ is not a principle
  ideal.

  It follows from Proposition~\ref{p:p,r} that $\TOEC(p,r)\ne\emptyset$. We
  will construct an IFS $\ms\in\TOEC(p,r)$ such that $I_\ms=I$ according to
  the proof of Theorem~\ref{t:ideal}.

  Observe that $I=(2,p+1)$ and $1-p=p(p+1)+2p\cdot 2$. By~\eqref{eq:pk},
  \eqref{eq:aibi}, \eqref{eq:u,u+1} and~\eqref{eq:d}, we may set
  \[\begin{cases}
    a_1=1+p,\\ a_2=2=1+1,
  \end{cases}\quad
  \begin{cases}
    b_1=p,\\ b_2=2p=p+p,
  \end{cases}
  \begin{cases}
    \ell=1,\\ d=2.
  \end{cases}\]
  By~\eqref{eq:T}, \eqref{eq:Ti} and~\eqref{eq:Sij}, we may set
  \[\mt_1=\{T_\emptyset,r\cdot T_{\{1,2\}}\},\quad
  \mt_2=\{T_\emptyset,T_{\{1\}}\},\]
  $S_0=rT_\emptyset$ and
  \begin{gather*}
    \ms_{1,1}=\bigl\{rT_\emptyset+(2/3,2/3),r^2T_{\{1,2\}}+(2/3,2/3)\bigr\},\\
    \ms_{2,1}=\bigl\{rT_\emptyset+(2/3,1/3),rT_{\{1\}}+(2/3,1/3)\bigr\},\\
    \ms_{2,2}=\bigl\{rT_\emptyset+(1/3,2/3),rT_{\{1\}}+(1/3,2/3)\bigr\}.
  \end{gather*}
  Finally, let $\ms=\{S_0\}\cup\ms_{1,1}\cup\ms_{2,1}\cup\ms_{2,2}$, see
  Figure~\ref{fig:ideal} for the corresponding self-similar set~$E_\ms$. By
  the proof of Theorem~\ref{t:ideal}, we know that $I_\ms=I$.
\end{exmp}

We also need the following special version of Jordan-Zassenhaus Theorem (see,
e.g., \cite{CurRe62}) to prove Theorem~\ref{t:TOCELCN}.

\begin{cthm}[Jordan-Zassenhaus Theorem]
  Suppose that $\alpha$ is an algebraic integer, then the class number
  $h(\Z[\alpha])$ is finite.
\end{cthm}
We remark that $\Z[\alpha]$ is in general not a Dedekind domain, so the
conclusion on the finiteness of class number $h(\Z[\alpha])$ cannot be
derived directly by the corresponding result of Dedekind domain.

\begin{proof}[Proof of Theorem~\ref{t:TOCELCN}]
  By Theorem~\ref{t:TOCEpr} and~\ref{t:ideal}, we have $\LCN(p,r)=h(\Z[p])$
  when $\TOEC(p,r)\ne\emptyset$. Then by Lemma~\ref{l:Z[p]}(g), we have
  $\LCN(p,r)=h(\Z[p])\le h(\Z[p^{-1}])$. Finally, since $p^{-1}$ is an
  algebraic integer (Lemma~\ref{l:Z[p]}(a)), by the Jordan-Zassenhaus
  Theorem, the class number $\LCN(p,r)\le h(\Z[p^{-1}])$ is finite.
\end{proof}

\subsection{Principle ideal}

Let $\ms=\{S_1,S_2,\dots,S_N\}$ be an IFS satisfying the OSC. We write $\xs
O_\ms$ to denote all the open sets satisfying the OSC for the IFS~$\ms$ and
$\partial_\ms=E_\ms\setminus\bigcup_{O\in\xs O_\ms}O$, where $E_\ms$ is the
self-similar set generated by~$\ms$. Notice that $\partial_\ms=\emptyset$ if
and only if $\ms$ satisfies the SSC. We say that a point~$x\in\partial_\ms$
is \emph{separated} if there is a finite word $\bm i=i_1\dots
i_n\in\{1,\dots,N\}^n$ such that
\begin{equation}\label{eq:xbdr}
  S_{\bm i}(x)\notin\partial_\ms\quad\text{and}\quad
  S_{\bm i}(x)\notin S_{\bm j}(E_\ms)
\end{equation}
for every word~$\bm j$ of the same length as~$\bm i$ but $\bm j\ne\bm i$,
where $S_{\bm i}=S_{i_1}\circ\dots\circ S_{i_n}$.

We need the following theorem to prove Theorem~\ref{t:covopen}, which is also
of interesting in itself.
\begin{thm}\label{t:separated}
  Let $\ms\in\TOEC$. If the points in~$\partial_\ms$ are all separated, then
  $I_\ms=\Z[p_\ms]$.
\end{thm}
\begin{proof}
For each $x\in\partial_\ms$, let $\bm i_x$ be a word of finite length
satisfying~\eqref{eq:xbdr}. Choose a compact subset $F_x\subset E_\ms$
containing~$x$ such that $E_\ms\setminus F_x$ is also compact and $S_{\bm
i_x}(F_x)\cap S_{\bm j}(E_\ms)=\emptyset$ for every word~$\bm j$ of same
length as~$\bm i_x$ but $\bm j\ne\bm i_x$. Such $F_x$ does exist since
$E_\ms$ is totally disconnected and $\bm i_x$ satisfies~\eqref{eq:xbdr}. We
can further require $S_{\bm i_x}(F_x)$ to be an interior separated set due to
$S_{\bm i_x}(x)\notin\partial_\ms$. Note that $F_x$ is also an open subset in
the topology space~$E_\ms$. So $\{F_x\}_x$ is an open cover of the compact
set~$\partial_\ms$, thus we have a finite sub-cover $F_1$, \dots, $F_n$ and
the corresponding words $\bm i_1$, \dots, $\bm i_n$. Let
\[F^*_0=E_\ms\setminus\bigcup_{k=1}^nF_k,\quad
F^*_1=F_1,\quad F^*_2=F_2\setminus F_1,\ldots,\quad
F^*_n=F_n\setminus\bigcup_{k=1}^{n-1}F_k.\]
Then $\{F^*_k\}_{k=0}^n$ is a disjoint cover of~$E_\ms$. We claim that
$\mu_\ms(F^*_k)\in I_\ms$ for all $0\le k\le n$. Then
\[\sum_{k=0}^n\mu_\ms(F^*_k)=\mu_\ms(E_\ms)=1\in I_\ms.\]
Thus $I_\ms=\Z[p_\ms]$. It remains to prove the claim. For $1\le k\le n$,
observe that $S_{\bm i_k}(F^*_k)$ are all interior separated sets. And so
$\mu_\ms(F^*_k)\in I_\ms$ for $1\le k\le n$ since $p_\ms^{-1}\in\Z[p_\ms]$
(Lemma~\ref{l:Z[p]}(a)). For $F^*_0$, since $F^*_0$ is compact and each point
in $F^*_0$ can be covered by an interior separated set, we have $F^*_0$ is a
finite union of interior separated sets. Thus the claim $\mu_\ms(F^*_0)\in
I_\ms$ follows.
\end{proof}

\begin{proof}[Proof of Theorem~\ref{t:covopen}]
  According to Theorem~\ref{t:separated}, we only need to prove that the
  points in~$\partial_\ms$ are all separated. For this, fix a point
  $x_0\in\partial_\ms$. Let $F$ be an interior separated set. It is easy to
  see that, for $n$ large enough, there is a $\Lambda\subset\{1,\dots,N\}^n$
  such that
  \[F=\bigcup_{\bm i\in\Lambda}S_{\bm i}(E_\ms).\]
  Choose a word in $\Lambda$, say $\bm i^*$. If $S_{\bm i^*}(x_0)\notin
  S_{\bm j}(E_\ms)$ for all $\bm j\in\Lambda$ and $\bm j\ne\bm i^*$, then
  $x_0$ is separated since $F$ is an interior separated set. Thus, the proof
  is completed. Otherwise, suppose that $S_{\bm i^*}(x_0)\in S_{\bm
  j}(E_\ms)$ for some $\bm j$ other than~$\bm i^*$. Let $O$ be the convex
  open set satisfying the OSC. Since $S_{\bm i^*}(O)\cap S_{\bm j}(O)
  =\emptyset$, by convexity, there is a liner function~$H$ such that
  $H(x)<H(y)$ for all $x\in S_{\bm i^*}(O)$ and all $y\in S_{\bm j}(O)$.
  Since $S_{\bm i^*}(x_0)\in\overline{S_{\bm i^*}(O)}\cap\overline{S_{\bm
  j}(O)}$, we have
  \[H(S_{\bm i^*}(x_0))=\sup_{x\in S_{\bm i^*}(O)}H(x)
  =\max_{x\in S_{\bm i^*}(E_\ms)}H(x).\]
  Since each $S\in\ms$ has the form $S\colon x\mapsto r_S\bm Ax+b_S$ with
  $r_S\in(0,1)$, we have
  \[S_{\bm i}\colon x\mapsto r_{\bm i}\bm A^nx+b_{\bm i}\quad
  \text{for all $\bm i\in\Lambda$}.\]
  It follows that
  \[H(S_{\bm i}(x_0))=\sup_{x\in S_{\bm i}(O)}H(x)
  =\max_{x\in S_{\bm i}(E_\ms)}H(x)\quad\text{for all $\bm i\in\Lambda$}.\]
  And so $H(S_{\bm j}(x_0))>H(S_{\bm i^*}(x_0))$. Now let $\Lambda_1$ denote
  the set of all~$\bm i\in\Lambda$ such that $H(S_{\bm i}(x_0))>H(S_{\bm
  i^*}(x_0))$. We have
  \begin{itemize}
    \item $\Lambda_1\subsetneq\Lambda$ (since $\bm i^*\notin\Lambda_1$).
    \item For $\bm i\in\Lambda_1$, if $S_{\bm i}(x_0)\in S_{\bm j}(E_\ms)$
        for some $\bm j\in\Lambda$, then $\bm j\in\Lambda_1$ (since
        $H(S_{\bm i}(x_0))\le H(S_{\bm j}(x_0))$).
  \end{itemize}

  Repeat the above argument with replacing $\Lambda$ by~$\Lambda_1$, then
  either we find a word $\bm i^*\in\Lambda_1$ such that $S_{\bm
  i^*}(x_0)\notin S_{\bm j}(E_\ms)$ for all $\bm j\in\Lambda_1$ but $\bm
  j\ne\bm i^*$, this means $x_0$ is separated, or we get a
  subset~$\Lambda_2\subsetneq\Lambda_1$ such that for $\bm i\in\Lambda_2$, if
  $S_{\bm i}(x_0)\in S_{\bm j}(E_\ms)$ for some $\bm j\in\Lambda_1$, then
  $\bm j\in\Lambda_2$. If the latter case happens, then we repeat the
  argument again. The process stops when we find a desired word $\bm i^*$ to
  show that $x_0$ is separated. This completes the proof since $\Lambda$ is
  finite.
\end{proof}

\section{The Blocks Decomposition of Self-similar Sets}\label{sec:BD}

To understand the geometric structure of self-similar sets generated by IFSs
$\ms\in\TOEC$, we shall make use of the \emph{blocks decomposition}. Indeed,
the whole proof of our result is base on it.

In this section, we introduce the basic definitions of blocks decomposition
and give some important properties. From now on, fix an
$\ms=\{S_1,S_2,\dots,S_N\}\in\TOEC$. For notational convenience, we will
write $E_\ms$, $\mu_\ms$, $r_\ms$ and~$p_\ms$ as~$E$, $\mu$, $r$ and~$p$,
respectively. Let $r_i$ be the contraction ratio of~$S_i$ for $1\le i\le N$
and $s=\hdim E$. Write
\begin{equation}\label{eq:lambda}
  \lambda_i=\log r_i/\log r\quad\text{and}\quad
  \lambda=\max_{1\le i\le N}\lambda_i-1.
\end{equation}
Recall that $\gcd(\lambda_1,\dots,\lambda_N)=1$. For $\bm i=i_1i_2\dots
i_n\in\{1,2,\dots,N\}^n$, write $\bm i^-=i_1i_2\dots i_{n-1}$ and
\begin{equation}\label{eq:Siripi}
  S_{\bm i}=S_{i_1}\circ S_{i_2}\circ\dots\circ S_{i_n},
  \quad r_{\bm i}=r_{i_1}r_{i_2}\cdots r_{i_n},
  \quad p_{\bm i}=r_{\bm i}^s.
\end{equation}
Define
\begin{equation}\label{eq:Sk}
  \ms_k=\bigl\{S_{\bm i}\colon r_{\bm i}\le r^k<r_{\bm i^-}
  \bigr\}.
\end{equation}

\subsection{The definition of blocks decomposition}\label{ssec:BDdef}

\begin{defn}[level-$k$ blocks decomposition]\label{d:block}
  The decomposition $E=\bigcup_{j=1}^{n_k}B_{k,j}$ is called the level-$k$
  blocks decomposition of~$E$ ($k\ge0$), if each set
  \[\bigl\{x\colon\dist(x,B_{k,j})<r^k|E|/2\bigr\}\]
  is connected for $1\le j\le n_k$ and
  \[\dist(B_{k,i},B_{k,j})\ge r^k|E|,\quad\text{for $i\ne j$},\]
  where $|E|$ denotes the diameter of~$E$. The set $B_{k,j}$ is called
  a level-$k$ block of~$\ms$. The family of all the level-$k$ blocks will
  be denoted by~$\xs B_k$. Write $\xs B=\bigcup_{k\ge0}\xs B_k$.
\end{defn}
\begin{rem}\label{r:block}
  For $B\in\xs B_k$, write $\ms_B=\{S\in\ms_k\colon S(E)\subset B\}$.
  According to above definition, it is easy to check that
  $B=\bigcup_{S\in\ms_B}S(E)$.
\end{rem}

We shall use the natural measure~$\mu$ to describe the size of blocks.
Notice that $r_{\bm i} \in \{r^k, r^{k+1}, \dots, r^{k+\lambda}\}$ for
$S_{\bm i}\in\ms_k$. This leads to the following definition.

\begin{defn}[measure polynomial]\label{d:PB}
  For $B\in\xs B_k$ and $0\le \ell\le\lambda$, write
  \[\xi_{B,\ell}=\card\bigl\{S\in\ms_B\colon
  \text{the ratio of $S$ is $r^{k+\ell}$}\bigr\}.\]
  The polynomial
  \[P_B\colon t\mapsto \xi_{B,0} + \xi_{B,1}t + \dots +
   \xi_{B,\lambda}t^\lambda\]
  is called the measure polynomial of level-$k$ block~$B$. Write
  \[\xc P_k=\left\{P_B\colon B\in\xs B_k\right\}\quad\text{and}
  \quad \xc P=\bigcup_{k=0}^\infty \xc P_k.\]
\end{defn}
\begin{rem}\label{r:muB}
  For $B\in\xs B_k$, we have $\mu(B)=p^k P_B(p)$. This is why we call $P_B$
  the measure polynomial and $p$ the measure root of~$\ms$.
\end{rem}
\begin{rem}
  The measure polynomial of a block $B$ depends not only on~$B$ but also on
  the level of~$B$ since the level of~$B$ may be not unique. For example, let
  $\ms=\{S_1,S_2\}$ with $S_1\colon x\mapsto x/9$ and $S_2\colon x\mapsto
  x/3+2/3$. Then $r_\ms=1/3$ and
  \[\xs B_1=\bigl\{S_1(E_\ms),S_2(E_\ms)\bigr\},\quad
  \xs B_2=\bigl\{S_1(E_\ms),S_2\circ S_1(E_\ms),S_2\circ S_2(E_\ms)\bigr\}.\]
  Note that $S_1(E_\ms)\subset\xs B_1\cap\xs B_2$, so the level of
  $S_1(E_\ms)$ may be~$1$ or~$2$. Consequently, the measure polynomial of
  $S_1(E_\ms)$ may be $t\mapsto t$ for level-$1$ or $t\mapsto 1$ for
  level-$2$.
\end{rem}

Now we introduce the definition of interior blocks, which is the key to our
study.
\begin{defn}[interior block]\label{d:inblcok}
For $k\ge0$, $B\in\xs B_k$ is called a level-$k$ \emph{interior} block if
$\dist(B,O^c)\ge r^k|E|$ for some open set $O$ satisfying the SOSC. While
$B\in\xs B_k$ is called a level-$k$ \emph{boundary} block if $B$ is not a
level-$k$ interior block. Let $\xs B^\circ_k$ and $\xs B^\partial_k$ denote
the family of all level-$k$ interior blocks and the family of all level-$k$
boundary blocks, respectively. Write
  \[\xs B^\circ=\bigcup_{k\ge0}\xs B^\circ_k,\quad
  \xc P^\circ_k=\left\{P_B\colon B\in\xs B^\circ_k\right\},
  \quad \xc P^\circ=\bigcup_{k=0}^\infty \xc P^\circ_k.\]
\end{defn}
\begin{rem}\label{r:inblock}
  Suppose that $B\in\xs B^\circ_k$, then
  \begin{enumerate}[(a)]
    \item for all level-$l$ blocks $A\subset B$, we have $A\in\xs
        B^\circ_l$;
    \item suppose that $r_{\bm i}=r^l$, then $S_{\bm i}(B)\in\xs
        B^\circ_{k+l}$ and $P_{S_{\bm i}(B)}=P_B$.
  \end{enumerate}
\end{rem}

For further study of blocks decomposition, we introduce some more notations.
\begin{defn}[notations]\label{d:ntt}
  (a) Let $\xs C$ be a family of sets. Write
  \[\bigsqcup\xs C=\bigcup_{C\in\xs C}C.\]

  (b) Let $\xs A\subset\xs B_l$ be a nonempty family of level-$l$ blocks. For
  $k\ge0$, write
  \[\xs B_k(\xs A)=\Bigl\{B\in\xs B_{l+k}\colon B\subset
  \bigsqcup\xs A\Bigr\}.\]

  (c) Let $A\in\xs B_l$ be a level-$l$ block. For $k\ge0$, write
  \[\xs B^\circ_k(A)=\{B\in\xs B^\circ_{l+k}\colon B\subset A\},\quad
  \xs B^\partial_k(A)=\{B\in\xs B^\partial_{l+k}\colon B\subset A\}\]
  and $\xs B_k(A)=\{B\in\xs B_{l+k}\colon B\subset A\}$.
\end{defn}

The interior blocks have many advantages. The first is that, under the OSC,
different small copies of the self-similar set may has overlaps, but the
intersection of interior blocks in different small copies must be empty. The
second is that blocks in an interior block are still interior blocks (see
Remark~\ref{r:inblock}(a)). This means that we recover a form of disjointness
result for interior blocks. Therefore, the geometrical structure of interior
blocks is like the self-similar sets satisfying the SSC in some sense. The
last but not the least is the following lemma, which reveals the relationship
between the measure polynomials of interior blocks and the ideal of~$\ms$.
\begin{lem}\label{l:ideal}
  Let $I$ be the ideal of~$\Z[p]$ generated by $\{P(p)\colon P\in\xc
  P^\circ\}$, then $I=I_\ms$.
\end{lem}
\begin{proof}
  It follows from Remark~\ref{r:muB} and $p^{-1}\in\Z[p]$
  (Lemma~\ref{l:Z[p]}(a)) that $I$ is just the ideal generated by
  $\{\mu(B)\colon B\in\xs B^\circ\}$. We have $I\subset I_\ms$ since every
  interior block is an interior separated set. On the other hand, observe
  that each interior separated set can be written as a finite disjoint union
  of interior blocks, so $I_\ms\subset I$ holds too.
\end{proof}

\subsection{Finiteness of measure polynomials}

This subsection devoted to the finiteness of the measure polynomials, which
is the start point of our research. It follows from the totally
disconnectedness of the self-similar set.
\begin{prop}\label{p:finiteCP}
  There are only finitely many measure polynomials for every IFS
  $\ms\in\TOEC$.
\end{prop}

We need some lemmas to prove Proposition~\ref{p:finiteCP}. The first two,
Lemma~\ref{l:4compact} and~\ref{l:discnt}, are known facts in topology.

\begin{lem}[{\cite[\S2.10.21]{Feder69}}]\label{l:4compact}
  Let $X$ be a compact metric space and $\xs K(X)$ the set of all nonempty
  compact subset of~$X$, then $\xs K(X)$ is compact under the Hausdorff
  metric.
\end{lem}

\begin{lem}[see also~\cite{XiXi10}]\label{l:discnt}
  Let $\{F_i\}_{i=1}^n$ be a finite family of totally disconnected and
  compact subsets of a Hausdorff topology space, then $\bigcup_{i=1}^n F_i$
  is also totally disconnected.
\end{lem}

\begin{lem}\label{l:finite}
  Suppose that $x\in\R^d$ and $k\ge0$, then
  \[M=\sup_{x,k}\card\bigl\{S\in\ms_k\colon
   \dist\bigl(S(E),x\bigr)\le r^k|E|\bigr\}<\infty.\]
\end{lem}
\begin{proof}
  This is a simple consequence of the OSC. Let $O$ be an open
  set satisfying the OSC, then $\dist(O,E)=0$. And so
  $\dist\bigl(S(O),x\bigr)\le 2r^k|E|$ for all $S\in\ms_k$ such that
  $\dist\bigl(S(E),x\bigr)\le r^k|E|$. It follows that
  \[M\le\frac{\xc L\bigl(U(0,2|E|+|O|)\bigr)}{r^{\lambda d}\xc L(O)}\]
  since the diameter of~$S(O)$ is not less than $r^{k+\lambda}|O|$.
  Here $\xc L$ denotes the Lebesgue measure and $U(x,\rho)$ the open ball of
  radius~$\rho$ centered at~$x$.
\end{proof}

\begin{lem}\label{l:compact}
  Given $M\ge1$. Let $\xs F$ be the family of all nonempty compact
  subsets~$F$ of~$\R^d$ such that
  \begin{enumerate}[\upshape(i)]
    \item $F=\bigcup_{i=1}^MT_i(E)$, where each $T_i$ is a similar mapping
        with ratio lying in $\{1,r,\dots, r^\lambda\}$, we allow that
        $T_i=T_j$ for $i\ne j$;
    \item $\dist\bigl(T_i(E),0\bigr)\le|E|$ for $1\le i\le M$;
    \item $0\in F$.
  \end{enumerate}
  Then $\xs F$ is compact under the Hausdorff metric.
\end{lem}
\begin{proof}
  By Lemma~\ref{l:4compact} and Condition~(ii), it is sufficient to prove
  that $\xs F$ is closed. Suppose that $F_i=\bigcup_{j=1}^MT_{i,j}(E)\in\xs
  F$ and $F_i\to F$ under the Hausdorff metric, we shall show that $F\in\xs
  F$. Notice that the family of functions~$T_{i,j}$ is equicontinuous. By the
  Arzela-Ascoli Theorem and Condition~(ii), we can assume that $T_{i,j}$
  converge to some continuous mapping~$T_i$ under compact open topology as
  $j\to\infty$ for $1\le i\le M$ (that is, $T_{i,j}$ converge to~$T_i$
  uniformly on each compact set).

  Now let $F^*=\bigcup_{i=1}^MT_i(E)$. It is not difficult to check that
  $F_i\to F^*$ and $F^*\in\xs F$, and so $F=F^*\in\xs F$.
\end{proof}

\begin{lem}\label{l:continuous}
  Let $\xs F$ be as in Lemma~\ref{l:compact}. For $F\in\xs F$, let
  $F_\delta=\{x\colon \dist(x,F)\le\delta\}$ be the $\delta$-neighbourhood
  of~$F$ and $F_{\delta,0}$ the connected component of~$F_\delta$
  containing~$0$. Define
  \[\Delta(F)=\sup\bigl\{\delta\ge0\colon
  \text{$|x|<|E|$ for all $x\in F_{\delta,0}$}\bigr\},\] where $|x|$ denotes
  the usually absolute value of $x\in\R^d$ and $|E|$ the diameter of~$E$.
  Then $\Delta(F)>0$ for all $F\in\xs F$ and $\Delta$ is continuous on
  $\xs F$.
\end{lem}
\begin{proof}
  We first show that $\Delta(F)>0$ for all $F$. Suppose otherwise that
  $\Delta(F)=0$ for some $F\in\xs F$. Then for every $\delta>0$,
  $F_{\delta,0}$ contains an $x_\delta$ with $|x_\delta|\ge|E|$. We can pick
  $\delta_i\to0$ such that $F_{\delta_i,0}\to F_0$ (under the Hausdorff
  metric) and $x_{\delta_i}\to x_0$ for some compact set~$F_0$ and some point
  $x_0$ with $|x_0|\ge|E|$. It follows that $F_0$ is a connected component of
  $F$ containing two distinct points: $0$ and~$x_0$. This contradicts the
  fact that $F$ is totally disconnected (by Lemma~\ref{l:discnt}).

  Next we claim that $|\Delta(F)-\Delta(G)|\le\hdist(F,G)$ for $F,G\in\xs
  F$, where $\hdist$ denotes the Hausdorff metric, and so $\Delta$ is
  continuous. By the symmetry, it is sufficient to show that
  $\Delta(F)\le\Delta(G)+\hdist(F,G)$. Pick an $\delta>\Delta(G)$. Then
  $F_{\delta+\hdist(F,G)}\supset G_\delta\supset G_{\delta,0}$ which
  contains an $x$ with $|x|\ge|E|$. So we have
  $\Delta(F)\le\delta+\hdist(F,G)$. Desired inequality follows from the
  arbitrary of~$\delta$.
\end{proof}

\begin{proof}[Proof of Proposition~\ref{p:finiteCP}]
  We claim that there exists a positive integer $K$ such that for all $k\ge
  K$ and all $B\in\xs B_k$, we have $|B|\le 2r^{k-K}|E|$. Assume this is
  true, pick an open set~$O$ satisfying the OSC, we will show that
  $\sup_{P\in\xc P}P\bigl(\xc L(O)\bigr)<\infty$ ($\xc L$ denotes the
  Lebesgue measure), and so $\xc P$ must be finite.

  Let $B\in\xs B_k$ with $k\ge K$. By the claim and Remark~\ref{r:block}, we
  have $\bigl|\bigcup_{S\in\ms_B}S(E)\bigr|=|B|\le 2r^{k-K}|E|$. It follows
  from $E\subset\overline O$ (the closure of~$O$) that
  \[\biggl|\bigcup_{S\in\ms_B}S(O)\biggr|\le r^k\bigl(2r^{-K}|E|+2|O|\bigr)\]
  since the diameter of~$S(O)$ is not greater than $r^k|O|$ for $S\in\ms_B$.
  Thus
  \[P_B\bigl(\xc L(O)\bigr)=r^{-dk}\sum_{S\in\ms_B}\xc L\bigl(S(O)\bigr)
  \le\xc L\bigl(U(0,2r^{-K}|E|+2|O|)\bigr).\]
  This obviously implies that $\sup_{P\in\xc P}P\bigl(\xc L(O)\bigr)<\infty$.

  It remains to verify our claim. Consider the family~$\xs F$ in
  Lemma~\ref{l:compact} with the constant~$M$ being as in
  Lemma~\ref{l:finite}. By Lemma~\ref{l:continuous}, we can find a positive
  integer~$K$ such that
  \[0<r^K|E|<\inf_{F\in\xs F}\Delta(F).\]
  We will show this~$K$ is desired. For this, let $B\in\xs B_k$ with $k\ge
  K$ and $x\in B$, consider the similar mapping $T\colon y\mapsto
  r^{K-k}(y-x)$. Let
  \[F=\bigcup_{\substack{S\in\ms_{k-K}\\\dist(T\circ S(E),0)\le|E|}}
  T\circ S(E).\]
  Then $F\in\xs F$ by Lemma~\ref{l:finite}. Thus $0=T(x)\in T(B)\subset
  F_{r^K|E|,0}\subset U(0,|E|)$. This means that $|B|\le 2r^{k-K}|E|$.
\end{proof}
\begin{rem}\label{r:diameter}
  From the proof of Proposition~\ref{p:finiteCP}, we conclude that there
  exists a constant $\varpi>1$ such that for all $k\ge0$ and all $B\in\xs
  B_k$,
  \[\varpi^{-1}r^k|E|\le|B|\le\varpi r^k|E|.\]
\end{rem}

\subsection{The cardinality of blocks}\label{ssec:BDnum}

In this subsection, we show that almost all blocks are interior blocks. This
conclusion follows from two lemmas.
\begin{lem}\label{l:C(k)}
  Let $\zeta(k)=\card\xs B^\partial_k$ be the number of all level-$k$ boundary
  blocks, then
  \[\lim_{k\to\infty}p^k\cdot \zeta(k)=0.\]
\end{lem}
\begin{proof}
  Let $O$ be an open set satisfying the SOSC. Write
  \[\xs B^O_k=\bigl\{B\in\xs B_k\colon\dist(B,O^c)\ge r^k|E|\bigr\},\]
  then $\xs B^O_k\subset\xs B^\circ_k$. Notice that $p^k \cdot \zeta(k) \le
  \bigl(\min_{P\in\xc P}P(p)\bigr)^{-1} \sum_{B\in\xs
  B^\partial_k}p^kP_B(p)$. By Remark~\ref{r:muB},
  \[\begin{split}
    \sum_{B\in\xs B^\partial_k}p^kP_B(p)
    &\le\sum_{B\in\xs B_k\setminus\xs B^O_k}p^kP_B(p)
    =\sum_{B\in\xs B_k\setminus\xs B^O_k}\mu(B)\\
    &=\mu\biggl(\bigcup_{B\in\xs B_k\setminus\xs B^O_k}B
    \biggr)\to\mu(E\setminus O)=0,\quad\text{as $k\to\infty$}.\qedhere
  \end{split}\]
\end{proof}

\begin{lem}\label{l:CP(k)}
  Let $\zeta_P(k)=\card\{B\in\xs B^\circ_k\colon P_B=P\}$ be the number of all
  level-$k$ interior blocks which measure polynomial is~$P$, then for each
  measure polynomial $P\in\xc P^\circ$,
  \[\liminf_{k\to\infty}p^k\cdot \zeta_P(k)>0.\]
\end{lem}

\begin{proof}
  For $\ell\in\{0,1,\dots,\lambda\}$ and $k\ge1$, write
  \begin{equation}\label{eq:numkj}
    \xi_{k,\ell}=\card\{S\in\ms_k\colon
    \text{the ratio of $S$ is $r^{k+\ell}$}\}.
  \end{equation}
  Let $\bm\xi_k=(\xi_{k,0},\xi_{k,1},\dots,\xi_{k,\lambda})^T$ and
  \begin{equation}\label{eq:Xi}
    \bm\Xi=\begin{pmatrix}
      \xi_{1,0} & 1 & 0 & \hdotsfor4 & 0 \\
      \xi_{1,1} & 0 & 1 & 0 & \hdotsfor3 & 0 \\
      \xi_{1,2} & 0 & 0 & 1 & 0 & \hdotsfor2 & 0 \\
      \vdots & \vdots & \vdots & \vdots & \vdots & \vdots & \vdots & \vdots\\
      \vdots & \vdots & \vdots & \vdots & \vdots & \vdots & \vdots & \vdots\\
      \xi_{1,\lambda-2} & 0 & \hdotsfor3 & 0 & 1 & 0 \\
      \xi_{1,\lambda-1} & 0 & \hdotsfor4 & 0 & 1 \\
      \xi_{1,\lambda} & 0 & \hdotsfor5 & 0 \\
    \end{pmatrix}
  \end{equation}
  be a $(\lambda+1)\times(\lambda+1)$ matrix. Then we have the following
  recursion formula.
  \begin{equation}\label{eq:numktok+1}
    \bm\xi_{k+1}=\bm\Xi\bm\xi_k,\qquad k\ge1.
  \end{equation}
  Note that $\ms_1=\ms$, and so
  \[\xi_{1,0}p+\xi_{1,1}p^2+\dots+\xi_{1,\lambda}p^{\lambda+1}=
  p^{\lambda_1}+p^{\lambda_2}+\dots+p^{\lambda_N}=1,\]
  where $\lambda_i$ are as in~\eqref{eq:lambda}. Therefore, the
  matrix~$\bm\Xi$ satisfies the conditions of Lemma~\ref{l:primitive}. So
  $\bm\Xi$ is primitive and $p^{-1}$ is the Perron-Frobenius eigenvalue.

  For $P\in\xc P^\circ$, there exists a level-$l$ interior block $B$ for some
  $l\ge1$ such that $P_B=P$. It follows from Remark~\ref{r:inblock}(b) that
  \[\zeta_P(k)\ge\xi_{k-l,0}\quad\text{for $k>l$}.\]
  Then this lemma follows from the recursion formula~\eqref{eq:numktok+1} and
  the fact that $p^{-1}$ is the Perron-Frobenius eigenvalue of the primitive
  matrix~$\bm\Xi$.
\end{proof}

\begin{rem}\label{r:ideal}
  Let $O$ and $\xs B^O_k$ be as in the proof of Lemma~\ref{l:C(k)}. From the
  proof of Lemma~\ref{l:C(k)}, we have
  \[\lim_{k\to\infty}p^k\card(\xs B_k\setminus\xs B^O_k)=0.\]
  Together with Lemma~\ref{l:ideal}, Lemma~\ref{l:C(k)} and
  Lemma~\ref{l:CP(k)}, we see that $I_\ms$ can be generated by
  $\bigl\{\mu(B)\colon B\in\xs B^O_k\ \text{for some $k\ge1$}\bigr\}$, where
  $O$ is an arbitrary open set satisfying the SOSC. This means that we need
  only find a specific open set satisfying the SOSC when we want to determine
  the ideal of an IFS.
\end{rem}

The following lemma is a corollary of Lemma~\ref{l:C(k)} and~\ref{l:CP(k)}.
Recall notations in Definition~\ref{d:ntt}(c). For $k\ge0$, define
\[\zeta^\partial(k)=\max_{B\in\xs B}\card\xs B^\partial_k(B)\quad\text{and}
\quad\zeta^\circ(k)=\min_{B\in\xs B,\,P\in\xc P^\circ}
\card\bigl\{A\in\xs B^\circ_k(B)\colon P_A=P\bigr\}.\]
\begin{lem}\label{l:CB(k)}
  We have
  \[\lim_{k\to\infty}p^k\zeta^\partial(k)=0\quad\text{and}\quad
  \liminf_{k\to\infty}p^k\zeta^\circ(k)>0.\]
\end{lem}
\begin{proof}
  Let $B\in\xs B_l$ and $P_B$ the measure polynomial of~$B$. Recall
  that $P_B(t)=\sum_{\ell=0}^\lambda\xi_{B,\ell}t^\ell$, where
  $\xi_{B,\ell}=\{S\in\ms_B\colon\text{the ratio of~$S$ is $r^{l+\ell}$}\}$.

  To prove the first limit, we use Remark~\ref{r:inblock}(b) to obtain that,
  for $k>\lambda$,
  \[\card\xs B^\partial_k(B)\le\sum_{\ell=0}^\lambda\xi_{B,\ell}\zeta(k-\ell)
    \le P_B(1)\sum_{\ell=0}^\lambda \zeta(k-\ell).\]
  So for $k>\lambda$,
  \[\zeta^\partial(k)\le\max_{P\in\xc P}P(1)\sum_{\ell=0}^\lambda \zeta(k-\ell).\]
  By Lemma~\ref{l:C(k)}, we have $\lim_{k\to\infty}p^k\zeta^\partial(k)=0$.

  To prove the second limit, let $P\in\xc P^\circ$, also by
  Remark~\ref{r:inblock}(b), we have,
  \[\card\bigl\{A\in\xs B^\circ_k(B)\colon P_A=P\bigr\}
  \ge\sum_{\ell=0}^\lambda\xi_{B,\ell}\zeta_P(k-\ell)
  \ge\min_{0\le \ell\le\lambda}\zeta_P(k-\ell)\]
  for $k>\lambda$. And so for $k>\lambda$,
  \[\zeta^\circ(k)\ge\min_{0\le \ell\le\lambda,\,P\in\xc P^\circ}\zeta_P(k-\ell).\]
  By Lemma~\ref{l:CP(k)}, we have $\liminf_{k\to\infty}p^k\zeta^\circ(k)>0$.
\end{proof}

We close this subsection by the proof of Theorem~\ref{t:gdid}.

\begin{proof}[Proof of Theorem~\ref{t:gdid}]
  Let $I$ denote the ideal generated by
  \[\bigl\{\xc H^s(E_i)/\xc H^s(E_\ms)\colon i\in\xc V_O\bigr\}.\]
  Since the graph-directed sets $\{E_i\}_{i\in\xc V}$ are dust-like, the set
  $S_{e_1}\circ S_{e_2}\circ\dots\circ S_{e_k}(E_i)\subset O$, where
  $(e_1,e_2,\dots,e_k)\in\xc E_{0,i}^k$, is an interior separated set
  of~$E_\ms=E_0$. So the ideal $I_\ms$ contains $\mu(S_{e_1}\circ\dots\circ
  S_{e_k}(E_i))$. By the condition that $\log r_e/\log r\in\N$ for all
  $e\in\xc E$, we have
  \[\mu(S_{e_1}\circ\dots\circ S_{e_k}(E_i))=p^\ell\xc H^s(E_i)/\xc H^s(E_\ms)
  \]
  for some positive integer $\ell$. This means that $\xc H^s(E_i)/\xc
  H^s(E_\ms)\in I_\ms$ for all $i\in\xc V_O$ since $p^{-1}\in\Z[p]$ by
  Lemma~\ref{l:Z[p]}(a). Thus we have $I\subset I_\ms$.

  Conversely, by Remark~\ref{r:ideal}, we know that $I_\ms$ is generated by
  \[\bigl\{\mu(B)\colon B\in\xs B^O_k\ \text{for some $k\ge1$}\bigr\}.\]
  Fix a $B\in\bigcup_{k\ge1}\xs B^O_k$. It follows from $B$ is a block that,
  for large enough positive integer~$\ell$,
  \[B=\bigcup_{i\in\xc V_O}\bigcup_{
  \substack{(e_1,\dots,e_\ell)\in\xc E_{0,i}^\ell\\
  S_{e_1}\circ\dots\circ S_{e_\ell}(E_i)\subset B}} S_{e_1}\circ\dots\circ
  S_{e_\ell}(E_i).\]
  We have $\mu(B)\in I$ since the above union is disjoint. And so
  $I_\ms\subset I$.
\end{proof}

\section{Main Ideas of the Proof}\label{sec:idea}

The most difficult part in our proof is the sufficient part of
Theorem~\ref{t:TOCElip}. It is rather tedious and technical, requiring
delicate composition and decomposition of blocks. Although the proof is very
complicated, the main ideas behind it is simple. This section is devoted to
the introduction of these ideas.

\subsection{Cylinder structure and dense island structure}

It is usually very difficult to define bi-Lipschitz mappings between given
sets. However, if these sets have some special structure, things become
somewhat easy. In this paper, we make use of two special structures: the
cylinder structure and the dense island structure.

Lemma~\ref{l:cylin} consider the Lipschitz equivalence between sets with
structure of nested Cantor sets. We call this cylinder structure
(Definition~\ref{d:qscldr}). Lemma~\ref{l:djt} consider the dense island
structure (Definition~\ref{d:dnsilnd}), which involves the idea of extension
of bi-Lipschitz mapping. This idea is also used by Llorente and
Mattila~\cite{LloMa10}.

\begin{defn}
  A family of disjoint subsets of a set~$F$ is called a \emph{partition}
  of~$F$ if the union of the family is~$F$.
\end{defn}

\begin{defn}
  Let $\xs C_1$ and $\xs C_2$ be two partitions of a set $F$. We say that
  $\xs C_2$ is \emph{finer} than~$\xs C_1$, denoted by $\xs C_1\prec\xs C_2$,
  if each set in~$\xs C_2$ is a subset of some set in~$\xs C_1$. This is
  equivalent to that each set in~$\xs C_1$ is a union of some sets in~$\xs
  C_2$.
\end{defn}

\begin{defn}[cylinder structure]\label{d:qscldr}
  Let $F$ be a compact subset of a metric space. We say $F$ has
  \emph{$(\varrho,\iota)$-cylinder structure} for $\varrho\in(0,1)$ and
  $\iota\ge1$ if there exist families $\xs C_k$ for $k\ge1$ such that
  \begin{enumerate}[(i)]
    \item each $\xs C_k$ is a partition of~$F$;
    \item $\xs C_1\prec\xs C_2\prec\dots\prec\xs C_k\prec\xs
        C_{k+1}\prec\dotsb$;
    \item for each $k\ge1$,
        \begin{alignat*}{2}
          \iota^{-1}\varrho^k&\le|C|/|F|\le\iota\varrho^k &&\qquad
          \text{for all $C\in\xs C_k$}; \\
          \iota^{-1}\varrho^k&\le\dist(C_1,C_2)/|F| &&\qquad
          \text{for distinct $C_1,C_2\in \xs C_k$};
        \end{alignat*}
        where $|\cdot|$ denotes the diameter.
  \end{enumerate}
  The sets in $\xs C_k$ ($k\ge1$) are called cylinders and the families $\xs
  C_k$ are called cylinder families.
\end{defn}
\begin{exmp}
  Let $X=\{0,1\}^\N$ be the symbolic space with metric
  \[\rho(x,y)=2^{-\inf\{k\colon x_k\ne y_k\}}\]
  for $x=x_1x_2\dotsc$ and $y=y_1y_2\dotsc$. For each $k\ge1$ and each word
  $w=w_1w_2\dots w_k$ of length~$k$, define cylinder
  \[[w]=\{x\in X\colon x_1x_2\dots x_k=w_1w_2\dots w_k\}.\]
  Let $\xs C_k=\{[w]\colon\text{$w$ has length~$k$}\}$ for $k\ge1$. We see
  that $X$ has $(1/2,1)$-cylinder structure.
\end{exmp}

\begin{exmp}
  Let $\ms\in\TOEC$ and $\xs C_k=\xs B_k$. By the definition of blocks
  (Definition~\ref{d:block}) and Remark~\ref{r:diameter}, we know that the
  self-similar set~$E_\ms$ has $(r_\ms,\varpi_\ms)$-cylinder structure.
\end{exmp}

\begin{defn}\label{d:smcyl}
  Suppose that $F$ and $F'$ have $(\varrho,\iota)$-cylinder structure with
  cylinder families $\xs C_k$ and $\xs C'_k$, respectively. We say $F$ and
  $F'$ have the same $(\varrho,\iota)$-cylinder structure if there exists a
  one-to-one mapping $\tilde f$ of $\bigcup_{k=1}^\infty\xs C_k$ onto
  $\bigcup_{k=1}^\infty\xs C'_k$ such that
  \begin{enumerate}[(i)]
    \item $\tilde f$ maps $\xs C_k$ onto $\xs C'_k$ for all $k\ge1$;
    \item for $C_1\in\xs C_{k_1}$ and $C_2\in\xs C_{k_2}$, where $k_1<k_2$,
        we have $\tilde f(C_1)\supset\tilde f(C_2)$ if and only if
        $C_1\supset C_2$.
  \end{enumerate}
  We call $\tilde f$ the cylinder mapping.
\end{defn}

\begin{lem}\label{l:cylin}
  Suppose that $F$ and $F'$ have the same $(\varrho,\iota)$-cylinder
  structure, then there is a bi-Lipschitz mapping~$f$ of~$F$ onto~$F'$ such
  that
  \[\varrho\iota^{-2}\rho(x,y)/|F|\le\rho(f(x),f(y))/|F'|\le
  \varrho^{-1}\iota^2\rho(x,y)/|F|\quad\text{for $x,y\in F$}.\]
  Here $\rho$ denotes the metric. In particular, $F\simeq F'$.
\end{lem}
\begin{proof}
  We use the same notations as in Definition~\ref{d:smcyl}. For $x\in F$,
  there exists a unique $C_k\in\xs C_k$ for each $k\ge1$ such that $x\in C_k$
  since $F=\bigsqcup\xs C_k$ and the union is disjoint. Since $C_1\supset
  C_2\supset\dots\supset C_k\supset\dotsb$, we have
  \[\tilde f(C_1)\supset\tilde f(C_2)\supset\dots\supset\tilde
  f(C_k)\supset\dotsb.\]
  Together with the fact that $|\tilde f(C_k)|\to0$ as $k\to\infty$, we know
  that there is a unique $x'\in\bigcap_{k=1}^\infty\tilde f(C_k)$. This leads
  to a mapping $f\colon F\to F'$, $x\mapsto x'$. It remains to show that $f$
  is the desired mapping. For this, let $x,y\in F$. Then there exists a
  $k\ge1$, $C\in\xs C_k$ and distinct $C_x,C_y\in\xs C_{k+1}$ such that
  $x,y\in C$ and $x\in C_x$, $y\in C_y$. By the definition of~$f$, we have
  $f(x),f(y)\in\tilde f(C)$ and $f(x)\in\tilde f(C_x)$, $f(y)\in\tilde
  f(C_y)$. It follows from the definition of quasi cylinder structure that
  \begin{gather*}
    \iota^{-1}\varrho^{k+1}\le\dist(C_x,C_y)/|F|\le\rho(x,y)/|F|\le|C|/|F|
    \le\iota\varrho^k, \\
    \iota^{-1}\varrho^{k+1}\le\dist(\tilde f(C_x),\tilde f(C_y))/|F'|\le
    \rho(f(x),f(y))/|F'|\le|\tilde f(C)|/|F'|\le\iota\varrho^k.
  \end{gather*}
  Thus, $f$ satisfies the inequality in this lemma. Finally, we have
  $f(F)=F'$ since $f(F)$ is compact and dense in~$F'$.
\end{proof}

Recall that $\bigsqcup\xs D=\bigcup_{D\in\xs D}D$ for any family~$\xs D$ of
sets (Definition~\ref{d:ntt}(a)).
\begin{defn}[dense island structure]\label{d:dnsilnd}
  Let $F$ be a compact set in a metric space. A subset $D$ of $F$ is called
  an $\iota$-island for $\iota>0$ if
  \[|D|\le\iota\dist(D,F\setminus D).\]
  We say that $F$ has dense $\iota$-island structure if there exists a
  family~$\xs D$ of disjoint $\iota$-islands of~$F$ such that $\bigsqcup\xs
  D$ is dense in~$F$.
\end{defn}

\begin{exmp}
  Let $X=\{0,1\}^\N$ be the symbolic space with metric
  \[\rho(x,y)=2^{-\inf\{k\colon x_k\ne y_k\}}\]
  for $x=x_1x_2\dotsc$ and $y=y_1y_2\dotsc$. Then $X$ has dense $1/2$-island
  structure with families $\xs D=\{[0^k1]\colon k\ge0\}$.
\end{exmp}

\begin{defn}\label{d:smdi}
  Suppose that $F$ and $F'$ have dense $\iota$-island structure with disjoint
  $\iota$-island families $\xs D$ and $\xs D'$, respectively. We say that $F$
  and $F'$ have the same dense $\iota$-island structure if there exist a
  one-to-one mapping $\tilde f$ of $\xs D$ onto $\xs D'$ and a constant
  $\tilde L>1$ such that
  \begin{enumerate}[(i)]
    \item $\tilde L^{-1}\dist(D_1,D_2)/|F|\le\dist(\tilde f(D_1),\tilde
        f(D_2))/|F'|\le\tilde L\dist(D_1,D_2)/|F|$ for each two distinct
        $\iota$-islands $D_1,D_2\in\xs D$;
    \item for each $\iota$-island $D\in\xs D$, there is a bi-Lipschitz
        mapping $f_D$ of $D$ onto $\tilde f(D)$ such that
        \[\tilde L^{-1}\rho(x,y)/|F|\le\rho(f_D(x),f_D(y))/|F'|\le
        \tilde L\rho(x,y)/|F|\quad\text{for $x,y\in D$}.\]
        Here $\rho$ denotes the metric.
  \end{enumerate}
  We call $\tilde f$ the island mapping.
\end{defn}

\begin{lem}\label{l:djt}
  Let $F$ and $F'$ have dense $\iota$-island structure with the
  $\iota$-island families $\xs D$ and $\xs D'$, respectively. If $F$ and $F'$
  have the same dense $\iota$-island structure with island mapping~$\tilde f$
  and constant $\tilde L$. Then there is a bi-Lipschitz mapping~$f$ of~$F$
  onto~$F'$ such that
  \begin{equation}\label{eq:eddjt}
    L^{-1}\rho(x,y)/|F|\le\rho(f(x),f(y))/|F'|\le L\rho(x,y)/|F|
    \quad\text{for $x,y\in F$}.
  \end{equation}
  Here $L=(2\iota+1)\tilde L$. In particular, $F\simeq F'$.
\end{lem}
\begin{proof}
  Let $f$ be the one-to-one mapping of $\bigsqcup\xs D$ onto $\bigsqcup\xs
  D'$ such that the restriction of $f$ to $D$ is just $f_D$, i.e.,
  $f|_D=f_D$, for every $D\in\xs D$, where $f_D$ is as in
  Definition~\ref{d:smdi}. We claim that $f$ and $L=(2\iota+1)\tilde L$
  satisfies the inequality~\eqref{eq:eddjt} for $x,y\in\bigsqcup\xs D$. For
  this, let $x,y\in\bigsqcup\xs D$. There are two cases. If $x,y\in D$ for
  some $D\in\xs D$, then $f$ satisfies the inequality~\eqref{eq:eddjt} for
  $L=\tilde L$ since $f|_D=f_D$. Suppose otherwise that $x\in D_x$ and $y\in
  D_y$ for distinct $D_x,D_y\in\xs D$. Then by the definition of~$f$,
  $f(x)\in\tilde f(D_x)$ and $f(y)\in\tilde f(D_y)$. Therefore,
  \begin{gather*}
    \dist(D_x,D_y)\le\rho(x,y)\le
    |D_x|+\dist(D_x,D_y)+|D_y|\le(2\iota+1)\dist(D_x,D_y), \\
  \begin{split}
    \dist(\tilde f(D_x),\tilde f(D_y))\le\rho(f(x),f(y))
    &\le|\tilde f(D_x)|+\dist(\tilde f(D_x),\tilde f(D_y))+|\tilde f(D_y)|\\
    &\le(2\iota+1)\dist(\tilde f(D_x),\tilde f(D_y)).
  \end{split}
  \end{gather*}
  Together with Condition~(i) of Definition~\ref{d:smdi}, we have
  \[((2\iota+1)\tilde L)^{-1}\rho(x,y)/|F|\le\rho(f_D(x),f_D(y))/|F'|
  \le(2\iota+1)\tilde L\rho(x,y)/|F|.\]
  A summary of the above two cases shows that $f$ and $L=(2\iota+1)\tilde L$
  satisfy the inequality~\eqref{eq:eddjt} for $x,y\in\bigsqcup\xs D$.
  Finally, note that $f$ can be extended to a bi-Lipschitz mapping from $F$
  onto $F'$ since $\bigsqcup\xs D$ and $\bigsqcup\xs D'$ are dense in $F$ and
  $F'$, respectively. We also denote this mapping by~$f$. Then $f$ and $L$
  satisfy the inequality~\eqref{eq:eddjt} for $x,y\in F$.
\end{proof}

\subsection{Measure linear}\label{ssec:ML}

To make use of Lemma~\ref{l:cylin} and~\ref{l:djt}, we need to construct
corresponding structure for given sets. The difficult is how to do it. An
important observation obtained by Cooper and Pignataro~\cite{CooPi88} gives
the key hint. This observation is called measure linear.
\begin{defn}[measure linear]\label{d:msrln}
  Let $(X,\mu)$ and $(Y,\nu)$ be two measure spaces. A map $f\colon X\to Y$
  is called \emph{measure linear} if there is a constant $a>0$ such that for
  all $\mu$-measurable sets $A\subset X$, $f(A)$ is $\nu$-measurable and
  $\nu(f(A))=a\mu(A)$.
\end{defn}
Let $f$ be a bi-Lipschitz mapping from a self-similar set~$E$ onto another
self-similar set. Cooper and Pignataro~\cite{CooPi88} showed that the
restriction of $f$ to some small copy of~$E$ is measure linear for
$s$-dimensional Hausdorff measure, where $s=\hdim E$, provided that the two
self-similar sets both satisfy the SSC (see Lemma~\ref{l:ML}). In fact,
measure linear property also holds in our setting (see Lemma~\ref{l:ml}).

Inspired by the observation of measure linear, it is natural to require
\begin{equation}\label{eq:ML}
  \xc H^s(\tilde f(C))/\xc H^s(C)=\xc H^s(F')/\xc H^s(F)
  \quad\text{for all cylinders $C$}
\end{equation}
in the construction of same cylinder structure for $F$ and $F'$. In fact,
this is just the case in the proofs of Lemma~\ref{l:4blockslip} and the
sufficient part of Proposition~\ref{p:BlipB}. We remark that a cylinder
mapping~$\tilde f$ satisfying~\eqref{eq:ML} induces a bi-Lipschitz
mapping~$f$ (as in the proofs of Lemma~\ref{l:cylin}) such that $f$ is
measure linear on the whole set~$F$. We also remark that all the bi-Lipschitz
mappings appearing in~\cite{DenHe12,LuoLa12,RaRuX06,RuWaX12,XiRu07,XiXi10}
have the measure linear property on the whole set.

However, in many cases, e.g., the rotation version of \{1,3,5\}-\{1,4,5\}
problem in~\cite{XioXi12}, bi-Lipschitz mappings which are measure linear on
the whole set do not exist. In other words, there only exist bi-Lipschitz
mappings which are measure linear on subset. In fact, this is exactly what
obtained by Cooper and Pignataro~\cite{CooPi88}. In such cases, we cannot
construct the same cylinder structure by the inspiration of measure linear.
This makes our proof much more complicated. In such cases, we must consider
all the subsets on which some bi-Lipschitz mapping is measure linear.
Falconer and Marsh~\cite{FalMa92} showed that the union of such subsets are
dense in the whole set. This is why we introduce the dense island structure.
Lemma~\ref{l:djt} is our tool to deal with these cases, see the proofs of
Lemma~\ref{l:4blockslip'} and Proposition~\ref{p:BlipE}.

We close this subsection with an example in~\cite{XioXi12} to show that
bi-Lipschitz mappings which are measure linear on whole set do not exist.
\begin{exmp}\label{e:NML}
  We consider the $\{1,3,5\}$-set and the $\{1,-4,5\}$-set (see
  Figure~\ref{fig:135145r}). Recall that the two sets are the self-similar
  sets such that
  \begin{align*}
    E_{1,3,5} & =(E_{1,3,5}/5)\cup(E_{1,3,5}/5+2/5)\cup(E_{1,3,5}/5+4/5), \\
    E_{1,-4,5} & =(E_{1,-4,5}/5)\cup(-E_{1,-4,5}/5+4/5)\cup(E_{1,-4,5}/5+4/5).
  \end{align*}
  It follows from Theorem~\ref{t:srt} that $E_{1,3,5}\simeq E_{1,-4,5}$. But
  bi-Lipschitz mappings of~$E_{1,3,5}$ onto~$E_{1,-4,5}$ which are measure
  linear on~$E_{1,3,5}$ do not exist.

  Suppose on the contrary that $f$ is a such mapping. Then we have
  \[\nu(f(A))=\mu(A)\quad\text{for all Borel sets $A\subset E_{1,3,5}$},\]
  where $\mu$ and $\nu$ are the natural measure of $E_{1,3,5}$ and
  $E_{1,-4,5}$, respectively. Let $F=2E_{1,3,5}$, then rewrite
  \[E_{1,-4,5}=(E_{1,-4,5}/5)\cup(F/5+3/5).\]
  We say that $A$ is a small copy of a self-similar set~$E$ if $A=S(E)$,
  where $S$ can be written as a composition of similarities in the IFS
  of~$E$. Now let $A$ be a small copy of~$E_{1,3,5}$ such that
  $f(A)\subset(F/5+3/5)$, then $\mu(A)=3^{-k}$ for some positive integer~$k$
  and
  \[f(A)=(F_1\cup\dots\cup F_n)/5+3/5\subset(F/5+3/5),\]
  where $F_i$ is a small copy of $F$ for $1\le i\le n$. So for each $1\le
  i\le n$, there is a positive integer~$k_i$ such that $\nu(F_i)=2\cdot
  3^{-k_i}$. Therefore,
  \[\nu(f(A))=2\cdot3^{-1}(3^{-k_1}+\dots+3^{-k_n})=\mu(A)=3^{-k}.\]
  But this is impossible.
\end{exmp}

\subsection{Suitable decomposition}\label{ssec:BDsd}

In the construction of cylinder mapping satisfying~\eqref{eq:ML}, it is often
required to decompose interior blocks into small parts with measures equal to
given numbers. We call such decompositions the suitable decomposition
(Definition~\ref{d:suitde}). This subsection deals with the problem of the
existence of suitable decomposition (Lemma~\ref{l:suitde}).

We begin with the definition of suitable decomposition by using the notations
in Definition~\ref{d:ntt}.
\begin{defn}[suitable decomposition]\label{d:suitde}
  Let $\xs A\subset\xs B_l$ be a nonempty family of level-$l$ blocks. We call
  \[\xs B_k(\xs A)=\bigcup_{i=1}^n\xs A_i\]
  an order-$k$ suitable decomposition for positive numbers
  $a_1,a_2,\dots,a_n$, if $\xs A_i\cap\xs A_j=\emptyset$ for $i\ne j$ and
  \[\mu\left(\bigsqcup\xs A_i\right)=\sum_{B\in\xs A_i}\mu(B)
  =a_i\quad\text{for $1\le i\le n$}.\]
\end{defn}
\begin{rem}\label{r:suit}
  It is plain to observe that the existence of an order-$k$ suitable
  decomposition ensures the existence of an order-$K$ suitable decomposition
  for all $K\ge k$ since each level-$(l+k)$ block can be written as a
  disjoint union of some level-$(l+K)$ blocks.
\end{rem}
\begin{defn}\label{d:suit}
  Let $\alpha_1$, $\alpha_2$, $\dots$, $\alpha_\eta$ be positive numbers and
  $\xs A\subset\xs B_l$ a nonempty family of level-$l$ blocks. We say that
  $a_1$, $a_2$, \dots, $a_n$ are $(\xs
  A;\alpha_1,\dots,\alpha_\eta)$-suitable if
  \[\sum_{i=1}^na_i=\sum_{B\in\xs A}P_B(p)\]
  and $a_i\in\{\alpha_1,\dots,\alpha_\eta\}$ for all $i=1,2,\dots,n$.
\end{defn}
\begin{lem}\label{l:suitde}
  Given positive numbers $\alpha_1,\alpha_2,\dots,\alpha_\eta\in I_\ms$,
  there exists an integer $K\ge1$ depending only on $\ms$ and $\alpha_1$,
  $\dots$, $\alpha_\eta$ has the following property. For each ${l\ge1}$, each
  nonempty family of level-$l$ interior blocks $\xs A\subset\xs B^\circ_l$
  and each $(\xs A;\alpha_1,\dots,\alpha_\eta)$-suitable sequence
  $a_1,a_2,\dots,a_n$, there is an order-$K$ suitable decomposition
  \[\xs B_K(\xs A)=\bigcup_{i=1}^n\xs A_i\]
  for $p^la_1,\dots,p^la_n$.
\end{lem}
Lemma~\ref{l:suitde} ensures the measure linear property of bi-Lipschitz
mapping in our proof. Before proving it, we present two technical lemmas.

\begin{lem}\label{l:idealP}
  Let $\alpha\in I_\ms$ be a positive number. Then for each sufficiently
  large integer $\ell$, there exist integers $\gamma_{\ell,P}\ge0$ for
  $P\in\xc P^\circ$ such that
  \[\alpha=p^\ell\sum_{P\in\xc P^\circ}\gamma_{\ell,P}P(p).\]
\end{lem}
\begin{proof}
  Let $I^*$ denotes the set of all the positive numbers $\alpha\in I_\ms$
  satisfying the property in the lemma. We claim that $p^m P(p)\in I^*$ for
  all integers~$m$ and all $P\in\xc P^\circ$. For this, suppose that $P=P_A$
  for some $A\in\xs B^\circ_l$. By Remark~\ref{r:muB} and
  Remark~\ref{r:inblock}(a), for all $k\ge1$,
  \[p^m P(p)=p^{m-l}\mu(A)=p^{m-l}\sum_{B\in\xs B_k(A)}
  \mu(B)=p^{m+k}\sum_{B\in\xs B^\circ_k(A)}P_{B}(p),\]
  where $\xs B_k(A)=\{B\in\xs B_{l+k}\colon B\subset A\}$ and $\xs
  B^\circ_k(A)=\{B\in\xs B_{l+k}\colon B\subset A\}$. This means that $p^m
  P(p)\in I^*$. Now let $\alpha\in I_\ms$ be a positive number, then by
  Lemma~\ref{l:ideal}, Lemma~\ref{l:Z[p]}(e) and~(f), we have
  \[\alpha=\sum_{P\in\xc P^\circ}b_PP(p),\]
  where each $b_P$ can be written as $p^{m_1}+\dots+p^{m_\kappa}$ for some
  integers $m_1$, \dots, $m_\kappa$. Observe that $I^*+I^*\subset I^*$.
  This fact together with $p^m P(p)\in I^*$ yields $\alpha\in I^*$, and so
  $I^*=\{\alpha\in I_\ms\colon \alpha>0\}$.
\end{proof}

\begin{lem}\label{l:dvsum}
  Given positive numbers $\alpha_1,\alpha_2,\dots,\alpha_\eta$ and
  $\beta_1,\beta_2,\dots,\beta_\tau$, let $\Sigma_\theta$ be the set of all
  vectors $(\kappa_1,\kappa_2,\dots,\kappa_{\eta+\tau})$ with nonnegative
  integer entries such that
  \[\kappa_1\alpha_1 + \kappa_2\alpha_2 + \dots + \kappa_\eta\alpha_\eta =
  \kappa_{\eta+1}\beta_1 + \kappa_{\eta+2}\beta_2 + \dots +
  \kappa_{\eta+\tau}\beta_\tau>\theta,\]
  for $\theta\ge0$. Suppose that $\Sigma_0\ne\emptyset$. Then there exists a
  constant $\theta>0$ such that each vector
  $(\kappa_1,\kappa_2,\dots,\kappa_{\eta+\tau})\in\Sigma_\theta$ can be
  written as the sun of two vectors in~$\Sigma_0$, i.e., there exist
  $(\kappa'_1, \dots, \kappa'_{\eta+\tau}), (\kappa''_1, \dots,
  \kappa''_{\eta+\tau}) \in \Sigma_0$ such that
  \[\kappa_i=\kappa'_i+\kappa''_i,\quad\text{for}\ 1\le i\le \eta+\tau.\]
\end{lem}
\begin{proof}
  Suppose to the contrary that such $\theta$ does not exist. So we can find a
  sequence of vectors $(\kappa^{(j)}_1, \dots, \kappa^{(j)}_{\eta+\tau}) \in
  \Sigma_0$ such that each vector $(\kappa^{(j)}_1, \dots,
  \kappa^{(j)}_{\eta+\tau})$ can not be written as the sum of two vectors
  in~$\Sigma_0$ and the sequences $\sum_{i=1}^\eta \kappa^{(j)}_i\alpha_i =
  \sum_{i=1}^\tau \kappa^{(j)}_{\eta+i}\beta_i$ are strictly increasing.

  Now consider the sequence of nonnegative integer~$\{\kappa_1^{(j)}\}$. If
  $\sup_j\kappa^{(j)}_1<\infty$, then there is a constant subsequence
  of~$\{\kappa_1^{(j)}\}$; otherwise $\sup_j\kappa^{(j)}_1=\infty$, then
  there is a strictly increasing subsequence of~$\{\kappa_1^{(j)}\}$.
  Therefore, by taking a subsequence, we can assume that
  $\{\kappa_1^{(j)}\}$ is either constant or strictly increasing. Applying
  the same argument, we can further assume that $\{\kappa_i^{(j)}\}$ is
  either constant or strictly increasing for $1\le i\le\eta+\tau$. But then
  $(\kappa^{(2)}_1-\kappa^{(1)}_1, \dots, \kappa^{(2)}_{\eta+\tau} -
  \kappa^{(1)}_{\eta+\tau}) \in \Sigma_0$ and
  \[\kappa^{(2)}_i = (\kappa^{(2)}_i-\kappa^{(1)}_i)+\kappa^{(1)}_i,\qquad
   1\le i\le \eta+\tau,\]
  contradicting the fact that $(\kappa^{(2)}_1, \dots,
  \kappa^{(2)}_{\eta+\tau})$ can not be written as the sum of two vectors
  in~$\Sigma_0$.
\end{proof}

\begin{proof}[Proof of Lemma~\ref{l:suitde}]
  We begin with applying Lemma~\ref{l:dvsum} by taking the
  $\alpha_1,\dots,\alpha_\eta$ in Lemma~\ref{l:dvsum} just as the given
  $\alpha_1,\dots,\alpha_\eta$ and $\{\beta_1,\dots,\beta_\tau\}=\{P(p)\colon
  P\in\xc P^\circ\}$. According to Lemma~\ref{l:dvsum}, we may assume that
  \begin{equation}\label{eq:<theta}
    \sum_{B\in\xs A}P_B(p)=\sum_{i=1}^na_i\le\theta.
  \end{equation}
  We shall prove the lemma for specified $\xs A$ and $a_1,\dots,a_n$, and
  show that the constant~$K$ depends only on the two sequences $P_B$
  ($B\in\xs A$) and $a_1,\dots,a_n$. Since there are only finitely many such
  two sequences under the assumption~\eqref{eq:<theta}, the lemma follows.

  Now fix $\xs A\subset\xs B^\circ_l$ and $a_1,\dots,a_n$ with
  $a_i\in\{\alpha_1,\dots,\alpha_\eta\}$ such that $\sum_{B\in\xs
  A}P_B(p)=\sum_{i=1}^na_i$. We divide the proof into two steps.

  \medskip

  In the first step, we consider a closely related decomposition
  of the family
  \[\ms_{\xs A,k}=\Bigl\{S\in\ms_{l+k}\colon S(E)\subset\bigsqcup\xs A\Bigr\},
  \]
  where $\xc S_{l+k}$ is defined by~\eqref{eq:Sk}. We will find a integer
  $K'>0$, which depends on the two sequences $P_B$ ($B\in\xs A$) and
  $a_1,\dots,a_n$, such that there is a decomposition
  \begin{equation}\label{eq:4dvsum}
    \ms_{\xs A,K'}=\bigcup_{i=1}^n\xc A_i
  \end{equation}
  satisfying
  \begin{equation}\label{eq:4dvsum'}
    \sum_{S\in\xc A_i}\mu\bigl(S(E)\bigr)=p^la_i
    \quad\text{for $1\le i\le n$}.
  \end{equation}

  By Lemma~\ref{l:Z[p]}(e), there exist an integer $K_1\ge1$ and integers
  $a_{i,\ell}\ge0$ such that
  \begin{equation}\label{eq:Zpnum}
    a_i=p^{K_1}\sum_{\ell=0}^\lambda a_{i,\ell}p^\ell
    \quad \text{for}\ 1\le i\le n.
  \end{equation}
  Here $\lambda$ is as in~\eqref{eq:lambda}. Note that $K_1$ depends on
  $\alpha_1,\dots,\alpha_\eta$ since $a_i\in\{\alpha_1,\dots,\alpha_\eta\}$.
  Write $\bm a_i=\bigl(a_{i,0},\dots,a_{i,\lambda}\bigr)^T$ for $1\le i\le
  n$. Let $\bm\Xi$ be the matrix in~\eqref{eq:Xi} and $\bm
  p=(1,p,\dots,p^\lambda)$. Recall that $\bm\Xi$ is primitive, $p^{-1}$ is
  the Perron-Frobenius eigenvalue of~$\bm\Xi$ and $\bm p$ is the
  corresponding left-hand Perron-Frobenius eigenvector. Consequently,
  by~\eqref{eq:Zpnum},
  \[a_i=p^{K_1}\bm p\bm a_i=p^{K_1+k}\bm p\bm\Xi^k\bm a_i,\quad
  \text{for all $k\ge0$}.\]
  Write
  \[b_{k,\ell} = \card\bigl\{S\in\ms_{\xs A,k}\colon\text{the ratio of $S$ is
    $r^{l+k+\ell}$}\bigr\}\quad\text{for $0\le \ell\le\lambda$},\]
  and $\bm b=(b_{K_1,0},\dots,b_{K_1,\lambda})^T$. Then we have
  \[\bm p\sum_{i=1}^n\bm a_i=p^{-K_1}\sum_{i=1}^na_i=p^{-K_1}\sum_{B\in\xs A}
  P_B(p)=p^{-l-K_1}\sum_{B\in\xs A}\mu(B) =\bm p\bm b.\]
  By Lemma~\ref{l:limMatrix}, as $k\to\infty$,
  \[p^k\bm\Xi^k\bm a_i\to(\bm p\cdot\bm a_i)\bm q\ \text{for}\ 1\le i\le n
  \quad\text{and}\quad p^k\bm\Xi^k\bm b\to(\bm p\cdot\bm b)\bm q.\]
  It follows that there exists an integer~$K_2>0$ such that
  \begin{equation}\label{eq:a<b}
    \bm\Xi^{K_2}\sum_{i=1}^{n-1}\bm a_i<\bm\Xi^{K_2}\bm b.
  \end{equation}
  Let $K'=K_1+K_2$. Note that $K'$ depends on the two sequences $P_B$
  ($B\in\xs A$) and $a_1,\dots,a_n$ since this holds for $K_1$ and $K_2$. We
  shall show that $K'$ has the desired property.

  Write $\bm\Xi^{K_2}\bm a_i=(a'_{i,0},\dots,a'_{i,\lambda})^T$ for $1\le
  i\le n-1$. By the definition of the matrix~$\bm\Xi$, we have
  $\bm\Xi^{K_2}\bm b=(b_{K',0}, \dots, b_{K',\lambda})^T$. It follows
  from~\eqref{eq:a<b} that
  \[a'_{1,\ell}+a'_{2,\ell}+\dots+a'_{n-1,\ell}<b_{K',\ell}\quad
  \text{for $0\le\ell\le\lambda$}.\]
  Note that $a'_{i,\ell}\ge0$ for $1\le i\le n$ and $0\le \ell\le\lambda$
  since $a_{i,\ell}\ge0$. Recall that
  \[b_{K',\ell}= \card\bigl\{S\in\ms_{\xs A,K'}\colon\text{the ratio of $S$ is
    $r^{l+K'+\ell}$}\bigr\}\quad\text{for $0\le \ell\le\lambda$}.\]
  So there exist disjoint subsets $\xc A_1,\dots,\xc A_{n-1}$ of $\ms_{\xs
  A,K'}$ such that
  \[\card\bigl\{S\in\xc A_i \colon \mu\bigl(S(E)\bigr)
  =p^{l+K'+\ell}\bigr\}=a'_{i,\ell}\]
  for $1\le i\le n-1$ and $0\le \ell\le\lambda$. Write $\xc A_n=\ms_{\xs
  A,K'}\setminus\bigcup_{i=1}^{n-1}\xc A_i$. We claim that the
  decomposition
  \[\ms_{\xs A,K'}=\bigcup_{i=1}^n\xc A_i\]
  satisfying
  \[\sum_{S\in\xc A_i}\mu\bigl(S(E)\bigr)=p^la_i
  \quad\text{for $1\le i\le n$}.\]
  In fact, for $1\le i\le n-1$,
  \[\sum_{S\in\xc A_i}\mu\bigl(S(E)\bigr)=p^{l+K'}\bm p
  \bm\Xi^{K_2}\bm a_i=p^{l+K_1}\bm p\bm a_i=p^la_i.\]
  And so
  \begin{multline*}
    \sum_{S\in\xc A_n}\mu\bigl(S(E)\bigr)=\sum_{S\in\ms_{\xs A,K'}}\mu
    \bigl(S(E)\bigr)-\sum_{i=1}^{n-1}\sum_{S\in\xc A_i}\mu\bigl(S(E)\bigr) \\
    =\sum_{B\in\xs A}\mu(B)-p^l\sum_{i=1}^{n-1}a_i=p^la_n.
  \end{multline*}

  \medskip

  In the second step, we will use the decomposition~\eqref{eq:4dvsum} to
  obtain a suitable decomposition. The idea is to approximate
  $\bigcup_{S\in\xc A_i}S(E)$ by $\bigcup_{S\in\xc A_i}\bigcup_{B\in\xs
  B^\circ_k}S(B)$, while the latter is a disjoint union of interior blocks
  (see Remark~\ref{r:inblock}(b)).

  For $1\le i\le n$, by~\eqref{eq:4dvsum'},
  \begin{multline*}
    \sum_{S\in\xc A_i}\biggl(\mu\bigl(S(E)\bigr)-\sum_{B\in\xs B^\circ_k}
    \mu\bigl(S(B)\bigr)\biggr)
    =\sum_{S\in\xc A_i}\sum_{B\in\xs B^\partial_k}
    \mu\bigl(S(B)\bigr)\\
    =\sum_{S\in\xc A_i}\sum_{B\in\xs B^\partial_k}
    \mu\bigl(S(E)\bigr)\cdot\mu(B)
    =p^{l+k}a_i\sum_{B\in\xs B^\partial_k}P_B(p).
  \end{multline*}
  Note that $a_iP_B(p)>0$ and $a_iP_B(p)\in I_\ms$ since $a_i\in I_\ms$. By
  Lemma~\ref{l:idealP}, we can further require that the positive integer $K'$
  in the first step also satisfies the condition that there exist integers
  $\gamma_{i,P_B,P}\ge0$ such that
  \[a_iP_B(p)=p^{K'}\sum_{P\in\xc P^\circ}\gamma_{i,P_B,P}P(p)\]
  for all $1\le i\le n$ and all $B\in\xs B^\partial_k$. Since
  $a_i\in\{\alpha_1,\dots,\alpha_\eta\}$ and $P_B\in\xc P$, such $K'$ depends
  on $\alpha_1,\dots,\alpha_\eta$ and $\ms$ rather than~$k$ or~$\xs
  B^\partial_k$. By above computation
  \begin{equation}\label{eq:addition}
    \sum_{S\in\xc A_i}\biggl(\mu\bigl(S(E)\bigr)-\sum_{B\in\xs B^\circ_k}
    \mu\bigl(S(B)\bigr)\biggr)=p^{l+K'+k}\sum_{P\in\xc P^\circ}
    \sum_{B\in\xs B^\partial_k}\gamma_{i,P_B,P}P(p).
  \end{equation}
  Recall that $\sum_{i=1}^na_i<\theta$, and so $n\le\theta\max_{1\le
  i\le\eta}\alpha_i^{-1}$. Write
  \[\gamma^*=\max_{i,P_B,P} \gamma_{i,P_B,P}.\]
  By Lemma~\ref{l:C(k)} and~\ref{l:CP(k)}, there exists an integer~$K_3>0$
  relying on~$\ms$ and $\alpha_1,\dots,\alpha_\eta$ such that for all
  $P\in\xc P^\circ$,
  \begin{equation}\label{eq:CP(K3)}
    \zeta_P(K_3)\ge\gamma^*\zeta(K_3)\theta\max_{1\le i\le\eta}\alpha_i^{-1}\ge
    n\gamma^*\zeta(K_3)\ge\sum_{i=1}^n\sum_{B\in\xs B^\partial_{K_3}}
    \gamma_{i,P_B,P}.
  \end{equation}
  For $1\le i\le n$ and $P\in\xc P^\circ$, write
  \[\xs A'_i=\bigl\{S(B)\colon S\in\xc A_i,B\in\xs B^\circ_{K_3}\bigr\}
  \quad\text{and}\quad\Upsilon_{i,P}=\sum_{B\in\xs B^\partial_{K_3}}
  \gamma_{i,P_B,P}.\]

  Now we consider $\xs A'_1$, $\xs A'_2$, \dots, $\xs A'_{n-1}$, which are
  families consisting of interior blocks. Equality~\eqref{eq:addition} says
  that, for each $P\in\xc P^\circ$ and each $1\le i\le n-1$, if we can find
  $\Upsilon_{i,P}$ many level-$(l+K'+K_3)$ interior blocks~$B$ such that
  $P_B=P$, and add them to~$\xs A'_i$, then $\mu(\bigsqcup\xs A'_i)$ is just
  equal to $p^la_i$. So we need to find such many interior blocks outside of
  $\xs A'_1$, $\xs A'_2$, \dots, $\xs A'_{n-1}$.

  It follows from \eqref{eq:a<b} that
  \[a'_{1,0}+a'_{2,0}+\dots+a'_{n-1,0}<b_{K',0}.\]
  According to the definition of~$\xc A_i$, the above inequality implies that
  $\xc A_n$ contains at least one $S$, say $S^*$, whose ratio is $r^{l+K'}$.
  By Remark~\ref{r:inblock}(b) and the definition of~$\zeta_P$, we know that, for
  each $P\in\xc P^\circ$, the family $\bigl\{S^*(B)\colon B\in\xs
  B^\circ_{K_3}\bigr\}$ contains $\zeta_P(K_3)$ many level-$(l+K'+K_3)$ interior
  blocks~$B$ such that $P_B=P$. Then inequality~\eqref{eq:CP(K3)} implies
  that there exist disjoint subfamilies $\xs A''_1$, \dots, $\xs A''_{n-1}$
  of $\bigl\{S^*(B)\colon B\in\xs B^\circ_{K_3}\bigr\}$ such that
  \[\card\{B\in\xs A''_i\colon P_B=P\}=\Upsilon_{i,P}\quad
  \text{for $1\le i\le n-1$ and $P\in\xc P^\circ$}.\]
  Since $S^*\in\xc A_n$, the families $\xs A'_1$, \dots, $\xs A'_{n-1}$,
  $\xs A''_1$, \dots, $\xs A''_{n-1}$ are disjoint.

  Notice that, for $B\in\bigcup_{i=1}^{n-1}\xs A'_i$, the level of~$B$ is at
  most $l+K'+K_3+\lambda$; while for all $B\in\bigcup_{i=1}^{n-1}\xs A''_i$
  the level of~$B$ is $l+K'+K_3$. Let $K=K'+K_3+\lambda$, then $K$ depends on
  $\ms$ and $\alpha_1$, \dots, $\alpha_\eta$. Define
  \[\xs A_i=\Bigl\{B\in\xs B_K(\xs A)\colon
    B\subset\bigsqcup\xs A'_i\cup\bigsqcup\xs A''_i\Bigr\}
    \quad\text{for $1\le i\le n-1$},
  \]
  and $\xs A_n=\xs B_K(\xs A)\setminus\bigcup_{i=1}^{n-1}\xs A_i$. Then $\xs
  B_K(\xs A)=\bigcup_{i=1}^n\xs A_i$ is a suitable decomposition, since
  \[\sum_{B\in\xs A_i}\mu(B)=\sum_{B\in\xs A'_i\cup\xs A''_i}\mu(B)
    =\sum_{S\in\xc A_i}S(E)=p^la_i
  \]
  for $1\le i\le n-1$ due to~\eqref{eq:addition}, and so $\sum_{B\in\xs
  A_n}\mu(B)=p^la_n$.
\end{proof}

\section{Interior Blocks and the Whole Set}\label{sec:LipBE}

Fix an $\ms=\{S_1,S_2,\dots,S_N\}\in\TOEC$. For notational convenience, we
denote $E_\ms$, $\mu_\ms$, $r_\ms$ and~$p_\ms$ shortly by~$E$, $\mu$, $r$
and~$p$, respectively. We also use notations as in Definition~\ref{d:ntt}.

The aim of this section is to prove the following proposition.
\begin{prop}\label{p:BlipE}
  We have $B_0\simeq E$ for all interior blocks~$B_0$.
\end{prop}
Let us say something about the proof of Proposition~\ref{p:BlipE}.
Example~\ref{e:NML} implies that, in some cases, it is impossible to
construct the same cylinder structure for $E$ and~$B_0$ under the guidance of
measure linear property. Fortunately, we find that $E$ and $B_0$ have the
same dense island structure. And so Proposition~\ref{p:BlipE} follows from
Lemma~\ref{l:djt}. The difficult is how to define the islands and the
bi-Lipschitz mapping $f_D$ between two islands $D$ and $\tilde f(D)$. The
results in Section~\ref{ssec:BDnum} say that almost all blocks are interior
blocks. So it is natural to define the islands to be the finite unions of
interior blocks. This means that we need consider the bi-Lipschitz mappings
between two finite unions of interior blocks. We do this in
Section~\ref{ssec:LipBB}, and then give the proof of
Proposition~\ref{p:BlipE} in Section~\ref{ssec:LipBE}.

\subsection{The Lipschitz equivalence of interior blocks}\label{ssec:LipBB}

For $A\in\xs B_l$ and $\xs A\subset\xs B_l$ ($l\ge0$), recall that
$\bigsqcup\xs A=\bigcup_{A\in\xs A}A$,
\[\xs B_k(A)=\bigl\{B\in\xs B_{l+k}\colon B\subset A\bigr\}
\quad\text{and}\quad
\xs B_k(\xs A)=\left\{B\in\xs B_{l+k}\colon B\subset\bigsqcup\xs A\right\}.\]

\begin{defn}\label{d:strc}
  Let $\xs A$ be a nonempty family of interior blocks. We say that $\xs A$
  has $(A;l,k)$-structure if $A\in\xs B_l$ is a level-$l$ block and $\xs
  A\subset\xs B_k^\circ(A)$.
\end{defn}
\begin{lem}\label{l:blockslip}
  For each $k_0\ge1$, there exists a constant~$L$ depending only on~$k_0$
  with the following property. Let $\xs A$ and $\xs A'$ be two nonempty
  families of interior blocks such that $\xs A$ has $(A;l_1,k_1)$-structure
  and $\xs A'$ has $(A';l_2,k_2)$-structure, where $\max(k_1,k_2)\le k_0$.
  Then there is a bi-Lipschitz mapping $f\colon\bigsqcup\xs A\to\bigsqcup\xs
  A'$ such that
  \[L^{-1}r^{-l_1}|x-y|\le r^{-l_2}|f(x)-f(y)|\le Lr^{-l_1}|x-y|
  \quad\text{for}\ x,y\in\bigsqcup\xs A.\]
  In particular, we have $\bigsqcup\xs A\simeq\bigsqcup\xs A'$ for any two
  nonempty families of interior blocks.
\end{lem}

The proof of Lemma~\ref{l:blockslip} is based on two special cases,
Lemma~\ref{l:4blockslip} and~\ref{l:4blockslip'}.

\begin{lem}\label{l:4blockslip}
  For every $k_0\ge1$, there exists a constant~$L$ depending only on~$k_0$
  with the following property. Let $\xs A$ and $\xs A'$ be two nonempty
  families of interior blocks such that $\xs A$ has $(A;l_1,k_1)$-structure
  and $\xs A'$ has $(A';l_2,k_2)$-structure, where $\max(k_1,k_2)\le k_0$. If
  \[\sum_{B\in\xs A}P_B(p)=\sum_{B\in\xs A'}P_B(p),\]
  then there is a bi-Lipschitz mapping $f\colon\bigsqcup\xs
  A\to\bigsqcup\xs A'$ such that
  \[L^{-1}r^{-l_1}|x-y|\le r^{-l_2}|f(x)-f(y)|\le Lr^{-l_1}|x-y|
  \quad\text{for}\ x,y\in\bigsqcup\xs A.\]
\end{lem}
\begin{proof}
  Let $F=\bigsqcup\xs A$ and $F'=\bigsqcup\xs A'$. We shall show that $F$ and
  $F'$ have the same cylinder structure, then this lemma follows from
  Lemma~\ref{l:cylin}. For this, we make use of Lemma~\ref{l:suitde} to
  define a cylinder mapping satisfying~\eqref{eq:ML}. Take the positive
  numbers $\alpha_1,\dots,\alpha_\eta$ in Lemma~\ref{l:suitde} to be
  $\{P(p)\colon P\in\xc P^\circ\}$ and $K$ the corresponding integer
  constant. The cylinder families $\xs C_k$ and $\xs C'_k$ are defined as
  follows. Define $\xs C_1=\xs A$ and
  \[\begin{cases}
    \xs C_k=\xs B_{(k-1)K}(\xs A)&\text{for $k$ is odd};\\
    \xs C'_k=\xs B_{(k-1)K}(\xs A')&\text{for $k$ is even}.
  \end{cases}\]
\begin{center}
  \begin{tabular}{*8c}
    \hline
    & $k=1$ & $k=2$ & $k=3$ & $k=4$ & $k=5$ & $k=6$ & \dots
    \rule[-1ex]{0pt}{3.5ex} \\
    \hline
    $\xs C_k$:  & $\xs A$ &   & $\xs B_{2K}(\xs A)$ &
      & $\xs B_{4K}(\xs A)$ &   & \dots \rule[-1ex]{0pt}{3.5ex} \\
    $\xs C'_k$: &   & $\xs B_K(\xs A')$ &   & $\xs B_{3K}(\xs A')$
      &   & $\xs B_{5K}(\xs A')$ & \dots \rule[-1ex]{0pt}{3.5ex} \\
    \hline
  \end{tabular}
\end{center}
  It remains to define $\xs C_k$ for $k$ is even, $\xs C'_k$ for $k$ is odd
  and the cylinder mapping~$\tilde f$.

  We begin with the definition of~$\xs C'_1$ by making use of
  Lemma~\ref{l:suitde}. Since $P_B\in\xc P^\circ$ for all $B\in\xs C_1=\xs A$
  and
  \[\sum_{B\in\xs C_1}P_B(p)=\sum_{B\in\xs A}P_B(p)=\sum_{B\in\xs A'}P_B(p),\]
  we know that the sequence $\{P_B(p)\colon B\in\xs C_1\}$ are $(\xs
  A';\alpha_1,\dots,\alpha_\eta)$-suitable by Definition~\ref{d:suit},
  where $\{\alpha_1\dots,\alpha_\eta\}=\{P(p)\colon P\in\xc P^\circ\}$. Since
  $\xs A'\subset\xs B_{k_2}^\circ(A')\subset\xs B_{l_2+k_2}^\circ$, by
  Lemma~\ref{l:suitde}, for positive numbers $\{p^{l_2+k_2}P_B(p)\colon
  B\in\xs C_1\}$, there is a suitable decomposition
  \[\xs B_K(\xs A')=\bigcup_{B\in\xs C_1}\xs C_B\]
  such that
  \[\sum_{B'\in\xs C_B}\mu(B')=p^{l_2+k_2}P_B(p)\quad
  \text{for all $B\in\xs C_1=\xs A$}.\]
  Define
  \[\xs C'_1=\left\{\bigsqcup\xs C_B\colon B\in\xs C_1\right\},\]
  and $\tilde f\colon\xs C_1\to\xs C'_1$ by $\tilde f(B)=\bigsqcup\xs C_B$.
  It is easy to see that $\xs C'_1\prec\xs C'_2=\xs B_K(\xs A')$. We also
  have that
  \[p^{-l_1-k_1}\mu(B)=p^{-l_2-k_2}\mu(\tilde f(B))\quad
  \text{for all $B\in\xs C_1$},\]
  since $p^{-l_2-k_2}\mu(\tilde f(B))=p^{-l_2-k_2}\sum_{B'\in\xs C_B}\mu(B')
  =P_B(p)=p^{-l_1-k_1}\mu(B)$.

  Now suppose that the cylinder families $\xs C_1$, \dots, $\xs C_{k-1}$, $\xs
  C'_1$, \dots, $\xs C'_{k-1}$ and the cylinder mapping $\tilde f$ have
  been defined such that $\tilde f$ maps $\xs C_j$ onto $\xs C'_j$ for
  $1\le j\le k-1$ and
  \begin{equation}\label{eq:muC}
    p^{-l_1-k_1}\mu(C)=p^{-l_2-k_2}\mu(\tilde f(C))\quad
   \text{for all $C\in\bigcup_{j=1}^{k-1}\xs C_j$}.
  \end{equation}
  We shall define $\xs C_k$, $\xs C'_k$ and $\tilde f\colon\xs C_k\to\xs
  C'_k$. Suppose without loss of generality that $k$ is even, then $\xs
  C'_k=\xs B_{(k-1)K}(\xs A')$. We consider the suitable decomposition of $\xs
  B_K(B_0)$ for each $B_0\in\xs C_{k-1}=\xs B_{(k-2)K}(\xs A)$.
  By~\eqref{eq:muC},
  \begin{multline*}
    \sum_{\substack{B'\subset\tilde f(B_0)\\B'\in\xs C'_k}}P_{B'}(p)
    =p^{-l_2-k_2-(k-1)K}\mu(\tilde f(B_0))\\
    =p^{-l_1-k_1-(k-1)K}\mu(B_0)
    =\sum_{B\in\xs B_K(B_0)}P_B(p).
  \end{multline*}
  And so the sequence $\{P_{B'}(p)\colon B'\subset\tilde f(B_0), B'\in\xs
  C'_k\}$ are $\{\xs B_K(B_0);\alpha_1,\dots,\alpha_\eta\}$-suitable by
  Definition~\ref{d:suit}, where
  $\{\alpha_1\dots,\alpha_\eta\}=\{P(p)\colon P\in\xc P^\circ\}$. By
  Lemma~\ref{l:suitde}, for positive numbers
  \[\Bigl\{p^{l_1+k_1+(k-1)K}P_{B'}(p)\colon B'\subset\tilde f(B_0),
    B'\in\xs C'_k\Bigr\},\]
  there is a suitable decomposition
  \begin{equation}\label{eq:cylin}
    \xs B_K(\xs B_K(B_0))=\xs B_{2K}(B_0)=
    \bigcup_{\substack{B'\subset\tilde f(B_0)\\B'\in\xs C'_k}}\xs C_{B'}
  \end{equation}
  such that
  \[\sum_{B\in\xs C_{B'}}\mu(B)=p^{l_1+k_1+(k-1)K}P_{B'}(p)\quad
  \text{for all $B'\in\xs C'_k$ and $B'\subset\tilde f(B_0)$}.\]
  Indeed, we obtain $\xs C_{B'}$ for all $B'\in\xs C'_k$ by above argument
  since for every $B'\in\xs C'_k$, there is a unique $B_0\in\xs C_{k-1}$ such
  that $B'\subset\tilde f(B_0)$. Then we define
  \[\xs C_k=\left\{\bigsqcup\xs C_{B'}\colon B'\in\xs C'_k\right\}\]
  and $\tilde f\colon\xs C_k\to\xs C'_k$ by $\tilde f\left(\bigsqcup\xs
  C_{B'}\right)=B'$. By~\eqref{eq:cylin}, we know that
  \[\xs B_{(k-2)K}(\xs A)=\xs C_{k-1}\prec\xs C_k\prec\xs C_{k+1}
  =\xs B_{kK}(\xs A).\]
  We also have $p^{-l_1-k_1}\mu(C)=p^{-l_2-k_2}\mu(\tilde f(C))$ for all
  $C\in\xs C_k$. If $k$ is odd, we can define $\xs C_k$, $\xs C'_k$ and
  $\tilde f\colon\xs C_k\to\xs C'_k$ by a similar argument. Thus, by
  induction on~$k$, we finally obtain all the cylinder families $\xs C_k$,
  $\xs C'_k$ and the cylinder mapping~$\tilde f$.

  To prove $F=\bigsqcup\xs A$ and $F'=\bigsqcup\xs A'$ have the same cylinder
  structure, it remains to compute the constants $\varrho$ and $\iota$. Since
  $F\subset A\in\xs B_{l_1}$ and $F$ contains at least one level-$(l_1+k_1)$
  interior block, by Remark~\ref{r:diameter}, we have
  \[\varpi^{-1}r^{l_1+k_1}|E|\le|F|\le\varpi r^{l_1}|E|.\]
  Let $C,C_1,C_2\in\xs C_k$, where $C_1$ and $C_2$ are distinct. If $k$ is
  odd, then $\xs C_k=\xs B_{(k-1)K}(\xs A)\subset\xs
  B_{l_1+k_1+(k-1)K}^\circ$, by Remark~\ref{r:diameter} and the definition of
  blocks (Definition~\ref{d:block}), we have
  \begin{align*}
    \varpi^{-1}r^{l_1+k_1+(k-1)K}|E|&\le|C|\le\varpi r^{l_1+k_1+(k-1)K}|E|;\\
    r^{l_1+k_1+(k-1)K}|E|&\le\dist(C_1,C_2).
  \end{align*}
  If $k$ is even, then by the definition of $\xs C_k$, we know that $\xs
  B_{(k-2)K}(\xs A)=\xs C_{k-1}\prec\xs C_k\prec\xs C_{k+1}=\xs B_{kK}(\xs
  A)$. And so
  \begin{align*}
    \varpi^{-1}r^{l_1+k_1+kK}|E|&\le|C|\le\varpi r^{l_1+k_1+(k-2)K}|E|;\\
    r^{l_1+k_1+kK}|E|&\le\dist(C_1,C_2).
  \end{align*}
  As a summary, $F$ has the $(\varrho,\iota)$-cylinder structure for
  $\varrho=r^K$ and $\iota=\varpi^2r^{-k_0-2K}$, (recall that
  $k_0\ge\max(k_1,k_2)$). A similar argument shows that $F'$ also has the
  $(\varrho,\iota)$-cylinder structure for $\varrho=r^K$ and
  $\iota=\varpi^2r^{-k_0-2K}$. Then by the cylinder mapping~$\tilde f$, we
  know that $F$ and $F'$ have the same $(\varrho,\iota)$-cylinder structure.
  Therefore, by Lemma~\ref{l:cylin}, there is a bi-Lipschitz mapping~$f$ such
  that
  \[\varrho\iota^{-2}|x-y|/|F|\le|f(x)-f(y)|/|F'|\le
  \varrho^{-1}\iota^2|x-y|/|F|\quad\text{for distinct $x,y\in F$}.\]
  It is easy to check that
  \[\varpi^{-2}r^{l_1-l_2+k_0}\le|F|/|F'|\le\varpi^2r^{l_1-l_2-k_0}.\]
  Therefore, we have
  \[L^{-1}r^{-l_1}|x-y|\le r^{-l_2}|f(x)-f(y)|\le Lr^{-l_1}|x-y|
    \quad\text{for}\ x,y\in F=\bigsqcup\xs A,\]
  where $L=\varrho^{-1}\iota^2\varpi^2r^{-k_0}=\varpi^6r^{-3k_0-5K}$ depends
  only on~$k_0$ since $\varpi$ and $K$ are all constants related to the
  IFS~$\ms$.
\end{proof}

\begin{lem}\label{l:4blockslip'}
  For each $k_0\ge1$, there exists a constant~$L$ depending only on~$k$ with
  the following property. Let $\xs A$ and $\xs A'$ be two nonempty families
  of interior blocks such that $\xs A$ has $(A;l_1,k_1)$-structure and $\xs
  A'$ has $(A';l_2,k_2)$-structure, where $\max(k_1,k_2)\le k_0$. If all
  interior blocks in~$\xs A\cup\xs A'$ have the same measure polynomial, then
  there is a bi-Lipschitz mapping $f\colon\bigsqcup\xs A\to\bigsqcup\xs A'$
  such that
  \[L^{-1}r^{-l_1}|x-y|\le r^{-l_2}|f(x)-f(y)|\le Lr^{-l_1}|x-y|
  \quad\text{for}\ x,y\in\bigsqcup\xs A.\]
\end{lem}

\begin{proof}
  It suffices to prove the lemma in the case that $\xs A$ or $\xs A'$
  consists of only one interior block. If this is true, for general $\xs A$
  and $\xs A'$, pick $B\in\xs A$, then by the assumption, there are constant
  $L$ and bi-Lipschitz mappings $f_1\colon B\to\bigsqcup\xs A$, $f_2\colon
  B\to\bigsqcup\xs A'$ satisfying the condition in the lemma. Thus $f_2\circ
  f_1^{-1}\colon\bigsqcup\xs A\to\bigsqcup\xs A'$ satisfies
  \[L^{-2}r^{-l_1}|x-y|\le r^{-l_2}|f_2\circ f_1^{-1}(x)-f_1\circ f_0^{-1}(y)|
  \le L^2r^{-l_1}|x-y|.\]
  And so the lemma holds for the general case. Thus we need only to prove the
  case that $\xs A=\{B_1\}$, $\xs A'=\{B'_1,\dots,B'_m\}$ and
  $P_{B_1}=P_{B'_1}\dots=P_{B'_m}=P_0$ for some $m>1$. (If $m=1$, this
  follows from Lemma~\ref{l:4blockslip}.)

  By Lemma~\ref{l:CB(k)}, there exists an integer~$\ell$ such that
  \[\zeta^\circ(\ell)\ge p^{-k_0}\frac{\max_{P\in\xc P}P(p)}
  {\min_{P\in\xc P}P(p)}.\]
  Since
  \[mp^{l_2+k_0}P_0(p)\le mp^{l_2+k_2}P_0(p)=\sum_{i=1}^m\mu(B'_i)
  \le\mu(A')\le p^{l_2}P_{A'}(p),\]
  we have
  \[m\le p^{-k_0}\frac{\max_{P\in\xc P}P(p)}{\min_{P\in\xc P}P(p)}
  \le\zeta^\circ(\ell).\]
  This means that, for every $B\in\xs B$ and every $P\in\xc P^\circ$, $\xs
  B_\ell(B)$ contains at least $m$ interior blocks whose measure polynomial
  is~$P$. Notice that the constant $\ell$ depends on~$k_0$ rather than~$m$.

  Let $F=B_0$ and $F'=\bigcup_{i=1}^m B'_i$. We shall show that $F$ and $F'$
  have the same dense $\iota$-island structure. Then the lemma follows from
  Lemma~\ref{l:djt}. For this, we need to define $\iota$-island families $\xs
  D$, $\xs D'$, the island mapping~$\tilde f$ and bi-Lipschitz mappings $f_D$
  for each~$D$.

  By the property of~$\ell$, we have that both $\xs B_\ell(B_1)$ and $\xs
  B_\ell(B'_1)$ contain at least $m$ interior blocks, say, $B^{(1)}_1$,
  \dots, $B^{(1)}_m$ and $B'^{(1)}_1$, \dots, $B'^{(1)}_m$, respectively,
  whose measure polynomial are all~$P_0$. By induction, suppose that
  $B^{(j)}_1$, \dots, $B^{(j)}_m$ and $B'^{(j)}_1$, \dots, $B'^{(j)}_m$ have
  been defined for $j=1,\dots,k-1$. We define $B^{(k)}_1$, \dots, $B^{(k)}_m$
  and $B'^{(k)}_1$, \dots, $B'^{(k)}_m$ to be $m$ interior blocks in~$\xs
  B_\ell(B^{(k-1)}_1)$ and $\xs B_\ell(B'^{(k-1)}_1)$, respectively, whose
  measure polynomial are all~$P_0$. For convenience, we also write
  \[B^{(0)}_1=B_1\quad\text{and}\quad
  B'^{(0)}_i=B'_i\ \text{for $i=1,\dots,m$}.\]
  For $k\ge1$, define
  \begin{gather*}
    \xs A^{(k)}_1=\{B^{(k)}_2,\dots,B^{(k)}_m\},\quad
    \xs A^{(k)}_2=\xs B_\ell(B^{(k-1)}_1)\setminus
    \{B^{(k)}_1,\dots,B^{(k)}_m\}, \\
    \xs A'^{(k)}_1=\{B'^{(k-1)}_2,\dots,B'^{(k-1)}_m\},\quad
    \xs A'^{(k)}_2=\xs B_\ell(B'^{(k-1)}_1)\setminus
    \{B'^{(k)}_1,\dots,B'^{(k)}_m\},
  \end{gather*}
  and
  \begin{gather*}
    D^{(k)}_1=\bigsqcup\xs A^{(k)}_1,\quad D^{(k)}_2=\bigsqcup\xs A^{(k)}_2,
    \quad \xs D=\bigcup_{k=1}^\infty\left\{D^{(k)}_1,D^{(k)}_2\right\},\\
    D'^{(k)}_1=\bigsqcup\xs A'^{(k)}_1,\quad
    D'^{(k)}_2=\bigsqcup\xs A'^{(k)}_2,\quad
    \xs D'=\bigcup_{k=1}^\infty\left\{D'^{(k)}_1,D'^{(k)}_2\right\}.
  \end{gather*}

  Observe that there are $x$ and $x'$ such that
  \[F\setminus\bigsqcup\xs D=\bigcap_{k=1}^\infty B^{(k)}_1=\{x\}
  \quad\text{and}\quad F'\setminus\bigsqcup\xs D'=\bigcap_{k=1}^\infty
  B'^{(k)}_1=\{x'\},\]
  so $\bigsqcup\xs D$ and $\bigsqcup\xs D'$ are both dense in $F$ and $F'$,
  respectively. For $k\ge1$, $\xs A^{(k)}_1,\xs A^{(k)}_2\subset\xs
  B_\ell(B^{(k-1)}_1)\subset\xs B^\circ_{l_1+k_1+k\ell}$; $\xs
  A'^{(k+1)}_1,\xs A'^{(k)}_2\subset\xs B_\ell(B'^{(k-1)}_1)\subset\xs
  B^\circ_{l_2+k_2+k\ell}$ and $\xs A'^{(1)}_1\subset\xs
  B^\circ_{k_2}(A')\subset\xs B^\circ_{l_2+k_2}$. By Remark~\ref{r:diameter},
  for $k\ge1$
  \begin{gather*}
    \varpi^{-1}r^{l_1+k_1+k\ell}|E|\le|D^{(k)}_1|,|D^{(k)}_2|
    \le\varpi r^{l_1+k_1+(k-1)\ell}|E|, \\
    \varpi^{-1}r^{l_2+k_2+k\ell}|E|\le|D'^{(k+1)}_1|,|D'^{(k)}_2|
    \le\varpi r^{l_2+k_2+(k-1)\ell}|E|,\\
    \varpi^{-1}r^{l_2+k_2}|E|\le|D'^{(1)}_1|\le\varpi r^{l_2}|E|.
  \end{gather*}
  By the definition of blocks (Definition~\ref{d:block}), for $k\ge1$,
  \begin{gather*}
    r^{l_1+k_1+k\ell}|E|\le\dist(D^{(k)}_i,F\setminus D^{(k)}_i),
    \quad i=1,2; \\
    r^{l_2+k_2+(k-1)\ell}|E|\le\dist(D'^{(k)}_1,F'\setminus D'^{(k)}_1);\\
    r^{l_2+k_2+k\ell}|E|\le\dist(D'^{(k)}_2,F'\setminus D'^{(k)}_2).
  \end{gather*}
  As a summary, we see that both $F$ and $F'$ have dense $\iota$-island
  structure (Definition~\ref{d:dnsilnd}) for $\iota=\varpi r^{-k_0-2\ell}$,
  (recall that $\max(k_1,k_2)\le k_0$).

  Now define the island mapping $\tilde f\colon\xs D\to\xs D'$ by $\tilde
  f(D^{(k)}_i)=D'^{(k)}_i$ for $k\ge1$ and $i=1,2$. To show that $F$ and $F'$
  have the same dense $\iota$-island structure, we shall verify the
  conditions of Definition~\ref{d:smdi}.

  For Condition~(i), note that $D^{(k)}_i,D^{(k')}_j\subset B^{(k-1)}_1$ for
  $1\le k\le k'$; $D'^{(k)}_i,D'^{(k')}_j\subset B'^{(k-2)}_1$ for $2\le k\le
  k'$ and $D'^{(1)}_i,D'^{(k')}_j\subset A'$. Together with the definition of
  blocks (Definition~\ref{d:block}), we have
  \begin{gather*}
    r^{l_1+k_1+k\ell}|E|\le\dist(D^{(k)}_i,D^{(k')}_j)
    \le\varpi r^{l_1+k_1+(k-1)\ell}|E|
    \quad\text{for $k\ge1$}, \\
    r^{l_2+k_2+k\ell}|E|\le\dist(D'^{(k)}_i,D'^{(k')}_j)
    \le\varpi r^{l_2+k_2+(k-2)\ell}|E|
    \quad\text{for $k\ge2$},\\
    r^{l_2+k_2+\ell}|E|\le\dist(D^{(1)}_i,D^{(k')}_j)
    \le\varpi r^{l_2}|E|.
  \end{gather*}
  Since $F=B_0\in\xs B^\circ_{l_1+k_1}$ and $F'=\bigcup_{i=1}B'_i\subset
  A'\in\xs B_{l_2}$, where $B'_i\in\xs B^\circ_{l_2+k_2}$, by
  Remark~\ref{r:diameter},
  \begin{gather*}
    \varpi^{-1}r^{l_1+k_1}|E|\le|F|\le\varpi r^{l_1+k_1}|E|, \\
    \varpi^{-1}r^{l_2+k_2}|E|\le|F'|\le\varpi r^{l_2}|E|.
  \end{gather*}
  By the definition of $\tilde f$, for distinct $D_1,D_2\in\xs D$,
  \[L_1^{-1}\dist(D_1,D_2)/|F|\le\dist(\tilde f(D_1),\tilde f(D_2))/|F'|
  \le L_1\dist(D_1,D_2)/|F|,\]
  where $L_1=\varpi^3r^{-k_0-2\ell}$.

  For Condition~(ii) of Definition~\ref{d:smdi}, we use
  Lemma~\ref{l:4blockslip} to obtain~$f_D$ for each $D$. Observe that, for
  $k\ge1$,
  \begin{gather*}
    \sum_{B\in\xs A^{(k)}_1}P_B(p)=\sum_{B\in\xs A'^{(k)}_1}
     P_B(p)=(m-1)P_0(p), \\
    \sum_{B\in\xs A^{(k)}_2}P_B(p)=\sum_{B\in\xs A'^{(k)}_2}
     P_B(p)=(p^{-\ell}-(m-1))P_0(p);
  \end{gather*}
  and $\xs A^{(k)}_1$, $\xs A^{(k)}_2$ have
  $(B^{(k-1)}_1;l_1+k_1+(k-1)\ell,\ell)$-structure; $\xs A'^{(k+1)}_1$, $\xs
  A'^{(k)}_2$ have $(B'^{(k-1)}_1;l_2+k_2+(k-1)\ell,\ell)$-structure for
  $k\ge1$; $\xs A'^{(1)}_1$ has $(A';l_2,k_2)$-structure. By
  Lemma~\ref{l:4blockslip}, there are $L_2>1$ and bi-Lipschitz mappings~$f_D$
  of~$D$ onto~$\tilde f(D)$ for each $D\in\xs D$, such that
  \[L_2^{-1}r^{-l_1}|x-y|\le r^{-l_2}|f_D(x)-f_D(y)|
    \le L_2r^{-l_1}|x-y|\quad\text{for $x,y\in D$}.\]
  Here $L_2$ depends on~$\ell$ and~$k_0$, so finally depends on~$k_0$. Note
  that
  \begin{equation}\label{eq:F/F'}
    \varpi^{-2}r^{l_1-l_2+k_0}\le|F|/|F'|\le\varpi^2r^{l_1-l_2-k_0}.
  \end{equation}
  This means that $f_D$ satisfy the Condition~(ii) of Definition~\ref{d:smdi}
  for $L_3=L_2\varpi^2r^{-k_0}$. Therefore, $f_D$ and $\tilde
  L=\max(L_1,L_3)$ satisfy the conditions of Definition~\ref{d:smdi}. It
  follows that $F$ and $F'$ have the same dense $\iota$-island structure with
  $\iota=\varpi r^{-k_0-2\ell}$.

  Finally, by Lemma~\ref{l:djt}, there is a bi-Lipschitz mapping~$f$ of~$F$
  onto~$F'$ such that
  \begin{equation*}
    L_4^{-1}\rho(x,y)/|F|\le\rho(f(x),f(y))/|F'|\le L_4\rho(x,y)/|F|
    \quad\text{for $x,y\in F$}.
  \end{equation*}
  Here $L_4=(2\iota+1)\tilde L$ depends on~$k_0$. Together with
  inequality~\eqref{eq:F/F'}, the lemma follows.
\end{proof}

\begin{proof}[Proof of Lemma~\ref{l:blockslip}]
  By Lemma~\ref{l:CB(k)}, there is an integer~$\ell$ depends only on the
  IFS~$\ms$ such that for all $B\in\xs B$ and all $P\in\xc P^\circ$, the
  family $\xs B_\ell(B)$ contain at least one block whose measure polynomial
  is~$P$.

  Now for each $P\in\xc P^\circ$, write
  \[\xs C_P=\left\{B\in\xs B_\ell(\xs A)\colon P_B=P\right\}
  \quad\text{and}\quad
  \xs D_P=\left\{B\in\xs B_\ell(\xs A')\colon P_B=P\right\}\]
  Then $\xs C_p,\xs D_P\ne\emptyset$ and
  \[\bigsqcup\xs A=\bigcup_{P\in\xc P^\circ}\bigsqcup\xs C_P,\quad
  \bigsqcup\xs A'=\bigcup_{P\in\xc P^\circ}\bigsqcup\xs D_P.\]
  By Lemma~\ref{l:4blockslip'}, for each $P\in\xc P^\circ$, there is a
  bi-Lipschitz mapping $f_P\colon\bigsqcup\xs C_P\to\bigsqcup\xs D_P$ such
  that
  \begin{equation}\label{eq:|x-y|fP}
    L_1^{-1}r^{-l_1}|x-y|\le r^{-l_2}|f_P(x)-f_P(y)|\le L_1r^{-l_1}|x-y|
    \quad\text{for $x,y\in\bigsqcup\xs C_P$},
  \end{equation}
  where constant $L_1$ depends on $k_0+\ell$, so finally depends only
  on~$k_0$ and the IFS~$\ms$.

  Let $f\colon\bigsqcup\xs A\to\bigsqcup\xs A'$ be the bijection such that
  the restriction of~$f$ to $\bigsqcup\xs C_P$ is~$f_P$ for every $P\in\xc
  P^\circ$. We will show that $f$ is the desired bi-Lipschitz mapping. Let
  $x,y$ be two distinct points in~$\bigsqcup\xs A$, there are two cases to
  consider.

  \textit{Case~1}. $x,y\in\bigsqcup\xs C_P$ for some $P\in\xc P^\circ$. Then
  $f$ and $x,y$ satisfy the inequality~\eqref{eq:|x-y|fP}.

  \textit{Case~2}. $\{x,y\}\not\subset\xs C_P$ for all $P\in\xc P^\circ$. In
  this case, there are distinct $B_x,B_y\in\xs B_\ell(\xs A)$ such that $x\in
  B_x$ and $y\in B_y$. By the definition of~$f$, there are distinct
  $B'_x,B'_y\in\xs B_\ell(\xs A')$ such that $f(x)\in B'_x$ and $f(y)\in
  B'_y$. Therefore, by Remark~\ref{r:diameter},
  \begin{gather*}
    r^{l_1+k_1+\ell}|E|\le\dist(B_x,B_y)\le|x-y|\le|A|\le r^{l_1}\varpi|E|, \\
    r^{l_2+k_2+\ell}|E|\le\dist(B'_x,B'_y)\le|f(x)-f(y)|\le|A'|
    \le r^{l_2}\varpi|E|.
  \end{gather*}

  It follows from Case~1 and~2 that
  \[L^{-1}r^{-l_1}|x-y|\le r^{-l_2}|f(x)-f(y)|\le Lr^{-l_1}|x-y|,\]
  where constant $L=\max(L_1,\varpi r^{-(k_0+\ell)})$ depends only on~$k_0$.
\end{proof}

\subsection{Proof of Proposition~\ref{p:BlipE}}\label{ssec:LipBE}

Let $F=E$ and $F'=B_0$. Suppose that $B_0$ is a level-$l$ interior block. We
shall show that $F$ and $F'$ have the same dense $\iota$-island structure.
Then Proposition~\ref{p:BlipE} follows from Lemma~\ref{l:djt}. For this, we
need to define $\iota$-island families $\xs D$, $\xs D'$, the island
mapping~$\tilde f$ and bi-Lipschitz mappings $f_D$ for each~$D$. The key
point behind our construction is that almost all blocks are interior blocks,
see Lemma~\ref{l:C(k)}, \ref{l:CP(k)} and~\ref{l:CB(k)}.

By Lemma~\ref{l:CB(k)}, there exists an integer $\ell>0$ such that
\begin{equation}\label{eq:bd<in}
  \zeta^\partial(\ell)<\min_{B\in\xs B}\card\xs B_\ell(B).
\end{equation}
Let $\xs A_0=\{E\}$ and $\xs A_k=\xs B^\partial_{k\ell}$ for $k\ge1$. We
first define a injection
\[\Gamma\colon\bigcup_{k=0}^\infty\xs A_k\to
\bigcup_{k=0}^\infty\xs B_{k\ell}(B_0)\]
such that
\begin{enumerate}[(a)]
    \item $\Gamma$ maps $\xs A_k$ into $\xs B_{k\ell}(B_0)$ for all
        $k\ge1$;
    \item for $A_1\in\xs A_{k_1}$ and $A_2\in\xs A_{k_2}$, where $k_1\le
        k_2$, we have $\Gamma(A_1)\supset\Gamma(A_2)$ if and only if
        $A_1\supset A_2$.
  \end{enumerate}
We do this by induction on~$k$. When $k=0$, define $\Gamma(E)=B_0$. Now
suppose that $\Gamma$ have been defined on~$\xs A_k$. By
Remark~\ref{r:inblock}(a), $\xs A_{k+1}=\bigcup_{A\in\xs A_k}\xs
B^\partial_\ell(A)$. So we need only to define $\Gamma$ on each $\xs
B^\partial_\ell(A)$. Suppose that $\card\xs B^\partial_\ell(A)>0$, otherwise
there is nothing more to do. By~\eqref{eq:bd<in}, we have
\[\card\xs B^\partial_\ell(A)<\card\xs B_\ell\bigr(\Gamma(A)\bigl)\quad
\text{for all $A\in\xs A_k$}.\]
So there is a injection $\Gamma$ on $\xs B^\partial_\ell(A)$ such that
$\Gamma(B)\in\xs B_\ell\bigr(\Gamma(A)\bigl)\subset\xs B_{(k+1)\ell}(B_0)$
for all $B\in\xs B^\partial_\ell(A)$. Thus, we have finished the definition
of~$\Gamma$ on~$\xs A_{k+1}$. It is easy to check Condition~(a) and~(b).

Then for $A\in\bigcup_{k=0}^\infty\xs A_k$, define
\[\xs C_A=\xs B^\circ_\ell(A)\quad\text{and}\quad
\xs C'_A=\bigl\{B'\in\xs B_\ell\bigl(\Gamma(A)\bigr)\colon B'\ne\Gamma(B)\
\text{for all $B\in\xs B^\partial_\ell(A)$}\bigr\}.\]
By~\eqref{eq:bd<in}, we have $\xs C_A,\xs C'_A\ne\emptyset$ for all
$A\in\bigcup_{k=0}^\infty\xs A_k$. Let
\[D_A=\bigsqcup\xs C_A\quad\text{and}\quad D'_A=\bigsqcup\xs C'_A.\]
Define
\[\xs D=\biggl\{D_A\colon A\in\bigcup_{k=0}^\infty\xs A_k\biggr\}
\quad\text{and}\quad
\xs D'=\biggl\{D'_A\colon A\in\bigcup_{k=0}^\infty\xs A_k\biggr\}.\]

We shall show that both $F$ and $F'$ have the dense $\iota$-island structure
(Definition~\ref{d:dnsilnd}). First, it is easy to see that $\bigsqcup\xs D$
and $\bigsqcup\xs D'$ are both dense in $F=E$ and $F'=B_0$, respectively.
Second, for $D_A\in\xs D$ and $D'_A\in\xs D'$ with $A\in\xs A_k=\xs
B^\partial_{k\ell}$, we have $\xs C_A$ has $(A;k\ell,\ell)$-structure and
$\xs C'_A$ has $(\Gamma(A);k\ell+l,\ell)$-structure. So by
Remark~\ref{r:diameter},
\begin{gather*}
  \varpi^{-1}r^{(k+1)\ell}|E|\le|D_A|\le\varpi r^{k\ell}|E|, \\
  \varpi^{-1}r^{(k+1)\ell+l}|E|\le|D'_A|\le\varpi r^{k\ell+l}|E|.
\end{gather*}
By the definition of blocks (Definition~\ref{d:block}),
\begin{gather*}
  \dist(D_A,F\setminus D_A)\ge r^{(k+1)\ell}|E|, \\
  \dist(D'_A,F'\setminus D'_A)\ge r^{(k+1)\ell+l}|E|.
\end{gather*}
It follows that both $F$ and $F'$ have the dense $\iota$-island structure for
$\iota=\varpi r^{-\ell}$.

Now define the island mapping $\tilde f\colon\xs D\to\xs D'$ by $\tilde
f(D_A)=D'_A$ for $A\in\bigcup_{k=0}^\infty\xs A_k$. To show that $F$ and $F'$
have the same dense $\iota$-island structure, we shall verify the conditions
of Definition~\ref{d:smdi}.

For Condition~(i), we consider $D_{A_1}$ and $D_{A_2}$ for distinct
$A_1\in\xs A_{k_1},A_2\in\xs A_{k_2}$, where $k_1\le k_2$. There are two
cases to consider.

\textit{Case~1}. $A_1\supset A_2$. By the properties of~$\Gamma$, we have
$\Gamma(A_1)\in\xs B_{k_1\ell}(B_0)$ and $\Gamma(A_1)\supset\Gamma(A_2)$.
Recall that $D_A=\bigsqcup\xs C_A$, $\xs C_A=\xs B^\circ_\ell(A)$ and
$D'_A=\bigsqcup\xs C'_A$, $\xs C'_A\subset\xs B^\circ_\ell(\Gamma(A))$. We
have
\begin{gather*}
  r^{(k_1+1)\ell}|E|\le\dist(D_{A_1},D_{A_2})\le|A_1|
  \le\varpi r^{k_1\ell}|E|, \\
  r^{(k_1+1)\ell+l}|E|\le\dist(D'_{A_1},D'_{A_2})\le|\Gamma(A_1)|
  \le\varpi r^{k_1\ell+l}|E|.
\end{gather*}

\textit{Case~2}. $A_1\cap A_2=\emptyset$. Then there are $k\ge0$ and
$B_3\in\xs A_k$ and $B_1,B_2\in\xs A_{k+1}$ such that $A_1\cup A_2\subset
B_3$, $A_1\subset B_1$, $A_2\subset B_2$ and $B_1\ne B_2$. By the properties
of~$\Gamma$, we have $\Gamma(B_3)\in\xs B_{k\ell}(B_0)$,
$\Gamma(B_1),\Gamma(B_2)\in\xs B_{(k+1)\ell}(B_0)$,
$\Gamma(A_1)\cup\Gamma(A_2)\subset\Gamma(B_3)$,
$\Gamma(A_1)\subset\Gamma(B_1)$, $\Gamma(A_2)\subset\Gamma(B_2)$ and
$\Gamma(B_1)\ne\Gamma(B_2)$. And so
\begin{gather*}
  r^{(k+1)\ell}|E|\le\dist(B_1,B_2)\le\dist(D_{A_1},D_{A_2})\le|B_3|
  \le\varpi r^{k\ell}|E|, \\
  r^{(k+1)\ell+l}|E|\le\dist(\Gamma(B_1),\Gamma(B_2))
  \le\dist(D'_{A_1},D'_{A_2})\le|\Gamma(B_3)|\le\varpi r^{k\ell+l}|E|.
\end{gather*}
For the diameter of~$F=E$ and $F'=B_0$, recall that $B_0$ is a level-$l$
interior block. By Remark~\ref{r:diameter}, we have
\[|F|=|E|\quad\text{and}\quad\varpi^{-1}r^l|E|\le|F'|\le\varpi r^l|E|.\]
By the definition of $\tilde f$, for distinct $D_1,D_2\in\xs D$,
\[L_1^{-1}\dist(D_1,D_2)/|F|\le\dist(\tilde f(D_1),\tilde f(D_2))/|F'|
\le L_1\dist(D_1,D_2)/|F|,\]
where $L_1=\varpi^2r^{-\ell}$.

For Condition~(ii) of Definition~\ref{d:smdi}, we use Lemma~\ref{l:blockslip}
to obtain~$f_D$ for each~$D$. Observe that, for $A\in\xs A_k$ with $k\ge0$,
$\xs C_A$ has $(A;k\ell,\ell)$-structure and $\xs C'_A$ has
$(\Gamma(A);k\ell+l,\ell)$-structure. Recall that $D_A=\bigsqcup\xs C_A$ and
$D'_A=\bigsqcup\xs C'_A$. By Lemma~\ref{l:blockslip}, there are $L_2>1$ and
bi-Lipschitz mappings~$f_D$ of~$D$ onto~$\tilde f(D)$ for each $D\in\xs D$,
such that
\[L_2^{-1}|x-y|\le r^{-l}|f_D(x)-f_D(y)|\le L_2|x-y|
\quad\text{for $x,y\in D$}.\]
Here $L_2$ depends on~$\ell$. Note that
\begin{equation}\label{eq:F/F'2}
  \varpi^{-1}r^{-l}\le|F|/|F'|\le\varpi r^{-l}.
\end{equation}
This means that $f_D$ satisfy the Condition~(ii) of Definition~\ref{d:smdi}
for $L_3=L_2\varpi$. Therefore, $f_D$ and $\tilde L=\max(L_1,L_2)$ satisfy
the conditions of Definition~\ref{d:smdi}. It follows that $F$ and $F'$ have
the same dense $\iota$-island structure with $\iota=\varpi r^{-\ell}$.

Finally, by Lemma~\ref{l:djt}, we have $E=F\simeq F'=B_0$.

\section{The proof of Theorem~\ref{t:TOCElip}}\label{sec:proof}

This section is devoted to the proof of Theorem~\ref{t:TOCElip}. To this end,
we consider two IFS $\ms,\mt\in\TOEC$. Throughout this section, we adopt the
following notational conventions: $r_\ms$, $I_\ms$, $p$, $\mu$, $\xc
P^\circ$, $\xs A^\circ$, $\xs A^\circ_k$ will denote the ratio root, the
ideal, the measure root, the natural measure, the set of measure polynomials
of interior blocks, the family of interior blocks and the family of level-$k$
interior blocks of~$\ms$, respectively; while $r_\mt$, $I_\mt$, $q$, $\nu$,
$\xc Q^\circ$, $\xs B^\circ$, $\xs B^\circ_k$ will denote the ratio root, the
ideal, the measure root, the natural measure, the set of measure polynomials
of interior blocks, the family of interior blocks and the family of level-$k$
interior blocks of~$\mt$, respectively. For $A\in\xs A^\circ_l$ and $B\in\xs
B^\circ_l$, we use $P_A$ and $Q_B$ to denote the corresponding measure
polynomials. We also write
\[\xs A_k(A)=\bigl\{A'\in\xs A_{l+k}\colon A'\subset A\bigr\}
\quad\text{and}\quad\xs B_k(B)=\bigl\{B'\in\xs B_{l+k}\colon
B'\subset B\bigr\}.\]

By Proposition~\ref{p:BlipE} and Lemma~\ref{l:blockslip},
Theorem~\ref{t:TOCElip} can be obtained from the following proposition.

\begin{prop}\label{p:BlipB}
  Let $A_0\in\xs A^\circ$ and $B_0\in\xs B^\circ$ be two interior blocks of
  $\ms$ and $\mt$, respectively. Then $A_0\simeq B_0$ if and only if
  \begin{enumerate}[\upshape(i)]
    \item $\hdim E_\ms=\hdim E_\mt$;
    \item $\log r_\ms/\log r_\mt\in\Q$;
    \item $I_\ms=a I_\mt$ for some $a\in\R$.
  \end{enumerate}
\end{prop}

\subsection{Necessity}

This subsection is devoted to the proof of necessary part of
Proposition~\ref{p:BlipB}. For this, we always assume that $A_0\in\xs
A^\circ$, $B_0\in\xs B^\circ$ and $A_0\simeq B_0$. Fix a bi-Lipschitz mapping
$f\colon A_0\to B_0$. The idea used in the proof of necessary part is similar
to that given in~\cite{CooPi88,FalMa92}. In particular, the following two
lemmas are essentially the same as the corresponding results
in~\cite{CooPi88,FalMa92}, see also Lemma~\ref{l:ML} and~\ref{l:fE}.

\begin{lem}[measure linear, \cite{CooPi88}]\label{l:ml}
  There exists an interior block~$A_f\subset A_0$ such that the restriction
  $f|_{A_f}\colon A_f\to f(A_f)$ is measure linear, i.e., for any Borel
  subset $F\subset A_f$ with $\mu(F)>0$,
  \[\frac{\mu(F)}{\nu(f(F))}=\frac{\mu(A_f)}{\nu(f(A_f))}.\]
\end{lem}

\begin{lem}[\cite{FalMa92}]\label{l:image}
  There exists an integer~$K_f$ such that for each interior block $A\subset
  A_0$, there are interior block $B_A\subset B_0$ and $\xs B_A\subset\xs
  B_{K_f}(B_A)$ satisfying $f(A)=\bigsqcup\xs B_A\subset B_A$.
\end{lem}

We omit the proofs of the above two lemmas, since the arguments are similar
to those used in~\cite{CooPi88,FalMa92}. We only remark that the proof of
Lemma~\ref{l:ml} requires the finiteness of measure polynomials
(Proposition~\ref{p:finiteCP}).

Now we turn to the proof of the necessity. Condition~(i) is obviously
necessary since $\hdim A_0=\hdim E_\ms$ and $\hdim B_0=\hdim E_\mt$.

For condition~(ii), let $\xs B_A$ be as in Lemma~\ref{l:image} for all
interior blocks $A\subset A_0$. By Proposition~\ref{p:finiteCP}, there are
only finitely many measure polynomials. It follows that the set
\[\biggl\{\sum_{B\in\xs B_A}Q_B(q)\colon
\text{$A\subset A_0$ is an interior block}\biggr\}\]
is finite, where $Q_B$ is the measure polynomial of~$B$. Together with
Lemma~\ref{l:CB(k)}, we know that there are two interior blocks $A_1,A_2$
with the same measure polynomial $P$ such that $A_2\subsetneq A_1\subsetneq
A_f$ and
\[\sum_{B\in\xs B_{A_1}}Q_B(q)=\sum_{B\in\xs B_{A_2}}Q_B(q),\]
where $A_f$ is as in Lemma~\ref{l:ml} and $\xs B_{A_1},\xs B_{A_2}$ are as in
Lemma~\ref{l:image}. Suppose that $A_1$ is of level~$k_1$, $A_2$ of
level~$k_2$ and $\xs B_{A_1}\subset\xs B^\circ_{l_1}$, $\xs B_{A_2}\subset\xs
B^\circ_{l_2}$, then by Remark~\ref{r:muB} and Lemma~\ref{l:ml},
\[\frac{p^{k_1}P(p)}{q^{l_1}\sum_{B\in\xs B_{A_1}}Q_B(q)}
=\frac{\mu(A_1)}{\nu(f(A_1))}=\frac{\mu(A_2)}{\nu(f(A_2))}
=\frac{p^{k_2}P(p)}{q^{l_2}\sum_{B\in\xs B_{A_2}}Q_B(q)}.
\]
This reduces to $p^{k_2-k_1}=q^{l_2-l_1}$. We have $\log r_\ms/\log
r_\mt\in\Q$ since $k_1\ne k_2$, $l_1\ne  l_2$ and $p=r_\ms^s$, $q=r_\mt^s$,
where $s=\hdim E_\ms=\hdim E_\mt$.

For condition~(iii), let $A_f$ be as in Lemma~\ref{l:ml}, set
\[\frac{\mu(A_f)}{\nu(f(A_f))}=a,\]
we need to show that $I_\ms=a I_\mt$. By Lemma~\ref{l:idealP} and symmetry,
it suffices to prove that $p^lP(p)\in a I_\mt$ for all $l\ge0$ and all
$P\in\xc P^\circ$. By condition~(ii), we can assume that $p^m=q^n$ for two
positive integers $m$ and~$n$. It follows from Lemma~\ref{l:CB(k)} that there
are positive integer $k$ and interior block $A\in\xs A^\circ_{km+l}$ such
that $A\subset A_f$ and $P_A=P$. Thus
\[p^lP(p)=p^{-km}\mu(A)=a\cdot q^{-kn}\nu(f(A))\in a I_\mt\]
since $q^{-1}\in\Z[q]$ and $f(A)$ is an interior separated set.

\subsection{Sufficiency}

In this subsection, we will prove the sufficient part.
\begin{lem}\label{l:Is>0}
  Let $\alpha\in I_\ms$ be a positive number, then there exist integer $l$
  and $\xs C\subset\xs B^\circ_k$ for some $k\ge0$ such that
  \[\alpha=p^l\sum_{B\in\xs C}\mu(B).\]
\end{lem}
\begin{proof}
  By Lemma~\ref{l:idealP}, there exist integers $k_1$ and $\gamma_P\ge0$ such
  that
  \[\alpha=p^{k_1}\sum_{P\in\xc P^\circ}\gamma_PP(p).\]
  By Lemma~\ref{l:CP(k)}, there exists an integer $k_2>k_1$ such that
  \[\zeta_P(k_2)\ge\gamma_P\quad\text{for all $P\in\xc P^\circ$}.\]
  So there exists $\xs C\in\xs B^\circ_{k_2}$ such that
  \[\sum_{B\in\xs C}\mu(B)=p^{k_2}\sum_{P\in\xc P^\circ}
  \gamma_PP(p)=p^{k_2-k_1}\alpha.\]
  The proof is completed by taking $l=k_1-k_2$.
\end{proof}

It follows from the conditions and Lemma~\ref{l:Is>0} that there exist $\xs
C\subset\xs A^\circ_l$ and $\xs C'\in\xs B^\circ_{l'}$ for some positive
integer $l,l'$ such that $p^l=q^{l'}$ and
\begin{equation}\label{eq:meq}
  \sum_{A\in\xs C}\mu(A)=a\sum_{B\in\xs C'}\nu(B).
\end{equation}
We will show that $\bigsqcup\xs C\simeq\bigsqcup\xs C'$, this conclusion
together with Lemma~\ref{l:blockslip} implies $A_0\simeq B_0$. The proof of
$\bigsqcup\xs C\simeq\bigsqcup\xs C'$ is similar to the proof of
Lemma~\ref{l:4blockslip}.

Let $F=\bigsqcup\xs C$ and $F'=\bigsqcup\xs C'$. We shall show that $F$ and
$F'$ have the same cylinder structure, then this lemma follows from
Lemma~\ref{l:cylin}. For this, we make use of Lemma~\ref{l:suitde} to define
a cylinder mapping satisfying~\eqref{eq:ML}. In Lemma~\ref{l:suitde}, take
the IFS to be~$\ms$ and the constants $\{\alpha_i\}$ to be $\{a Q(q)\colon
Q\in\xc Q^\circ\}$, suppose that the corresponding integer is~$K$. Use
Lemma~\ref{l:suitde} again by taking the IFS to be~$\mt$ and the constants
$\{\alpha_i\}$ to be $\{a^{-1}P(p)\colon P\in\xc P^\circ\}$, suppose that the
corresponding integer is~$K'$. By Remark~\ref{r:suit} and $\log r_\ms/\log
r_\mt\in\Q$, we can further require that $p^K=q^{K'}$.

The cylinder families $\xs C_k$ and $\xs C'_k$ are defined as follows. Define
$\xs C_1=\xs C$ and
  \[\begin{cases}
    \xs C_k=\xs A_{(k-1)K}(\xs C)&\text{for $k$ is odd};\\
    \xs C'_k=\xs B_{(k-1)K'}(\xs C')&\text{for $k$ is even}.
  \end{cases}\]
\begin{center}
  \begin{tabular}{*8c}
    \hline
    & $k=1$ & $k=2$ & $k=3$ & $k=4$ & $k=5$ & $k=6$ & \dots
    \rule[-1ex]{0pt}{3.5ex} \\
    \hline
    $\xs C_k$:  & $\xs C$ &   & $\xs A_{2K}(\xs C)$ &
      & $\xs A_{4K}(\xs C)$ &   & \dots \rule[-1ex]{0pt}{3.5ex} \\
    $\xs C'_k$: &   & $\xs B_{K'}(\xs C')$ &   & $\xs B_{3K'}(\xs C')$
      &   & $\xs B_{5K'}(\xs C')$ & \dots \rule[-1ex]{0pt}{3.5ex} \\
    \hline
  \end{tabular}
\end{center}
It remains to define $\xs C_k$ for $k$ is even, $\xs C'_k$ for $k$ is odd and
the cylinder mapping~$\tilde f$.

We begin with the definition of~$\xs C'_1$ by making use of
Lemma~\ref{l:suitde}. By~\eqref{eq:meq} and $p^l=q^{l'}$,
\[a^{-1}\sum_{A\in\xs C_1}P_A(p)=a^{-1}\sum_{A\in\xs C}P_A(p)
=\sum_{B\in\xs C'}Q_B(q).\]
Since $P_A\in\xc P^\circ$ for all $A\in\xs C_1=\xs C$, we know that the
sequence $\{a^{-1}P_A(p)\colon A\in\xs C_1\}$ are $(\xs
C';\alpha_1,\dots,\alpha_\eta)$-suitable by Definition~\ref{d:suit}, where
\[\{\alpha_1\dots,\alpha_\eta\}=\{a^{-1}P(p)\colon P\in\xc P^\circ\}.\]
Since $\xs C'\subset\xs B_{l'}$, by Lemma~\ref{l:suitde}, for positive
numbers $\{q^{l'}a^{-1}P_A(p)\colon A\in\xs C_1\}$, there is a suitable
decomposition
\[\xs B_{K'}(\xs C')=\bigcup_{A\in\xs C_1}\xs C'_A\]
such that
\[\sum_{B\in\xs C'_A}\nu(B)=q^{l'}a^{-1}P_A(p)\quad
\text{for all $A\in\xs C_1=\xs C$}.\]
Define
\[\xs C'_1=\left\{\bigsqcup\xs C'_A\colon A\in\xs C_1\right\},\]
and $\tilde f\colon\xs C_1\to\xs C'_1$ by $\tilde f(A)=\bigsqcup\xs C'_A$. It
is easy to see that $\xs C'_1\prec\xs C'_2=\xs B_{K'}(\xs C')$. We also have
that
\[\mu(A)=a\nu(\tilde f(A))\quad\text{for all $A\in\xs C_1$},\]
since $a\nu(\tilde f(A))=a\sum_{B\in\xs C'_A}\nu(B)=q^{l'}P_A(p)
=p^lP_A(p)=\mu(A)$.

Now suppose that the cylinder families $\xs C_1$, \dots, $\xs C_{k-1}$, $\xs
C'_1$, \dots, $\xs C'_{k-1}$ and the cylinder mapping $\tilde f$ have been
defined such that $\tilde f$ maps $\xs C_j$ onto $\xs C'_j$ for $1\le j\le
k-1$ and
\begin{equation}\label{eq:muCanu}
  \mu(C)=a\nu(\tilde f(C))\quad
  \text{for all $C\in\bigcup_{j=1}^{k-1}\xs C_j$}.
\end{equation}
We shall define $\xs C_k$, $\xs C'_k$ and $\tilde f\colon\xs C_k\to\xs C'_k$.
Suppose without loss of generality that $k$ is even, then $\xs C'_k=\xs
B_{(k-1)K'}(\xs C')$. We consider the suitable decomposition of $\xs
A_K(A_0)$ for each $A_0\in\xs C_{k-1}=\xs A_{(k-2)K}(\xs C)$. Recall that
$p^l=q^{l'}$ and $p^K=q^{K'}$, by~\eqref{eq:muCanu},
\[\sum_{\substack{B\subset\tilde f(A_0)\\B\in\xs C'_k}}aQ_B(q)
  =aq^{-l'-(k-1)K'}\nu(\tilde f(A_0))=p^{-l-(k-1)K}\mu(A_0)
  =\sum_{A\in\xs A_K(A_0)}P_A(p).\]
And so the sequence $\{aQ_B(q)\colon B\subset\tilde f(A_0),B\in\xs C'_k\}$
are $\{\xs A_K(A_0);\alpha_1,\dots,\alpha_\eta\}$-suitable by
Definition~\ref{d:suit}, where $\{\alpha_1\dots,\alpha_\eta\}=\{a Q(q)\colon
Q\in\xc Q^\circ\}$. By Lemma~\ref{l:suitde}, for positive numbers
\[\Bigl\{p^{l+(k-1)K}aQ_B(q)\colon
B\subset\tilde f(A_0),B\in\xs C'_k\Bigr\},\]
there is a suitable decomposition
\begin{equation}\label{eq:finer}
  \xs A_K(\xs A_K(A_0))=\xs A_{2K}(A_0)=
  \bigcup_{\substack{B\subset\tilde f(A_0)\\B\in\xs C'_k}}\xs C_B
\end{equation}
such that
\[\sum_{A\in\xs C_B}\mu(A)=p^{l+(k-1)K}aQ_B(q)\quad
\text{for all $B\in\xs C'_k$ and $B\subset\tilde f(A_0)$}.\]
Indeed, we obtain $\xs C_B$ for all $B\in\xs C'_k$ by above argument since
for every $B\in\xs C'_k$, there is a unique $A_0\in\xs C_{k-1}$ such that
$B\subset\tilde f(A_0)$. Then we define
\[\xs C_k=\left\{\bigsqcup\xs C_B\colon B\in\xs C'_k\right\}\]
and $\tilde f\colon\xs C_k\to\xs C'_k$ by $\tilde f\left(\bigsqcup\xs
C_B\right)=B$. By~\eqref{eq:finer}, we know that
\[\xs A_{(k-2)K}(\xs C)=\xs C_{k-1}\prec\xs C_k\prec\xs C_{k+1}
=\xs A_{kK}(\xs C).\]
We also have $\mu(C)=a\nu(\tilde f(C))$ for all $C\in\xs C_k$. If $k$ is odd,
we can define $\xs C_k$, $\xs C'_k$ and $\tilde f\colon\xs C_k\to\xs C'_k$ by
a similar argument. Thus, by induction on~$k$, we finally obtain all the
cylinder families $\xs C_k$, $\xs C'_k$ and the cylinder mapping~$\tilde f$.

To prove $F=\bigsqcup\xs C$ and $F'=\bigsqcup\xs C'$ have the same cylinder
structure, it remains to compute the constants $\varrho$ and $\iota$. Since
$\xs C$ contains at least one level-$l$ interior block, by
Remark~\ref{r:diameter}, we have
\[\varpi_\ms^{-1}r_\ms^l|E_\ms|\le|F|\le|E_\ms|.\]
Let $C,C_1,C_2\in\xs C_k$, where $C_1$ and $C_2$ are distinct. If $k$ is odd,
then $\xs C_k=\xs A_{(k-1)K}(\xs C)\subset\xs A_{l+(k-1)K}^\circ$, by
Remark~\ref{r:diameter} and the definition of blocks
(Definition~\ref{d:block}), we have
\begin{align*}
  \varpi_\ms^{-1}r_\ms^{l+(k-1)K}|E_\ms|&\le|C|
  \le\varpi_\ms r_\ms^{l+(k-1)K}|E_\ms|;\\
  r_\ms^{l+(k-1)K}|E_\ms|&\le\dist(C_1,C_2).
\end{align*}
If $k$ is even, then by the definition of $\xs C_k$, we know that $\xs
A_{(k-2)K}(\xs C)=\xs C_{k-1}\prec\xs C_k\prec\xs C_{k+1}=\xs A_{kK}(\xs C)$.
And so
\begin{align*}
  \varpi_\ms^{-1}r_\ms^{l+kK}|E_\ms|&\le|C|
  \le\varpi_\ms r_\ms^{l+(k-2)K}|E_\ms|;\\
  r_\ms^{l+kK}|E_\ms|&\le\dist(C_1,C_2).
\end{align*}
As a summary, $F$ has the $(\varrho_1,\iota_1)$-cylinder structure for
$\varrho_1=r_\ms^K$ and $\iota_1=\varpi_\ms^2r_\ms^{-l-2K}$. A similar
argument shows that $F'$ also has the $(\varrho_2,\iota_2)$-cylinder
structure for $\varrho_2=r_\mt^{K'}$ and
$\iota_2=\varpi_\mt^2r_\ms^{-l'-2K'}$. Recall that $r_\ms^K=r_\mt^{K'}$. Then
by the cylinder mapping~$\tilde f$, we know that $F$ and $F'$ have the same
$(\varrho,\iota)$-cylinder structure for $\varrho=\varrho_1=\varrho_2$ and
$\iota=\max(\iota_1,\iota_2)$.

Finally, by Lemma~\ref{l:cylin}, we have $\bigsqcup\xs C=F\simeq
F'=\bigsqcup\xs C'$.

\section{The Non-Commensurable Case}\label{sec:NC}

\subsection{Proof of Theorem~\ref{t:Z+}}

Let $\ms=\{S_1,S_2\dots,S_N\}$ of ratios $r_1,r_2\dots,r_N$ and
$\mt=\{T_1,T_2\dots,T_M\}$ of ratios $t_1,t_2\dots,t_M$ be two IFSs
satisfying the SSC. For convenience, write $E=E_\ms$ and $F=E_\mt$. We also
use the following notations:
\begin{alignat*}{2}
  E_{\bm i}&=S_{i_1}\circ\dots\circ S_{i_n}(E),\qquad &
  r_{\bm i}&=r_{i_1}r_{i_2}\cdots r_{i_n}, \\
  F_{\bm j}&=T_{j_1}\circ\dots\circ T_{j_m}(F),\qquad &
  t_{\bm j}&=t_{j_1}t_{j_2}\cdots t_{j_m},
\end{alignat*}
where $\bm i=i_1i_2\dots i_n\in\{1,\dots,N\}^n$ and $\bm j=j_1j_2\dots
j_m\in\{1,\dots, M\}^m$.

Suppose that $E\simeq F$ and $\hdim E=\hdim F=s$. Let $f$ be a bi-Lipschitz
mapping from $E$ onto $F$. We need two known lemmas.

\begin{lem}[measure linear, \cite{CooPi88}]\label{l:ML}
  There is an $E_{\xb i}$ such that $f|_{E_{\xb i}}$ is measure linear, i.e.,
  for all Borel subsets $A\subset E_{\xb i}$ with $\xc H^s(A)>0$, we have
  \[\frac{\xc H^s(A)}{\xc H^s(f(A))}=\frac{\xc H^s(E_{\xb i})}
  {\xc H^s(f(E_{\xb i}))}.\]
\end{lem}

\begin{lem}[\cite{FalMa92}]\label{l:fE}
  There exists a positive integer~$K$ dependent on~$f$ such that for each
  word $\bm i$ of finite length, there exist a subset
  $\Lambda\subset\{1,\dots,M\}^K$ and a word $\bm j$ of finite length such
  that
  \[f(E_{\bm i})=\bigcup_{\bm j^*\in\Lambda}F_{\bm{jj}^*}.\]
\end{lem}

\begin{proof}[Proof of Theorem~\ref{t:Z+}]
  By symmetry, it suffices to prove that $t_j^s\in\Z^+[r_1^s,\dots,r_N^s]$
  for each~$j$. Let $E_{\xb i}$ be as in Lemma~\ref{l:ML}. By
  Lemma~\ref{l:fE}, we have
  \[f(E_{\xb i})=\bigcup_{\bm j^*\in\Lambda}F_{\xb j\bm j^*}
  \quad\text{for some $\Lambda\subset\{1,\dots,M\}^K$}.\]
  For each $j\in\{1,\dots,M\}$, there is a set $\Lambda_j$ consisting of finite
  many words of finite length such that
  \[f^{-1}\biggl(\bigcup_{\bm j^*\in\Lambda}F_{\xb j\bm j^*j}\biggr)=
  \bigcup_{\bm i^*\in\Lambda_j}E_{\xb i\bm i^*}.\]
  Applying Lemma~\ref{l:ML} with $A=\bigcup_{\bm i^*\in\Lambda_j}E_{\xb i\bm
  i^*}$, we have
  \[\frac{\xc H^s(E_{\xb i})}{\xc H^s(f(E_{\xb i}))}=
    \frac{\xc H^s(A)}{\xc H^s(f(A))}=\frac{\xc H^s
    \bigl(\bigcup_{\bm i^*\in\Lambda_j}E_{\xb i\bm i^*}\bigr)}
    {\xc H^s\bigl(\bigcup_{\bm j^*\in\Lambda}F_{\xb j\bm j^*j}\bigr)}
    =\frac{\xc H^s(E_{\xb i})\cdot\sum_{\bm i^*\in\Lambda_j}
    r_{\bm i^*}^s}{\xc H^s(f(E_{\xb i}))\cdot t_j^s}.
  \]
  This means $t_j^s=\sum_{\bm i^*\in\Lambda_j}r_{\bm i^*}^s$, and so
  $t_j^s\in\Z^+[r_1^s,\dots,r_N^s]$.
\end{proof}

\subsection{Proof of Theorem~\ref{t:sublip}}

With notations as in the proof of Theorem~\ref{t:Z+}. We need a lemma
obtained by Rao, Ruan and Wang~\cite{RaRuW12}, which is a corollary of
Lemma~\ref{l:ML} and~\ref{l:fE}. Fix a bi-Lipschitz mapping~$f$ of~$E$
onto~$F$. Let $E_{\xb i}$ be as in Lemma~\ref{l:ML}. According to
Lemma~\ref{l:fE}, for each word~$\bm i$ of finite length, there are a subset
$\Lambda\subset\{1,\dots,M\}^K$ and a word~$\bm j$ of finite length such that
\begin{equation}\label{eq:fEii}
  f(E_{\xb i\bm i})=\bigcup_{\bm j^*\in\Lambda}F_{\bm j\bm j^*}.
\end{equation}
\begin{lem}[\cite{RaRuW12}]\label{l:Mfnt}
  The set
  $\bigl\{\xc H^s(E_{\xb i\bm i})/\xc H^s(F_{\bm j})\colon
  \text{$\bm i$ and $\bm j$ satisfy~\eqref{eq:fEii}}\bigr\}$
  is finite.
\end{lem}

Let $G\subset(0,1)$ be a multiplicative semigroup. For $\bm i=i_1i_2\dots
i_n\in\{1,\dots,N\}^n$ and $\bm j=j_1j_2\dots j_m\in\{1,\dots,M\}^m$, define
\[\card_G\bm i=\card\{k\colon S_{i_k}\notin\ms^G\}\quad\text{and}\quad
\card_G\bm j=\card\{k\colon T_{j_k}\notin\mt^G\}.\]
As a corollary of Lemma~\ref{l:Mfnt}, we have
\begin{lem}\label{l:cardG}
  $\sup\bigl\{\card_G\bm j\colon\text{$\bm i$ and $\bm j$
  satisfy~\eqref{eq:fEii} and $\card_G\bm i=0$}\bigr\}<\infty$.
\end{lem}
\begin{proof}
  Write $\card_l\bm j=\card\{k\colon j_k=l\}$ for $1\le l\le M$, $\card_l\bm
  i=\card\{k\colon i_k=l\}$ for $1\le l\le N$. If this lemma is not true, we
  can find a sequence $(\bm i_k,\bm j_k)_{k\ge1}$ such that for all $k\ge1$,
  \begin{itemize}
    \item $\bm i_k$ and $\bm j_k$ satisfy~\eqref{eq:fEii} and $\card_G\bm
        i_k=0$;
    \item $\card_G\bm j_k<\card_G\bm j_{k+1}$;
  \end{itemize}
  If $\sup_k\card_1\bm i_k=\infty$, by choosing a subsequence, we can assume
  that $\card_1\bm i_k<\card_1\bm i_{k+1}$; otherwise, $\sup_k\card_1\bm
  i_k<\infty$, by choosing a subsequence, we can assume that $\card_1\bm i_k$
  is equal to a constant for all~$k$. In both cases, we can require that
  $\card_1\bm i_k\le\card_1\bm i_{k+1}$. Repeating the same argument, we can
  further require that for all $k\ge1$,
  \begin{itemize}
    \item $\card_l\bm i_k\le\card_l\bm i_{k+1}$ for $1\le l\le N$.
    \item $\card_l\bm j_k\le\card_l\bm j_{k+1}$ for $1\le l\le M$.
  \end{itemize}

  We shall show that
  \[\frac{\xc H^s(E_{\xb i\bm i_a})}{\xc H^s(F_{\bm j_a})}\ne
  \frac{\xc H^s(E_{\xb i\bm i_b})}{\xc H^s(F_{\bm j_b})}\quad \text{whenever
  $a\ne b$}.\] This contradicts Lemma~\ref{l:Mfnt}, and so the lemma
  follows. To verify the inequality, suppose $a<b$, we have
  \[\left(\frac{\xc H^s(E_{\xb i\bm i_a})}{\xc H^s(F_{\bm j_a})}\biggm/
  \frac{\xc H^s(E_{\xb i\bm i_b})}{\xc H^s(F_{\bm j_b})}\right)^{1/s}=
  \frac{t_{\bm j_b}/t_{\bm j_a}}{r_{\bm i_b}/r_{\bm i_a}}=
  \frac{t_1^{\beta_1}t_2^{\beta_2}\cdots t_M^{\beta_M}}
  {r_1^{\alpha_1}r_2^{\alpha_2}\cdots r_N^{\alpha_N}}:=\frac\varphi\phi,\]
  where $\alpha_l=\card_l\bm i_b-\card_l\bm i_a\ge0$ ($1\le l\le N$) and
  $\beta_l=\card_l\bm j_b-\card_l\bm j_a\ge0$ ($1\le l\le M$). Suppose that
  $T_1,\dots,T_\ell\notin\mt^G$ and $T_{\ell+1},\dots,T_M\in\mt^G$, then
  \[\card_G\bm j=\card_1\bm j+\card_2\bm j+\dots+\card_\ell\bm j.\]
  Since
  \[\beta_1+\dots+\beta_\ell=\sum_{l=1}^\ell(\card_l\bm j_b-\card_l\bm j_a)
  =\card_G\bm j_b-\card_G\bm j_a>0,\]
  we may assume that $\beta_1>0$. Then $\varphi/t_1\in\sgp\mt$. It follows
  from $\card_G\bm i_a=\card_G\bm i_b=0$ that there exists a $g\in\sgp\ms$
  such that $\phi g\in G$. Since $\sgp\ms\sim\sgp\mt$, there exists a
  positive integer~$u$ such that $g^u\in\sgp\mt$. Therefore,
  \[\varphi^ug^u=t_1\cdot\bigl(t_1^{u-1}(\varphi/t_1)^ug^u\bigr)\in
  t_1\cdot\sgp\mt.\]
  Notice that $(t_1\cdot\sgp\mt)\cap G=\emptyset$ since $T_1\notin\mt^G$.
  Therefore, $\varphi^ug^u\notin G$. Together with $\phi^ug^u\in G$, we have
  $\phi\ne\varphi$. The desired inequality follows.
\end{proof}

To prove Theorem~\ref{t:sublip}, we also need the following theorem obtained
independently by Llorente and Mattila~\cite{LloMa10} and Deng and
Wen~et.~al.~\cite{DeWXX11}.
\begin{thm*}[\cite{DeWXX11,LloMa10}]
  Let $E$ and $F$ be two self-similar sets satisfying the SSC. Then $E\simeq
  F$ if and only if there exist bi-Lipschitz mappings $f_1$ and $f_2$ such
  that $f_1\colon E\to f_1(E)\subset F$ and $f_2\colon F\to f_2(F)\subset
  E$.
\end{thm*}

\begin{proof}[Proof of Theorem~\ref{t:sublip}]
First note that we have $\ms^G=\ms$ and $\mt^G=\mt$ when $G=(0,1)$. So we
only need to show that $\ms\simeq\mt$ implies $\ms^G\simeq\mt^G$. Now fix a
multiplicative sub-semigroup $G\subset(0,1)$. By the above theorem and
symmetry, it remains to find a bi-Lipschitz mapping $f_1$ such that
$f_1\colon E_{\ms^G}\to f_1(E_{\ms^G})\subset E_{\mt^G}$.

According to Lemma~\ref{l:cardG}, we can find $\bm i_0,\bm j_0$
satisfying~\eqref{eq:fEii} and $\card_G\bm i_0=0$ such that
\[\card_G\bm j_0=\sup\bigl\{\card_G\bm j\colon\text{$\bm i$ and $\bm j$
  satisfy~\eqref{eq:fEii} and $\card_G\bm i=0$}\bigr\}<\infty.\]
We shall show that $f\circ S_{\xb i\bm i_0}(E_{\ms^G})\subset T_{\bm
j_0}(E_{\mt^G})$, where $f$ is the fixed bi-Lipschitz mapping of~$E$
onto~$F$. Then taking $f_1=T_{\bm j_0}^{-1}\circ f\circ S_{\xb i\bm i_0}$, we
have $f_1(E_{\ms^G})\subset E_{\mt^G}$ and the proof is complete.

Let $x\in E_{\ms^G}$ and $\bm i_k$ be the unique word of length~$k$
satisfying $x\in S_{\bm i_k}(E_{\ms^G})$ for each $k\ge1$, then $\card_G\bm
i_k=0$. For each $k\ge1$, let $\bm j_k$ satisfies
\[f(E_{\xb i\bm i_0\bm i_k})=\bigcup_{\bm j^*\in\Lambda_k}F_{\bm j_k\bm j^*},
\quad\text{where $\Lambda_k\subset\{1,\dots,M\}^K$.}\]
Note that
\[f(E_{\xb i\bm i_0})=\bigcup_{\bm j^*\in\Lambda}F_{\bm j_0\bm j^*}
\quad\text{for some $\Lambda\subset\{1,\dots,M\}^K$}\]
since $\bm i_0$ and $\bm j_0$ satisfy~\eqref{eq:fEii}. So we can write $\bm
j_k=\bm j_0\bm j'_k$ for each $k\ge1$. Since $\card_G\bm i_0\bm i_k=0$ and
$\card_G\bm j_0$ is maximal, we have $\card_G\bm j'_k=0$ for each $k\ge1$.
This means $f\circ S_{\xb i\bm i_0}(x)\in T_{\bm j_0}(E_{\mt^G})$, and so
$f\circ S_{\xb i\bm i_0}(E_{\ms^G})\subset T_{\bm j_0}(E_{\mt^G})$.
\end{proof}

\subsection*{Acknowledgement}

The authors are grateful to Ka-Sing Lau for his active interest about our
results and his enthusiasm to help us improve the manuscript. We also thank
Zhiying Wen and Zhixiong Wen for their valuable comments and suggestions. The
second author is indebted to Dejun Feng for his invitation to the Chinese
University of Hong Kong. Part of this work was done during the visit. The
second author is also glad to express his deep gratitude to Min Wu and Jihua
Ma for their constant encouragement.


\begin{thebibliography}{10}

\bibitem{Baker68} A.~Baker.
\newblock Linear forms in the logarithms of algebraic numbers. {IV}.
\newblock {\em Mathematika}, 15:204--216, 1968.

\bibitem{BanGr92} Christoph Bandt and Siegfried Graf.
\newblock Self-similar sets. {VII}. {A} characterization of self-similar
  fractals with positive {H}ausdorff measure.
\newblock {\em Proc. Amer. Math. Soc.}, 114(4):995--1001, 1992.

\bibitem{BaHuR06} Christoph Bandt, Nguyen~Viet Hung, and Hui Rao.
\newblock On the open set condition for self-similar fractals.
\newblock {\em Proc. Amer. Math. Soc.}, 134(5):1369--1374, 2006.

\bibitem{BanRa07} Christoph Bandt and Hui Rao.
\newblock Topology and separation of self-similar fractals in the plane.
\newblock {\em Nonlinearity}, 20(6):1463--1474, 2007.

\bibitem{Bonk06} Mario Bonk.
\newblock Quasiconformal geometry of fractals.
\newblock In {\em International {C}ongress of {M}athematicians. {V}ol. {II}},
  pages 1349--1373. Eur. Math. Soc., Z\"urich, 2006.

\bibitem{CooPi88} Daryl Cooper and Thea Pignataro.
\newblock On the shape of {C}antor sets.
\newblock {\em J. Differential Geom.}, 28(2):203--221, 1988.

\bibitem{CurRe62} Charles~W. Curtis and Irving Reiner.
\newblock {\em Representation theory of finite groups and associative
  algebras}.
\newblock Pure and Applied Mathematics, Vol. XI. Interscience Publishers, a
  division of John Wiley \& Sons, New York-London, 1962.

\bibitem{DasEd11} M.~Das and G.~A. Edgar.
\newblock Finite type, open set conditions and weak separation conditions.
\newblock {\em Nonlinearity}, 24(9):2489--2503, 2011.

\bibitem{DavSe97} Guy David and Stephen Semmes.
\newblock {\em Fractured fractals and broken dreams}, volume~7 of {\em Oxford
  Lecture Series in Mathematics and its Applications}.
\newblock The Clarendon Press Oxford University Press, New York, 1997.
\newblock Self-similar geometry through metric and measure.

\bibitem{DenHe12} GuoTai Deng and XingGang He.
\newblock Lipschitz equivalence of fractal sets in {${\mathbb R}$}.
\newblock {\em Sci. China Math.}, 55(10):2095--2107, 2012.

\bibitem{DeWXX11} Juan Deng, Zhi-Ying Wen, Ying Xiong, and Li-Feng Xi.
\newblock Bilipschitz embedding of self-similar sets.
\newblock {\em J. Anal. Math.}, 114:63--97, 2011.

\bibitem{DenLa} Qi-Rong Deng and Ka-Sing Lau.
\newblock On the equivalence of homogeneous iterated function systems.
\newblock preprint.

\bibitem{FalMa89} K.~J. Falconer and D.~T. Marsh.
\newblock Classification of quasi-circles by {H}ausdorff dimension.
\newblock {\em Nonlinearity}, 2(3):489--493, 1989.

\bibitem{FalMa92} K.~J. Falconer and D.~T. Marsh.
\newblock On the {L}ipschitz equivalence of {C}antor sets.
\newblock {\em Mathematika}, 39(2):223--233, 1992.

\bibitem{Falco97} Kenneth Falconer.
\newblock {\em Techniques in fractal geometry}.
\newblock John Wiley \& Sons Ltd., Chichester, 1997.

\bibitem{Falco03} Kenneth Falconer.
\newblock {\em Fractal geometry}.
\newblock John Wiley \& Sons Inc., Hoboken, NJ, second edition, 2003.
\newblock Mathematical foundations and applications.

\bibitem{FarMo98} Benson Farb and Lee Mosher.
\newblock A rigidity theorem for the solvable {B}aumslag-{S}olitar groups.
\newblock {\em Invent. Math.}, 131(2):419--451, 1998.
\newblock With an appendix by Daryl Cooper.

\bibitem{Feder69} Herbert Federer.
\newblock {\em Geometric measure theory}.
\newblock Die Grundlehren der mathematischen Wissenschaften, Band 153.
  Springer-Verlag New York Inc., New York, 1969.

\bibitem{FenWa09} De-Jun Feng and Yang Wang.
\newblock On the structures of generating iterated function systems of {C}antor
  sets.
\newblock {\em Adv. Math.}, 222(6):1964--1981, 2009.

\bibitem{Goldf85} Dorian Goldfeld.
\newblock Gauss's class number problem for imaginary quadratic fields.
\newblock {\em Bull. Amer. Math. Soc. (N.S.)}, 13(1):23--37, 1985.

\bibitem{Gromo07} Misha Gromov.
\newblock {\em Metric structures for {R}iemannian and non-{R}iemannian spaces}.
\newblock Modern Birkh\"auser Classics. Birkh\"auser Boston Inc., Boston, MA,
  english edition, 2007.
\newblock Based on the 1981 French original, With appendices by M. Katz, P.
  Pansu and S. Semmes, Translated from the French by Sean Michael Bates.

\bibitem{Hochm12} M.~{Hochman}.
\newblock {On self-similar sets with overlaps and inverse theorems for
  entropy}.
\newblock {\em ArXiv e-prints}, December 2012.

\bibitem{Hutch81} John~E. Hutchinson.
\newblock Fractals and self-similarity.
\newblock {\em Indiana Univ. Math. J.}, 30(5):713--747, 1981.

\bibitem{Kenyo97} Richard Kenyon.
\newblock Projecting the one-dimensional {S}ierpinski gasket.
\newblock {\em Israel J. Math.}, 97:221--238, 1997.

\bibitem{Lang94} Serge Lang.
\newblock {\em Algebraic number theory}, volume 110 of {\em Graduate Texts in
  Mathematics}.
\newblock Springer-Verlag, New York, second edition, 1994.

\bibitem{LauNg99} Ka-Sing Lau and Sze-Man Ngai.
\newblock Multifractal measures and a weak separation condition.
\newblock {\em Adv. Math.}, 141(1):45--96, 1999.

\bibitem{LloMa10} Marta Llorente and Pertti Mattila.
\newblock Lipschitz equivalence of subsets of self-conformal sets.
\newblock {\em Nonlinearity}, 23:875--882, 2010.

\bibitem{LuoLa12} J.~J. {Luo} and K.-S. {Lau}.
\newblock {Lipschitz equivalence of self-similar sets and hyperbolic
  boundaries}.
\newblock {\em ArXiv e-prints}, May 2012.

\bibitem{MatSa09} Pertti Mattila and Pirjo Saaranen.
\newblock Ahlfors-{D}avid regular sets and bilipschitz maps.
\newblock {\em Ann. Acad. Sci. Fenn. Math.}, 34(2):487--502, 2009.

\bibitem{MauWi88} R.~Daniel Mauldin and S.~C. Williams.
\newblock Hausdorff dimension in graph directed constructions.
\newblock {\em Trans. Amer. Math. Soc.}, 309(2):811--829, 1988.

\bibitem{MolWi91} R.~A. Mollin and H.~C. Williams.
\newblock On a determination of real quadratic fields of class number one and
  related continued fraction period length less than {$25$}.
\newblock {\em Proc. Japan Acad. Ser. A Math. Sci.}, 67(1):20--25, 1991.

\bibitem{MolWi92} R.~A. Mollin and H.~C. Williams.
\newblock On real quadratic fields of class number two.
\newblock {\em Math. Comp.}, 59(200):625--632, 1992.

\bibitem{Moran46} P.~A.~P. Moran.
\newblock Additive functions of intervals and {H}ausdorff measure.
\newblock {\em Proc. Cambridge Philos. Soc.}, 42:15--23, 1946.

\bibitem{Neuki99} J{\"u}rgen Neukirch.
\newblock {\em Algebraic number theory}, volume 322 of {\em Grundlehren der
  Mathematischen Wissenschaften [Fundamental Principles of Mathematical
  Sciences]}.
\newblock Springer-Verlag, Berlin, 1999.
\newblock Translated from the 1992 German original and with a note by Norbert
  Schappacher, With a foreword by G. Harder.

\bibitem{RaRuW12} Hui Rao, Huo-Jun Ruan, and Yang Wang.
\newblock Lipschitz equivalence of {C}antor sets and algebraic properties of
  contraction ratios.
\newblock {\em Trans. Amer. Math. Soc.}, 364(3):1109--1126, 2012.

\bibitem{RaRuX06} Hui Rao, Huo-Jun Ruan, and Li-Feng Xi.
\newblock Lipschitz equivalence of self-similar sets.
\newblock {\em C. R. Math. Acad. Sci. Paris}, 342(3):191--196, 2006.

\bibitem{RaoWe98} Hui Rao and Zhi-Ying Wen.
\newblock A class of self-similar fractals with overlap structure.
\newblock {\em Adv. in Appl. Math.}, 20(1):50--72, 1998.

\bibitem{RuWaX12} H.-J. {Ruan}, Y.~{Wang}, and L.-F. {Xi}.
\newblock {Lipschitz equivalence of self-similar sets with touching
  structures}.
\newblock {\em ArXiv e-prints}, July 2012.

\bibitem{Schie94} Andreas Schief.
\newblock Separation properties for self-similar sets.
\newblock {\em Proc. Amer. Math. Soc.}, 122(1):111--115, 1994.

\bibitem{Stark67} H.~M. Stark.
\newblock A complete determination of the complex quadratic fields of
  class-number one.
\newblock {\em Michigan Math. J.}, 14:1--27, 1967.

\bibitem{Stark07} H.~M. Stark.
\newblock The {G}auss class-number problems.
\newblock In {\em Analytic number theory}, volume~7 of {\em Clay Math. Proc.},
  pages 247--256. Amer. Math. Soc., Providence, RI, 2007.

\bibitem{SVe02} Grzegorz {\'S}wi{\c{a}}tek and J.~J.~P. Veerman.
\newblock On a conjecture of {F}urstenberg.
\newblock {\em Israel J. Math.}, 130:145--155, 2002.

\bibitem{Watki04} Mark Watkins.
\newblock Class numbers of imaginary quadratic fields.
\newblock {\em Math. Comp.}, 73(246):907--938 (electronic), 2004.

\bibitem{WenXi03} Zhi-Ying Wen and Li-Feng Xi.
\newblock Relations among {W}hitney sets, self-similar arcs and quasi-arcs.
\newblock {\em Israel J. Math.}, 136:251--267, 2003.

\bibitem{Xi04} Li-Feng Xi.
\newblock Lipschitz equivalence of self-conformal sets.
\newblock {\em J. London Math. Soc. (2)}, 70(2):369--382, 2004.

\bibitem{Xi07} Li-Feng Xi.
\newblock Quasi-{L}ipschitz equivalence of fractals.
\newblock {\em Israel J. Math.}, 160:1--21, 2007.

\bibitem{Xi10} Li-Feng Xi.
\newblock Lipschitz equivalence of dust-like self-similar sets.
\newblock {\em Math. Z.}, 266(3):683--691, 2010.

\bibitem{XiRu07} Li-feng Xi and Huo-jun Ruan.
\newblock Lipschitz equivalence of generalized {$\{1,3,5\}$}-{$\{1,4,5\}$}
  self-similar sets.
\newblock {\em Sci. China Ser. A}, 50(11):1537--1551, 2007.

\bibitem{XiXi10} Li-Feng Xi and Ying Xiong.
\newblock Self-similar sets with initial cubic patterns.
\newblock {\em C. R. Math. Acad. Sci. Paris}, 348:15--20, 2010.

\bibitem{XiXi12} Li-Feng Xi and Ying Xiong.
\newblock Lipschitz equivalence of fractals generated by nested cubes.
\newblock {\em Math. Z.}, 271(3):1287--1308, 2012.

\bibitem{XioXi12} Ying Xiong and Li~Feng Xi.
\newblock Lipschitz equivalence of self-similar sets satisfying the open set
  condition.
\newblock {\em Chinese J. Contemp. Math.}, 33(1):1--16, 2012.

\bibitem{XioXi09} Ying Xiong and Lifeng Xi.
\newblock Lipschitz equivalence of graph-directed fractals.
\newblock {\em Studia Math.}, 194(2):197--205, 2009.

\bibitem{Zerner1996} Martin P.~W. Zerner.
\newblock Weak separation properties for self-similar sets.
\newblock {\em Proc. Amer. Math. Soc.}, 124(11):3529--3539, 1996.

\end{thebibliography}
\end{document}